\documentclass{amsart}

\usepackage{amssymb}
\usepackage{amsfonts}
\usepackage{amssymb}
\usepackage{amsmath}
\usepackage{amsthm}

\usepackage[normalem]{ulem}
\usepackage{soul}
\usepackage{caption}
\usepackage{enumerate}
\usepackage{tabularx}
\usepackage{centernot}
\usepackage{mathtools}
\usepackage{stmaryrd}
\usepackage{amsthm,amssymb}
\usepackage{etoolbox}
\usepackage{tikz}
\usepackage{amssymb}
\usepackage{tablefootnote}
\usetikzlibrary{matrix}
\usepackage{graphics,graphicx}
\usepackage[all]{xy}
\usepackage{url}
\usepackage{marginnote}
\definecolor{mygray}{gray}{0.85}
\usepackage[linecolor=black,backgroundcolor=mygray,colorinlistoftodos,prependcaption,textsize=small]{todonotes}
\usepackage{xargs}
\usepackage{xcolor}

\renewcommand{\leq}{\leqslant}
\renewcommand{\geq}{\geqslant}

\newcommand\myrestriction{\mathord\restriction}

\def\ZZ{\mbox{\boldmath $Z$}}

\newcommand{\Ascr}{\mathfrak A}
\newcommand{\Bscr}{\mathcal B}
\newcommand{\Cscr}{\mathcal C}
\newcommand{\Dscr}{\mathfrak D}

\newcommand{\Mscr}{\mathcal M}

\newcommand{\Tscr}{\mathcal T}

\newcommand{\Wscr}{\mathcal{W}}

\newcommand{\Aback}{\overleftarrow {A}}
\newcommand{\Bback}{\overleftarrow {B}}
\newcommand{\Cback}{\overleftarrow {C}}

\def\abar{\mbox{\boldmath $a$}}
\def\bbar{{\bf b}}
\def\cbar{{\bf c}}

\def\hbar{{\bf h}}

\def\vbar{{\bf v}}

\def\xbar{{\bf x}}
\def\ybar{{\bf y}}

\newcommand{\NN}{\boldmath{N}}
\def\bigI{\mathcal{I}}

\def\dim{{\rm dim}}
\def\sdim{{\rm sdim}}

\def\aut{{\rm aut}}

\def\icl{{\rm icl}}
\def\sdcl{{\rm sdcl}}

\def\cl{{\rm cl}}

\def\acl{{\rm acl}}
\def\th{{\rm Th}}
\DeclareMathOperator{\dcl}{dcl}

\makeatletter
\def\subsection{\@startsection{subsection}{3}%
  \z@{.5\linespacing\@plus.7\linespacing}{.3\linespacing}%
  {\bfseries\centering}}
\makeatother

\makeatletter
\def\subsubsection{\@startsection{subsubsection}{3}%
  \z@{.5\linespacing\@plus.7\linespacing}{.3\linespacing}%
  {\centering}}
\makeatother

\makeatletter
\def\myfnt{\ifx\protect\@typeset@protect\expandafter\footnote\else\expandafter\@gobble\fi}
\makeatother

\newtheorem{theorem}{Theorem}[section]

\newtheorem{corollary}[theorem]{Corollary}
\newtheorem{definition}[theorem]{Definition}
\newtheorem{lemma}[theorem]{Lemma}
\newtheorem{claim}[theorem]{Claim}

\newtheorem{example}[theorem]{Example}
\newtheorem{axiom}[theorem]{Axiom}
\newtheorem{problem}[theorem]{Problem}

\newtheorem{question}[theorem]{QUESTION}
\newtheorem{observation}[theorem]{Observation}
\newtheorem{fact}[theorem]{Fact}
\newtheorem{conclusion}[theorem]{Conclusion}
\newtheorem{remark}[theorem]{Remark}
\newtheorem{notation}[theorem]{Notation}

\newtheorem{assumption}[theorem]{Assumption}
\newtheorem{conjecture}[theorem]{Conjecture}

\newtheorem{definition/fact}[theorem]{Definition/Fact}

\newtheorem{construction}[theorem]{Construction}

\def\bK{\mbox{\boldmath $K$}}

\def\bL{\mbox{\boldmath $L$}}
\def\Fscr{\mathcal F}

\def\Sdim{\mathrm{Steiner\mbox{-}dim}}
\def\Ssdim{\mathrm{Steiner\mbox{-}sdim}}
\def\Smoves{\mathrm{Steiner\mbox{-}moves}}
\def\Sdet{\mathrm{Steiner\mbox{-}determines}}

\def\bbL{\hat \bL_0}

\newcommand{\pureindep}[1][]{%
  \mathrel{
    \mathop{
      \vcenter{
        \hbox{\oalign{\noalign{\kern-.3ex}\hfil$\vert$\hfil\cr
              \noalign{\kern-.7ex}
              $\smile$\cr\noalign{\kern-.3ex}}}
      }
    }\displaylimits_{#1}
  }
}

\newcommand{\indep}[2]{%
  \mathrel{
    \mathop{
      \vcenter{
        \hbox{%
\oalign{
\noalign{\kern-.3ex}\hfil$\vert$\hfil\cr
              \noalign{\kern-.7ex}
              $\smile$\cr\noalign{\kern-.3ex}
}
}
      }
}^{\!\!\!\!\!#2}_{\!\!\hspace{-0.1em}#1}
  }
}

\begin{document}
\title{Towards a finer classification of strongly minimal sets}


\author{John T. Baldwin
\\University of Illinois at Chicago}
\author{Viktor V. Verbovskiy
\\ Satbayev University, Kazakhstan}
\thanks{
 Partially supported by Simons
travel grant G5402 (418609)(first author); grant AP09259295 of SC of the
MES of the Republic of Kazakhstan (second author).}

\renewcommand{\shortauthors}{Baldwin-Verbovskiy}
\renewcommand{\shorttitle}{Finer Classification}



\date{\today}    \maketitle

\begin{abstract} Let $M$ be strongly minimal and constructed by a `Hrushovski
construction' with a single ternary relation. If the Hrushovski
algebraization function $\mu$ is in a certain class $\Tscr$ ($\mu$ triples)
we show that for independent $I$ with $|I| >1$, $\dcl^*(I)= \emptyset$ (*
means not in $\dcl$ of a proper subset). This implies the only definable
truly  $n$-ary functions $f$ ($f$ `depends' on each argument),  occur when
$n=1$.
 We prove
  for Hrushovski's original construction and  for the strongly minimal
$k$-Steiner systems of Baldwin and Paolini that the symmetric definable
closure, $\sdcl^*(I) =\emptyset$ (Definition~\ref{defsdcl}). Thus, no such
theory admits elimination of imaginaries. As, we show that in an arbitrary
strongly minimal theory, elimination of imaginaries implies $\sdcl^*(I)
\neq   \emptyset$. In particular, such strongly minimal Steiner systems
with line-length at least 4 do not interpret a quasigroup, even though they
admit a coordinatization if $k = p^n$. The case structure depends on
properties of the Hrushovski $\mu$-function. The proofs depend on our
introduction, for appropriate $G \subseteq \aut(M)$ (setwise or pointwise
stablizers of finite independent sets), the notion of a $G$-normal
substructure $\Ascr$ of $M$ and of a $G$-decomposition of any finite such
$\Ascr$.
 These results lead to a finer classification of strongly minimal
structures with flat geometry, according to what sorts of definable
functions they admit.
\end{abstract}

{\bf An omission in Definition 2.7 and a correction to Definition 2.9 are in
the errata at the end along with more minor changes. }

The original motives for this paper were to show i) (first author) (Pure)
strongly minimal Steiner systems with line length greater than $3$ admit
neither $\emptyset$-definable quasigroups  nor more generally `non-trivial'
$\emptyset$-definable binary functions (Theorem~\ref{noqg},
Theorem~\ref{omni2S}) and ii) (second author) the original Hrushovski
construction does not admit elimination of imaginaries
(Theorem~\ref{no-sym-fun}).

We proved both of these conjectures. The failure of
elimination may not be surprising. Perhaps more surprising is that the proof
requires a detailed analysis of  actions  of two kinds of subgroups of the
automorphism group of a model to $T_\mu$ on the algebraic closure of a finite
sets and depends on specific values of the $\mu$-function. And given this
result, the surprise is the possibility   to construct a function in 2
variables in a very specific circumstance (Section~\ref{ce}).

The principal innovations of this paper are i) a new characterization of
elimination of imaginaries in terms symmetric definable closure
(Section~\ref{sdcl}), ii) decomposing models of flat strongly minimal
theories with respect to newly defined subgroups $G_I,G_{\{I\}}$
(Section~\ref{decomp})  and  using these tools to iii) demonstrate
imaginaries cannot be eliminated in certain flat strong minimal sets
(Section~\ref{geq3}-\ref{Steiner}). Our solutions led us to formulate a finer
classification of strongly minimal sets and to explore methods to vary the
Hrushovski strongly minimal set construction  to exhibit new classes, make
applications in combinatorics, and find  automorphism groups that act more
transitively.

 The question of elimination of imaginaries for
{\em ab initio} stongly minimal sets has largely lain fallow for 25 years
since B.~Baizhanov \cite{Baiz} asked whether any strongly minimal theory
 in a finite vocabulary
that admits elimination of imaginaries must be an algebraically closed field.
  \cite[p 160]{Hrustrongmin} observed that flat geometries  obeyed
   the weaker `geometric elimination of
   imaginaries' and this was adequate for studying the geometry. But,
    as for Zilber,
our goal is not to classify the geometries associated with strongly minimal
theories but to classify the theories.

\begin{notation}\label{defT}{\rm

\begin{enumerate}
\item $\hat T_\mu$ denotes a strongly minimal theory constructed with the
    same $\delta$, same vocabulary of one ternary relation $R$ (required to
    be a hypergraph), the same $\bL_0$, and an appropriate $\mu$  as in the
    main construction in \cite{Hrustrongmin}.
\item $T^S_\mu$ denotes a strongly minimal theory of Steiner systems
    constructed with the $\delta$ and vocabulary $\{R\}$ as in
    \cite{BaldwinPao})
(Definition~\ref{twodelta}).

\item $T_\mu$ is used for a strongly minimal theory of either sort; in both
    cases the geometry is {\em flat} but not {\em disintegrated (trivial)}.
\end{enumerate}}

\end{notation}

 We say a theory $T_\mu$ {\em triples} if for any good pair
$C/B$ (Definition~\ref{prealgebraic}.(3)), $\delta(B) \geq
    2$ and $|C|> 1$ imply $\mu(C/B)\geq 3$. This implies that  every primitive
extension of a `well-placed' (Definition~\ref{wpdef}) base has at least 3
copies in the generic.

To analyze the  elimination of imaginaries in {\em arbitrary} strongly
minimal theories we introduce the notion: $a \in (s)\dcl^*(X)$ (Section
\ref{sdcl}).
 The $*$ in $a \in (s)\dcl^*(X)$ means
every element of $X$ is used  to witness that $a$ is in the (symmetric)
definable closure of $X$. It is known ({\cite{Pillayacf},
\cite[1.6]{CasanovasFarreelim}) that for any strongly minimal set with
$\acl(\emptyset)$ infinite, elimination of imaginaries is equivalent to the
finite set property. It follows (Lemma~\ref{fckey})  that if $\sdcl^*(I) =
\emptyset$ for some independent set $I$, $T$ fails to eliminate imaginaries.
More strongly, $\dcl^*(I) = \emptyset$, implies there is no definable truly
$|I|$-ary function.  Below, $\emptyset$-definable abbreviates parameter-free
definable.

\begin{theorem}[Main Results]\label{mr1} 

 Let  $T_\mu$ be a strongly minimal theory  as in Notation~\ref{defT}.
  Let $I =\{a_1,\dots, a_v\}$ be a tuple of
independent points with $v\ge 2$.
\begin{enumerate}
\item  If $T_\mu$ triples then $\dcl^*(I) = \emptyset$ and every definable
    function is essentially unary (Definition~\ref{truen}).
    \item Without the tripling assumption, $\sdcl^*(I) =\emptyset$ and
        there are no $\emptyset$-definable symmetric (value does not depend
        on order of the arguments) truly $v$-ary function.
\end{enumerate}
Consequently, in both cases $T_\mu$ does not admit elimination of
imaginaries.\footnote{Note that each of the conditions $\dcl^*(I) =
\emptyset$ and $\sdcl^*(I) =\emptyset$ is stronger than failure to eliminate
of imaginaries. For example,  an affine space does not admit elimination of
imaginaries; however $\sdcl^*(I) \ne\emptyset$ because of the addition.}
Nevertheless the algebraic closure geometry is not disintegrated.
\end{theorem}

The crucial tool for this result is a close study of  the action of two kinds
of subgroups of the automorphism group of a structure $M$:
for a finite independent set $I$, the subgroups $G_I$ and $G_{\{I\}}$ of
$\aut(M)$ fixing $I$ setwise and pointwise respectively.

Fix a subgroup $G$  of the automorphism group of a generic structure  of
either sort.
The field-theoretic notion of normality in fields  uses  that the
definable
 closure $\dcl(X)$  is  the field (with a vector space structure) generated by   a set   of
parameters $X$ to study the solutions of algebraic equations.  In our
situation the definable closure may contain infinitely many elements
generated by definable unary functions over which we have little control. So,
we define the notion of a finite $G$-normal set (Definition~\ref{iclnot}).
 Then we provide a $G$-invariant tree
decomposition of such a set. The pointwise stabilizer is sufficient (using
triples) for an induction on the height of decompositions of $G$-normal sets
to show the absence of truly binary functions. But stronger result showing
(without assuming triples) that $\sdcl^*(I)=\emptyset$ and thus failure of
elimination of imaginaries (although possibly with non-trivial binary
functions) relies on $G_{\{I\}}$ and is more complex.


 Our results indicate a
profound distinction at the level of definable functions rather than geometry
between the `field-like' strongly minimal sets and the known counterexamples
to Zilber's conjecture, indeed,  within the counterexamples.
 We refine the classification of  strongly minimal theories (with flat geometry)
by the existence of definable truly $n$-ary functions.

Cases 1) and 2)  have distinct theories of the $\acl$-geometry (of saturated
models). By \cite{EvansFerII} and\footnote{In a private communication
Mermelstein showed the infinite rank case \cite{Paoliniwstb} of Steiner has
the same geometry as the original strongly minimal example.}
 \cite{Merflat} the geometries of the countable saturated models of
theories in  each of  2a) and 2b) are elementarily equivalent (indeed have
isomorphic localizations). Within 2) we find for $\dcl$ an essentially unary
nature similar to that which distinguishes `trivial $\acl$' in the Zilber
classification. The distinction concerns properties of $\dcl$ in $M$, not an
associated $\acl$-geometry.

We find the following:

\begin{remark}[Classes of Theories with flat $\acl$-geometries]\label{classes}\mbox{}\\

\begin{enumerate}
\item {\bf disintegrated geometry} For any  $A$, $acl(A)  = \bigcup_{a \in
    I} \acl(a)$;

    \smallskip

\item  {\bf strictly flat geometry}
    (Definition~\ref{allflatdef}.~\ref{defflat}) $M$   is not
    $\acl$-disintegrated but:
\begin{enumerate}
\item $M$ is {\em $\dcl$-disintegrated}: $\dcl(I) = \bigcup_{a \in I}
    \dcl(a)$ for independent $I$ (no $\emptyset$-definable truly $n$-ary
    functions);

\item $M$ is  not {$\dcl$-disintegrated}: For some $n$ there are truly
    $n$-ary 
    functions:
\begin{enumerate}
\item $M$ is {\em $\sdcl$-disintegrated}: $\sdcl(I) = \bigcup_{a \in I}
    \sdcl(a)$ for independent $I$ (no commutative $\emptyset$-definable
    truly $n$-ary functions);

\item $\emptyset$-definable binary functions with domain $M^2$ exist;
    e.g. quasigroups \cite{BaldwinsmssII}\footnote{The vocabulary in
    the construction has two ternary relations.}  and non-commutative
    counterexamples found here.
\end{enumerate}
\end{enumerate}
\item  {\bf Further examples: the strongly minimal set is a relativised
    reduct of an almost strongly minimal structure\footnote{The universe is
    the algebraic closure of a strongly minimal structure.}}

\begin{enumerate} \item ternary rings \cite{Baldwinautpp} that coordinatize a
    non-desarguesian plane\footnote{Because of the {\em ad hoc} nature of
    this construction, the methods of the current paper do not apply. The
    geometry has not been analyzed.}. There are parameter definable
    binary functions; but the ternary ring is not a composition of the
    `addition' and `multiplication' functions. \item 2-transitive
strongly minimal sets~\cite[Proposition 18]{Hrustrongmin},
    \cite{BaldwinsmssIII}.

\item  \cite{MullTent} construct a strongly minimal set which is not
    flat\footnote{More precisely, the $\acl$-geometry is $2$-ample but
    not $3$-ample.}  within a geometry of dimension $3$. The
    $\delta$-function is no longer submodular.
\end{enumerate}


\end{enumerate}
\end{remark}


  No
theory in class 1) or 2) (and finite vocabulary) admits elimination of
imaginaries. However, Verbovskiy \cite{Verbovskiy06} has an example with
elimination of imaginaries in an infinite vocabulary. Unlike those with
locally modular geometries, strongly minimal theories with strictly flat (not
disintegrated) geometries, like  the field type, have continuum many
automorphism of each countable model (Lemma~\ref{manyaut}).


 This
diversity is obtained by realizing that the `Hrushovski construction'
actually has 5 parameters: $(\sigma, \bL^*_0, \epsilon, L_0, {\bf U})$.
$\bL^*_0$ is a collection of finite structures in a vocabulary $\sigma$, {\em
not necessarily closed} under substructure. $\epsilon$ is a predimension as
in Hrushovski with the requirement that it be flat. $\bL_0$ is a subset of
$\bL^*_0$ defined using $\epsilon$. From such an $\epsilon$, one defines
notions of $\leq$, primitive extension, and good (minimally simply algebraic)
pair.  The function $\mu$ counts the number of allowed realizations of a good
pair. Hrushovski gave a technical admissibility condition on $\mu$ that
ensured the theory is strongly minimal rather than $\omega$-stable of rank
$\omega$. Fixing a class ${\bf U}$  of functions $\mu$ satisfying variants of
this condition provides examples satisfying a wide range of combinatorial and
algebraic conditions.

We show  the elimination of imaginaries fails when $\bL^*_0$ is the
collection of finite linear spaces. As noted above, B.~Baizhanov  asked
whether any strongly minimal theory
 in a finite vocabulary,
that admits elimination of imaginaries defines a field. We show the most
evident counterexamples do not eliminate imaginaries.
But, the
question of whether this can be extended when the class $\bL^*_0$ is expanded
to arbitrary $\forall \exists$ classes of finite structures seems wide open.

There are a number of applications of these various construction methods in
universal algebra and combinatorics. Early on, \cite{Baldwinautpp} used the
general method to construct an $\aleph_1$-categorical projective plane at the
very bottom of the Lenz-Barlotti hierarchy. We now describe some more recent
work.

Work by Barbina and Casanovas in model theory \cite{BarbinaCasa} and by
Horsley and Webb in
 combinatorics \cite{HorsleyWebb}
centers on the
 constructions of Steiner triple systems  that arise
 from Fra{\" \i}ss\'{e} constructions.  \cite{HorsleyWebb} obtain countably
 homogeneous Steiner systems that omit `good' classes of finite Steiner
 systems.
In contrast to the model theoretically complex locally finite generics
 \cite{BarbinaCasa}, the
construction techniques here give theories that are strongly minimal, the
geometric building blocks of model theoretically tame structures.

\cite{BaldwinsmssII},  building on \cite{Steinpnas, GanWer}, shows two
results. 1) if
 the line length $k$ is
a prime power then a strongly minimal Steiner system admits a
`coordinatization' by a quasigroup. We show here, as a by-product of our main
construction, this coordinatization is not definable in the vocabulary
$\{R\}$. 2) Nevertheless, if $k$ is a prime power, there are strongly
 minimal quasigroups\footnote{As in \cite{BarbinaCasa}, $H$ is the graph of the quasigroup operation and $R$ is
 collinearity.} $(Q,*)$ (created in the vocabulary $\{H,R\}$) with
 $\bL^*_0$ $\forall\exists$ axiomatizable) which
 induce Steiner systems.

 \cite{BaldwinsmssIII}
  generalizes the notion of cycle
decomposition \cite{CameronWebb} from Steiner 3-systems to that of path graph
for Steiner $k$-systems and
 introduce  a uniform method
of proof for the following results. For each of the following conditions,
there are $2^{\aleph_0}$ families of elementarily equivalent Steiner sytems,
each with $\aleph_0$ countable models and one model of each uncountable
cardinal satisfying the following condition: i) (extending \cite{Chicoetal})
each Steiner triple system is $\infty$-sparse and has a uniform
path graph; ii) (extending \cite{CameronWebb} each Steiner $k$-system (for
$k=p^n$) is $2$-transitive and has a uniform path graph (infinite cycles only) 
iii) extending \cite{Fujiwaramitre} each is anti-Pasch (anti-mitre); iv)
Items
   ii) and iii) have definable quasi-group structure. Moreover, by varying
$\bL^*_0$ classes we can demand all models are $2$-transitive. Unlike most
combinatorial constructions, each example presents a {\em class of models}
with the desired property.

Section~\ref{twocon} outline the general framework of {\em ab initio}
Hrushovski strongly minimal sets and then describes the variant for Steiner
system. In Section~\ref{sdcl}, we introduce subgroups $G_I$ and $G_{\{I\}}$
of $\aut(M)$ (fixing $I$ setwise and pointwise respectively) and explain the
connections of definable closure and symmetric definable closure with the
elimination of imaginaries. The key tool of $G$-tree-decomposition appears in
Section~\ref{decomp} along with the basic properties. There are now two steps
in the proof. Section~\ref{geq3} proves the non-$\emptyset$-definability of
truly $n$-ary functions for triplable theories, using $\dcl^*$ and $G_I$.
Section~\ref{ce} shows the necessity of the tripling hypothesis.
Section~\ref{bouflow} describes the distinction between flowers and bouquets
which is essential for Section~\ref{geq2}, the proof that $\sdcl^*(I)
=\emptyset$ for finite independent sets in a standard Hrushovki theory using
$G_{\{I\}}$. In Section~\ref{Steiner} we adapt that proof to Steiner systems
and prove that examples with line length at least 4 from \cite{BaldwinPao} do
not define quasigroups. We raise some further questions in
Section~\ref{furtherwork}.

\section{Two Contexts}\label{twocon}

In Subsection~\ref{gcon} we give an axiomatic description of the properties
of functions $\delta$ and $\mu$ for  Hrushovski constructions of
$\omega$-stable and strongly minimal sets. In Subsection~\ref{HLS} we
describe the specific definitions of $\delta$ for the two main cases
considered here: the original Hrushovski context and Steiner systems.

In Section~\ref{gcon} we denote our classes as $( \bL,\epsilon)$ with various
decorations; these conditions provide a general framework for the study of
the family of flat strongly minimal sets announced in the previous paragraph.
From Section~\ref{HLS} on, we write $(\bK,\delta)$ to emphasize the
restriction to a single ternary relation, particular choices of $\delta$ and
to prepare for the variations in \cite{BaldwinsmssIII}.

 \numberwithin{theorem}{subsection}

\subsection{General Context} \label{gcon}

The notions in this section are well-known, under various names. We both fix
the notation used here and give some of the definitions in a greater
generality needed here but not in e.g., \cite{Hrustrongmin}).
 Fix a countable relational vocabulary $\tau$.  We write $\bL^*$ for the
collection of all $\tau$-structures and $\bL^*_0$ for the finite
$\tau$-structures.

All constructions studied here satisfy the properties of 3.4--3.7 of
\cite{BaldwinShiJapan} and flatness, which  follow from flatness of the
underlying `predimension' function $\epsilon$ \cite[3.8, 3.10]{BaldwinPao}.


\begin{axiom} \label{yax}
Let ${\bL_0}$ be a countable subset of $\bL^*_0$ that is closed under
isomorphism\footnote{In this paper $\bL_0$ is closed under substructure. But
this condition is relaxed to construct flat strongly minimal quasigroups in
\cite{BaldwinsmssII}.}. Let $\bbL$ be the collection containing any union of
members of $\bL_0$. Further $\epsilon$ is a map from $L^*_0$ into $\ZZ$. We
require that $\bL_0, \hat \bL_0, \epsilon$ satisfy the following
requirements. Let $N\in {\bbL}$ and $A,  B, C \in {\bf L_0}$ be substructures
of $N$.

\begin{enumerate}
\item $\epsilon(\emptyset) = 0$
\item If $B \in \bL_0$ and  $A \subseteq B$ then $\epsilon(A) \geq 0$.
\item If $A$, $B$, and $C$ are disjoint then $\epsilon(C/A) \geq
    \epsilon(C/AB)$, where $\epsilon(C/A) = \epsilon(C\cup A)-\epsilon(A)$.
\item  If $A,B,C$ are disjoint subsets of $N\in \hat \bL_0$ and
    $\epsilon(A/B) - \epsilon(A/BC) = 0$ then $r(A,ABC,C) = \emptyset$.
    \footnote{See Definition~\ref{Rdef} for details concerning
    $r$. In the general case, we count each of a sequence of predicates
    $R_i$ separately. We make appropriate modification for linear spaces
    when they are considered.}
\item $\epsilon$ is flat (Definition~\ref{allflatdef})\cite{Hrustrongmin,
    BaldwinPao}).
\item {\bf Canonical Amalgamation} If $A\cap B = C$, $C\leq A$ and $A,B, C
    \in \bL_0$ there is a direct sum  $G = A \oplus_C B$ such that $ G \in
    \bL_0$.  Moreover, $\epsilon( A \oplus_{C} B)= \epsilon(A) +\epsilon(B)
    - \epsilon(C)$ and any $D$ with $C \subseteq D \subseteq A \oplus_{C}
    B$ is also  free.  Thus, $B \leq G$.
\end{enumerate}
\end{axiom}

Disjoint union is the canonical amalgamation for the basic Hrushovski
construction and Definition 3.14 of \cite{BaldwinPao} gives the appropriate
notion satisfying Axiom~\ref{yax}.5 for linear spaces. Axiom~\ref{yax}.2 can
be rephrased as: $B\subseteq C$ and $A \cap C = \emptyset$ implies
$\epsilon(A/B) \geq \epsilon(A/C)$; so we can make the following definition.

\begin{definition} Extend $\epsilon$ to $d: {\bbL} \times {\bL}_0
 \rightarrow \NN$ by for each $N \in {\bbL}$,
$d(N,A) = \inf \{\epsilon(B): A \subseteq B \subseteq_{\omega} N\}$,
$d_N(A/B) = d_M(A \cup B) -d_M(B)$. We usually write $d(N,A)$ as $d_N(A)$ and
omit the subscript $N$ when clear.
\end{definition}

Hrushovski defined  a crucial property of the algebraic closure (pre)
geometry of his geometries\footnote{A pregeometry/matroid becomes a geometry
by modding out $\cl(\emptyset)$.}: {\em flat}. \cite{BaldwinPao} generalized
the notion of flatness for a
 pregeometry
   to a general predimension function.

\begin{definition}\label{allflatdef}
\begin{enumerate}
\item \label{flatdelta} Consider a class $(\bL_0,\epsilon)$,
    $N \in \bL_0$ and a sequence $F_1, \dots, F_s$ of subsets of $N$.
    For $\emptyset \subsetneq T \subseteq \{1, \ldots , s\}=I$, we let $F_T
 = \bigcap_{i\in T} F_i$ and $F_{\emptyset} = \bigcup_{1\leq i \leq s}
 F_i$. We say that $\epsilon$ is {\em flat} if for all such $F_1, \ldots ,
 F_s$
we have:
$$ (*) \hspace{.1in} \epsilon(\bigcup_{1\leq i \leq s} F_i)
 \leq \sum_{\emptyset \neq T} (-1)^{|T|+1} \epsilon(F_T).$$
\item\label{defflat}  Suppose $(A, \cl)$ is a pregeometry on a structure
    $M$ with dimension function $d$ and $F_1, \ldots , F_s$ are
    finite-dimensional $d$-closed subsets of $A$.  Then $(A, \cl)$ is {\em
    flat} if $d$ satisfies equation $(*)$.

    \item $(A, \cl)$ is {\em strictly flat} if it is flat but not {\em
        distintegrated} $(\acl(ab) \ne \acl(a) \cup \acl(b))$.
\end{enumerate}
\end{definition}


%
%
%
%
%
%
%

What Hrushovski called {\em self-sufficient} closure is in the background.


\begin{definition}\label{strongdef}

	\begin{enumerate}[(1)]
\item We say $A\subseteq N$ is strong in $N$ and write $A\le N$  if
    $\epsilon(A) \le \epsilon(C)$ for any $C$ with $A\subseteq
    C\subseteq_\omega N$,  where $\subseteq_\omega$ stands for `be a finite
    subset'.

  \item For any $A \subseteq B \in L_0$, the intrinisic (self-sufficient)
      closure of $A$, denoted $\icl_B(A)$, is the smallest superset of $A$
      that is strong in $B$.
\end{enumerate}
\end{definition}

Note that $A\le N$ if and only if $d_N(A) = \epsilon(A)$. It is well-known
that when $\epsilon$ is integer valued then $\icl(A)$ is finite if $A$ is.
The following definitions describe  the pairs $B \subseteq A$ such that in
the generic model $M$, $A$ will be contained in the algebraic closure of $B$.

\begin{definition}\label{prealgebraic}
\begin{enumerate}
\item $A$ is a {\em  primitive extension\footnote{In \cite{Hrustrongmin},
    $0$-primitive is called {\em simply algebraic} and good
    \cite{Zieglersm} is called {\em minimally simply algebraic}.}} of $B$
    if $B\leq A\cup B$, $A\cap B = \emptyset$, and there is no $A_0$ with
    $\emptyset\subsetneq A_0 \subsetneq A$ such that $B \leq B\cup A_0 \leq
    B\cup A$.  $A$ is a {\em $k$-primitive extension} if, in addition,
    $\epsilon(A/B) =k$. We stress that in this definition,
while $B$ may be empty, $A$ cannot be.

Sometimes primitive is used with $B\subsetneq A$ and the primitive as
$A-B$. In that case we sometimes write $\hat A$ for $A-B$ when the
disjointness is essential.

	\item  We say that the $0$-primitive pair $A/B$ is {\em
good\footnote{In the Hrushovski case, it is equivalent to say  if there is
no $B' \subsetneq B$  such that $(A/B')$ is $0$-primitive. But for linear
spaces these conditions are  no longer equivalent.}} if every element of
$B$ is in some relation with an element of $A$.

\item If $A$ is $0$-primitive over $D$ and $B \subseteq A$ is such that
$A/B$ is good, then we say that $B$ is a {\em base} for $A$ over $D$ if $B$
is a minimal subset of $D$ such that $A/B$ is $0$-primitive.
\end{enumerate}
\end{definition}

\begin{remark}\label{minmax} {\rm
In the Hrushovski case, the definition of $\delta$ makes clear that the base
defined by minimality in Definition~\ref{prealgebraic}.3 is also the maximal
subset of $D$ that is related to an element of $A$. This fails in the Steiner
case; see Lemma~\ref{primchar}.}
\end{remark}

\begin{definition}     \label{Kmu}	Good pairs
 were defined in Definition~\ref{prealgebraic}.
	\begin{enumerate}[(1)]
	\item\label{itemKmu} {\bf Adequacy condition:} Fix a function $\mu$
assigning to every isomorphism type $\boldsymbol{\beta}$ of a good pair
$C/B$ in $\bL_0$
 a number $\mu(\beta) = \mu(B,C)  = \mu(C/B) \geq \delta(B)$. 

\item For any good pair $(A/B)$ with $B \subseteq M$ and $M \in \hat
    \bL_0$, $\chi_M(A/B)$ denotes the maximal number of disjoint copies of
    $A$ over $B$ in $M$. {\em A priori},  $\chi_M(A/B)$ may be $0$.

	\item\label{Kmuitem} Let $\bL_{\mu}$ be the class of structures $M$ in $ \bL_{0}$
 such that if $(B,C)$ is a good pair  then $\chi_M(B,C)  \leq \mu(B,C)$.
	\item\label{Kmuhatitem} $\hat \bL_\mu$ is the universal class generated by
 $\bL_\mu$. 

 \item\label{item_dcl}  [$d$-closed] For $M\in \hat{\bL}_\mu$ and $X
     \subseteq M$, $X$ is $d$-closed in $M$ if $d(a/X)= 0$ implies $a \in
     X$ (equivalently, for all $Y \subseteq_{\omega} M-X$, $d(Y/X) >0$).

\item\label{mu-d} Let $\bL^{\mu}_{d}$ consist of those $M\in \hat{\bL}_\mu$
    such  that $M \leq N$ and $N \in \hat \bL_\mu$ imply $M$ is $d$-closed
    in $N$.
\end{enumerate}	
\end{definition}

The restriction on $\mu$ in Definition~\ref{Kmu}.1 comes from
\cite{Hrustrongmin}.  It appeared as a useful    condition to guarantee the
amalgamation. Surprisingly, we find the following slight strengthening plays
a central role in preventing the definability of {\em any} truly binary
 functions.

\begin{definition}\label{triples}
We say $T_\mu$ {\em triples} if  $\delta(B)
\geq 2$ implies $\mu(C/B)\geq 3$ for any good pair $C/B$  with $|C|>1$.
\end{definition}

 We can show that any element of $\hat \bL_\mu$ (not just $\bL_\mu$)
 can be amalgamated (possibly with identifications) over a (necessarily finite)
  {\em strong} substructure $D$ of $F$ with a strong extension of $D$ to a member  $E$
   of $\bL_\mu$. This yields the following conclusions; they are largely
   the same as \cite{Hrustrongmin}; in order to treat
   line length $3$, \cite{BaldwinPao} make the adequacy requirement Definition~\ref{Kmu}.1
   apply only when $|B|\geq 3$ and add $\mu(\boldsymbol{\beta}) \geq 1$, if $\beta =
   \boldsymbol{\alpha}$ (Definition~\ref{linelength}).

Recall that a {\em generic model} for a class $(\bL_0,\leq)$ is an $M$ such
that if  $A\leq M$, $A \in \bL_0$ and if $A,A' \leq M$ are isomorphic, the
isomorphism extends to an automorphism of $M$.

  \begin{conclusion} \label{conclude} 
Suppose $L_\mu$ satisfies the properties described in Axiom~\ref{yax} and
Definition~\ref{Kmu}:

If $D \leq F \in \hat \bL_\mu$ and $D \leq E \in \bL_\mu$
   then there is $Q\in \hat \bL_\mu$ that embeds (possibly with identifications)
   both $F$ and $E$ over $D$.  Moreover, if $F \in \bL^\mu_d$, then $F = Q$.
   In particular, $(\bL_\mu, \leq)$ has the amalgamation property, and there is a
    generic structure $\mathcal{G}_\mu \in \hat{\bf L}_\mu$ for $(\bL_\mu, \leq)$.
  \end{conclusion}

The more refined conclusion of model completeness is shown in \cite{Holland5}
and for the linear space case in \cite{BaldwinPao}.
\begin{conclusion}\label{conclude2} Under the hypotheses of Conclusion~\ref{conclude} there is
a collection $\Sigma_\mu$ of $\pi_2$ sentences which \begin{enumerate}
\item axiomatize the complete theory $T_\mu$ of the class $\bL^{\mu}_d$,
    $d$-closed
    models
    in
    $\bbL$.
\item $T_\mu$ is model complete and strongly minimal.
\item The $\acl$-geometry of $T_\mu$ is flat
    (e.g.~\ref{allflatdef}.~\ref{defflat})
\end{enumerate}
\end{conclusion}

\subsection{$3$-hypergraphs and Linear Spaces}\label{HLS}

We now describe the main examples for this paper of the context axiomatized
in Subsection~\ref{gcon}. We replace $\bL$ and $\epsilon$ by
$\bK$ and $\delta$ to indicate that properties here may depend on the
specific definition of the class.

\begin{definition}[Context]\label{Rdef}
\begin{enumerate}
\item The vocabulary $\tau$ contains a single ternary relation $R$. We
    require that $R$ is a predicate of $3$-elements sets (distinct in any
    order).

\item Let $A,B,C$ each be a subset of $D \in \bK^*$.  We
    write $R(A,B,C)$ for the collection of tuples
     $\xbar$ such that $D \models R(\xbar)$ and
$\xbar$ intersects each of $A,B,C$. (The letters may be repeated to
    indicate only two sets are represented.)
		We write $R(A,B)$ for $R(A, A\cup B, B)$. We write $r(A,B,C)$ for
the number\footnote{For the Steiner system case count lines and compute
$\delta$ as in Section~\ref{Steiner}.} of tuples (up to permutation) in
$R(A, B, C)$. Finally for $A\in \bK^*_0$, $r(A) = r(A,A,A)$.
\end{enumerate}
\end{definition}

We restrict to integer valued  $\delta$ which is essential (but  not
sufficient) to guarantee $\omega$-stability.  The crucial distinction between
\cite{Hrustrongmin}  and linear spaces is restricting  the class of finite
structures by more than the assertion that $R$ is a ternary predicate of sets
($3$-hypergraph).

\begin{definition}{The  choices  here for $\delta$}
\label{twodelta} \begin{itemize}
\item ($3$-hypergraph) For a finite $\tau$-structure $A$, $\delta(A) = |A|-
    r(A)$.
\item (linear space)
\begin{enumerate}
\item
 A linear space is a  $\tau$-structure such that $2$-points determine a
    unique line.
We interpret $R$ as collinearity. By convention two unrelated elements
    constitute a {\em trivial line}.
\item For $B, \ell$ subsets of  $A$, we say $\ell \in L(B)$ (is supported
    by $B$) if $\ell$ is a maximal $R$-clique contained in $A$ and $|\ell
    \cap B| \geq 2$.

 \item Let
	$$\delta(A) = |A| - \sum_{\ell \in L(A)} (|\ell|-2).$$

\item Then  $\bK_0$ is the collection of finite  linear spaces $A$
such that for any   $A' \subseteq A$, $\delta(A') \geq 0$.
\end{enumerate}
\end{itemize}
\end{definition}

Note, that as opposed to Section~\ref{gcon}, we have restricted $\bK_0$ both by (a different) $\delta$ and
by the linear space axiom.

\begin{notation}
\label{bhat} Let $\delta(B/A)$ denote $\delta(A\cup B) - \delta(A)$.
Suppose $A\cap B = C$ and $A,B, C
    \in \bL_0$:
\begin{enumerate} \item We say $A$ and $B$ are $\delta$-independent
if $\delta( A \cup B)= \delta(A) +\delta(B) - \delta(C)$.
    \item We say $A$ and $B$ are  {\em fully independent} over $C$ if there
        are no relations involving elements from  each of $A-C$ and $B-C$
        and possibly $C$,  i.e. $R(A-C, A\cup B, B-C) = \emptyset$.
        \end{enumerate}
\end{notation}
Just rewriting the definition, we have $\delta(A\cup B/C) = \delta(A/B\cup C) +
\delta(B/C)$.

\begin{remark}\label{find} Note that when $\delta$ just counts relations as in
\cite{Hrustrongmin}, $\delta$-independence implies full independence. The
situation is more complicated for linear spaces; see Definition~\ref{decomp2}
and Remark~\ref{getfind}.
\end{remark}

 The following useful tool is easy from the definition
(\cite[4.7]{BaldwinPao}).

\begin{lemma}\label{primchar1}  If $C$ is $0$-primitive over $\Dscr$ with base
$B \subseteq \Dscr$ and $|C-B|\geq 2$, then every
point $b\in B$ satisfies $R(c_1,c_2,b)$ for some $c_1,c_2 \in C$ and every
point in $C$ satisfies at least two instances of $R$.
\end{lemma}

We single out an isomorphism type of a good pair (a $1$-element extension,
$\boldsymbol{\alpha}$) that plays a special role in the general proof that
elimination of imaginaries fails. $T_\mu$ is a Steiner $k$-system if
$\mu(\boldsymbol{\alpha}) =k-2$ \cite{BaldwinPao}.
%

\begin{notation}[Line length]\label{linelength} We write $\boldsymbol{\alpha}$ for the isomorphism
type of the good pair $(\{b_1,b_2\},a)$ with
 $R(b_1,b_2,a)$. 

 By Lemma 5.18
of \cite{BaldwinPao}, lines in models of the theory $T_\mu$ of a Steiner
system have length $k$ if and only if $\mu(\boldsymbol{\alpha}) = k-2$.
 \end{notation}


\section{Elimination of imaginaries and (Symmetric) Definable closure}\label{sdcl}
\numberwithin{theorem}{section}

In this section   we  study arbitrary strongly minimal sets
  and lay out the
connections among elimination of imaginaries, definable closure and a new
notion {\em symmetric definable closure}. Recall that
\cite[III.6]{Shelahbook} introduced extensions of structures by {\em
imaginary elements} and \cite{Poizatim} discovered the importance of a theory
not needing to add them.

\begin{definition}\label{defei} Elimination of Imaginaries:
\begin{enumerate}
\item  \cite{Poizatim} A theory $T$ admits {\em elimination of imaginaries}
    if for every model $M$ of $T$, for every formula $\varphi (\xbar,
    \ybar)$ and for every $\abar \in M^n$ there exists $\bbar \in M^m$ such
    that
$$
\{ f\in\aut (M) \mid f|\bbar = id_{\bbar} \} =
\{ f\in \aut (M) \mid f(\varphi (M,\abar))=
\varphi (M,\abar)\}
$$
\item \cite[Theorem 16.15]{Poizatbook} A theory $T$ admits weak elimination
    of imaginaries if and only if for every formula $\phi (\xbar; \abar)$
    there exists a formula  $\psi_{\abar} (\xbar;\ybar)$ such that there
    are only finitely many parameters $\bbar_1, \dots, \bbar_n$ such that
    each of $\psi_{\abar} (\xbar; \bbar_1), \dots, \psi_{\abar} (\xbar;
    \bbar_n)$ is equivalent to $\phi (\xbar; \abar)$.
\end{enumerate}
\end{definition}

Weak elimination arose in \cite{Poizatbook}; we use the
precise version of \cite[Def. 2.3]{Tsuboialgaut}.

\begin{definition}\label{deffc}\rm 
A finite set $F=\{\abar_1,\dots,\abar_k\}$ of tuples from $M$ is said to be
coded by $S=\{s_1,\dots,s_n\}\subset M$ over $A$ if
$$\sigma (F)=F \Leftrightarrow \sigma |S = \mbox{id}_S
\quad\mbox{for any }\sigma \in\aut(M/A).$$

We say $T=\th (M)$ has {\em the finite set property} if every finite set of
tuples $F$ is coded by some set $S$ over $\emptyset$.
\end{definition}

Part 1) of the next result appears in {\cite{Pillayacf}. Part 2) is
\cite[1.6]{CasanovasFarreelim}. \cite{YonedaweiI,YonedaweiII} clarifies the
relation among variants of elimination of imaginaries and shows Pillay's
assumption of infinite algebraic closure is essential.
\begin{fact}\label{wefc}
\begin{enumerate}
\item Every strongly minimal theory such that $\acl(\emptyset)$ is infinite
    has weak elimination of imaginaries.

\item If $T$ admits weak elimination of imaginaries then $T$ satisfies the
    finite set property if and only if $T$ admits elimination of imaginaries.
\end{enumerate}
\end{fact}

Immediately from Fact~\ref{wefc}, since in almost
all\footnote{\cite{BaldwinsmssIII}  gives general conditions that imply
$\acl(\emptyset)$ is infinite or empty.} in the examples studied here
$\acl(\emptyset)$ is infinite \cite[Fact 5.26]{BaldwinPao}, if $T$ admits
elimination of imaginaries there is an $\emptyset$-definable truly binary
function (given by the coding of a pair of independent points).

Note that elimination of imaginaries immediately yields finite coding, so in
proving that elimination of imaginaries fails it suffices to prove finite
coding fails.
 Section~\ref{ce} exhibits a strongly minimal theory with
a truly binary function  that still fails to (by Section~\ref{geq3})
eliminate imaginaries, as it is not commutative.

{\em Below $X$ denotes an arbitrary subset of a structure $M$
and $I$ denotes an $v$-ele\-ment independent set $\{a_1, \dots, a_v\}$ with $I \leq M$.}

We work with two groups of automorphisms; Section~\ref{decomp} treats
properties that hold of both of them so the group is denoted $G$.
Section~\ref{geq3} uses $G_I$ and  $G_{\{I\}}$ is needed in
Section~\ref{geq2}.

\begin{notation}\label{thegrps} Let $G_{\{I\}}$ be the set of automorphisms of $M$ that fix $I$
setwise and $G_I$ be the set of automorphisms of $M$ that fix $I$ pointwise.
\end{notation}

\begin{definition}\label{Ginv}
For $G$ either  $G_I$ or $G_{\{I\}}$,  $D$ is said to be $G$-invariant
if $D$ contains the $G$ orbits of each of its elements,
equivalently, $g(D) = D$ whenever $g\in G$.
\end{definition}

We introduce the (minimal) definable closure  $\dcl^*$ of a set $X$ to
distinguish points which depend on {\em all} elements of $X$.
Recall that for any first order theory $T$, if $X \subseteq M\models T$, then
$c \in \dcl(X)$  means $c$ is the unique solution of a formula with
parameters in $X$.  This implies the orbit of $c$ under $\aut_X(M)$  consists
of just  $c$ and the converse holds in any $\omega$-homogenous model. All the
models considered here are  $\omega$-homogeneous (since
$\aleph_1$-categorical \cite{BaldwinLachlansm}).

\begin{notation}\label{*cl}
 By $b \in \dcl^*(X)$ we mean $b \in
\dcl(X)$, but $b \not \in \dcl(U)$ for any proper subset $U$ of $X$ (and
analogously for $\acl^*$). Note that $\dcl^*(X)$  consists of the subset of
$\dcl(X)$ of elements not fixed by $G_T$ for any $T \subsetneq X$.
\end{notation}

The notion of symmetric definable closure, $\sdcl (I)$,   captures one
direction of finite coding.

\begin{notation}\label{defsdcl} The {\em symmetric definable closure} of $X$,  $\sdcl(X)$,
is those $a$ that are fixed\footnote{This suffices in our context since all
models are homogenous; more generally, it is easy to  give a `there exists a
formula' definition.} by every $g \in G_{\{X\}}$. $b \in \sdcl^*(X)$ means
$b\in \sdcl(X)$ but $b \not \in \sdcl(U)$ for any proper subset $U$ of $X$.
\end{notation}

Note $\sdcl^*(X) \subseteq \dcl^*(X)  \subseteq \dcl(X)$. However, $\sdcl(X)$
may not be contained in $\dcl^*(X)$.\footnote{For a simple example, consider
the theory of $(\ZZ,S,0)$. Then $dcl^*(a,b) =\emptyset$ for $a,b$ in distinct
$\ZZ$-chains but $0 \in \sdcl(\emptyset)\subseteq \sdcl(X)$, for
any $X$ with at least two elements.}




\begin{remark}\label{expreview} {\rm We will give in Example~\ref{examp2} a
theory $\hat T_\mu$ where for some $B$ with $\delta(B)=2$, there is a good
pair $A/B$ with $\mu(A/B)=2$, and $\hat T_\mu$ admits an independent
set $I = \{a,b\} \subseteq M \models \hat T_\mu$ 
with 
$\dcl^*(I) \neq \emptyset$.  Given a $v$-element independent set $I$ with
$v\ge 2$, we will show in Section~\ref{geq3} assuming $\mu( A/B)\geq 3$
(whenever $\delta(B) = 2$), that  $\dcl^*(I) = \emptyset$ and show  in
Section~\ref{geq2} that, even if some $\mu(A/B)$ might be $2$, $\sdcl^*(I)$
is empty.
}
\end{remark}


%
%

%

 \begin{definition}[non-trivial functions]\label{truen} Let $T$
be a strongly minimal theory.
\begin{enumerate}[1]
\item {\em (Essentially unary)} 
 An $\emptyset$-definable
    function
    $f(x_0 \ldots x_{n-1})$ is called {\em essentially unary} if there is
an $\emptyset$-definable function $g(u)$ such that for some $i$, for all
but a finite number of $c \in M$, and all but  a set of Morley rank $ <n$
of tuples $\bbar\in M^n$, $f(b_0 \ldots b_{i-1}, c, b_i \ldots b_{n-1}) =
g(c)$.

\item {\em (truly $n$-ary)} 
\begin{enumerate}
\item Let $\xbar =\langle x_0 \ldots x_{n-1}\rangle$:  a function
    $f(\xbar)$ {\em truly depends on $x_i$} if for any {\em independent}
    sequence $\abar$ and some (hence any) independent\footnote{This
    definition is more restrictive than the standard (e.g. \cite[p.
    35]{Gratzer}) as for our definition, in a ring the polynomial $xy+z$
    does not depend on $y$ while usually one is allowed to substitute $0$
    to witness dependence.} $\abar'$ which
disagrees  with $\abar'$ only in the $i$th place $f(\abar) \neq
    f(\abar')$.
    \item $f$ is truly $n$-ary if $f$ truly depends on all its arguments
        and $f(\abar)$ is not a component of $\abar$ for all but  a set
        of Morley rank $ <n$ of tuples $\abar \in M^n$ .
        \end{enumerate}
        \end{enumerate}
\end{definition}

\begin{lemma}\label{getun} For a strongly minimal $T$ with $\acl(\emptyset)$
infinite,
the following conditions are equivalent:
\begin{enumerate}
\item for any $n>1$ and any independent set $I= \{a_1,a_2,\ldots a_n\}$, $\dcl^*(I) =
    \emptyset$;
    \item
    every
    $\emptyset$-definable
    $n$-ary function ($n>0$) is essentially unary;
\item for each $n>1$ there is no $\emptyset$-definable truly $n$-ary
    function in any $M\models T$.
\end{enumerate}
\end{lemma}

\begin{proof} 1) implies 2). Fix $I$ as in the statement and let $f$ be an $\emptyset$-definable
$n$-ary function. Then 1) implies that for some $i$, say, 1, there is an
$\emptyset$-definable function $g(u)$, with $g(a_1) = f(\abar) =d$. Let $p^1$
denote the generic type over $\emptyset$ realized by each   $a_i$ and $p^n$
the type of the $n$-tuple. There are parameter-free formulas $\psi(u,y)$ and
$\chi(u,\vbar,w)$  ($\lg(\vbar) = n-1$) such that $p^1(u)$ entails
$\psi(u,y)$ defines $g$ (i.e. $y=g(u)$) and $p^n$ implies $\chi$ defines $f$
(i.e. $w=f(u,\vbar)$).

Now we  have $\phi(a_1,d) \wedge \chi(\abar,d)$.  But since $a_1$ is
independent from $(a_2, \dots, a_n)$, for cofinitely many $c\in M$, we have $\exists y(\phi(c,y)
\wedge \chi(c,a_2\ldots a_n ,y))$.
For each $k$ with $1\leq k < n$ let  $$A^c_{k+1} = \{x: (\exists^{\infty} x_1
\ldots \exists^{\infty} x_k) \  g(x) = f(x_1,x_2\ldots x_k, x, a_{k+2} \ldots
a_n)\}.$$ By induction, since $a_{k+1} \in A^c_{k+1}$ and  is independent
from the parameters defining $ A^c_{k+1}$, each $A^c_{k+1}$ is cofinite, so,
its complement $A_{k+1}$ is finite. Thus, the subset on which $f(\xbar) \neq
g(x_1)$ is contained in $\bigcup_{1\leq i \leq n} A_i \times M^{n-1}$, which
has Morley rank at most $n-1$. Thus, $f$ is essentially unary witnessed by
$g$.


2) implies 3): Suppose $f$ is such a definable truly
     $n$-ary function,  let $\abar$ enumerate an independent set $I$.  By 2)
      there are $i,g$ with $g(a_i) = f(\abar)$ and this holds on
      any independent sequence.
%
      For some $j\neq i$, let
     $\abar'$ be obtained from $\abar$ by replacing $a_j$ by an $a'_j$ such
     that $\abar'$ is independent.  Then $f(\abar) = g(a_i) = f(\abar')$ so
     $f$ is not truly $n$-ary since it doesn't depend on $x_i$.

3) implies 1): Suppose 1) fails. Fix the least $n\ge 2$ such that
$\dcl^*(I)\ne \emptyset$ for some independent set $I= \{a_1,a_2,\ldots
a_n\}$. Let $d\in  \dcl^*(I)$. By  the definition of $\dcl$, there exists
$\varphi(x, \ybar)$ such that $\models \exists !x \varphi(x, \abar) \land
\varphi(d, \abar)$, so $\varphi$ defines some $n$-ary function $f$.
We now show $f$ is truly $n$-ary. If $f(\abar)$ is a
component of $\abar$, then $d=a_i$ for some $i$ and $d\in \dcl(a_i)$,
contradicting $d \in dcl^*(I)$. If $f$ is not truly $n$-ary,
there exists $i$ such that for any independent sequences $\bbar$ and $\bbar'$
which disagree in the $i$th place it holds that $f(\bbar) = f(\bbar')$. We
choose $a_i'$ so that $\abar a_i'$ is independent, then $\abar_i'= (a_1,
\dots, a_{i-1}, a_i', a_{i+1}, \dots, a_n)$ is independent, too. Then
$f(\abar) = f(\abar_i')$. Let $\psi(x', \abar)$ denote $f(\abar) = f(a_1,
\dots, a_{i-1}, x', a_{i+1}, \dots, a_n)$. Since $\abar a_i'$ is independent
and $a_i'$ satisfies $\psi(x', \abar)$, this formula has cofinitely many
solutions. Then the formula
\begin{multline*}
\exists x (\exists^{\infty}x') (f(a_1, \dots, a_{i-1}, x, a_{i+1}, \dots, a_n)=f(a_1, \dots, a_{i-1}, x', a_{i+1}, \dots, a_n) \land \\
\land y = f(a_1, \dots, a_{i-1}, x, a_{i+1}, \dots, a_n))
\end{multline*}
defines $d$, so $d\in\dcl^*(I-\{a_i\})$, for a contradiction.
\end{proof}

The strongly minimal affine space $(Q, R^4)$ with $R(x, y, z, t)
= x+y = z+t$ fails weak elimination of imaginaries; weak elimination is
obtained if a point is named.

\begin{remark}\label{truenexamp} {\rm We cannot isolate `non-triviality' by simply saying there is no
     definable $n$-ary function, nor even none
   which depends on all its variables. The insight is that if $a \in
     \acl(B)$ by a formula $\phi(v,\bbar)$ which has $k$ solutions, any
     solution is in the definable closure of $B$ and the other $k-1$.
     Steiner systems with line length $k$ give a stark example. Consider
     the function of $k-1$ variables  which projects on the first entry
     unless the $k-1$ arguments are a clique (partial line) and gives the
     last element of the line in that case. This function satisfies the
     `depends' hypothesis but not the projection hypothesis (Definition~\ref{truen}.2). So, although
     $\emptyset$-definable,  it is not a  {\em truly} $(k-1)$-variable
     function.}
\end{remark}

The main aim of this paper is to study the definable closure in Hrushovski's
example in order to get to know which algebraic operations can be defined in
those examples. As a by-product, we obtain the  application in
Lemma~\ref{fckey} of our results to elimination of imaginaries.

    \begin{lemma}\label{fckey} Let $I =\{a_0,a_1, \dots, a_{v-1}\}$ be an independent set with
    $I \leq M$ and $M$ be a model of a strongly minimal theory constructed as in Section~\ref{gcon}.
    \begin{enumerate}
    \item For any finite $X$, if $\sdcl(X) = \emptyset$ then $X$ is not
        finitely coded.
    \item If $\sdcl^*(I) = \emptyset$ then $I$ is not finitely coded.
 \item   If $\dcl^*(I) = \emptyset$ then $I$ is not finitely coded and
     there is no parameter-free definable truly $n$-ary function for
     $n=|I|$.
\end{enumerate}
    \end{lemma}

    \begin{proof} 1) is immediate from Definition~\ref{deffc} and 3) follows immediately
    from 2) since $\sdcl^*(I) \subseteq \dcl^*(I)$.  2) requires some effort.
    Suppose $t \in T$ where $T$ is a finite code for $I$.
If $t \in \sdcl(I)- \sdcl^*(I)$, then either $t \in \dcl(\emptyset)$ or $t\in
\dcl(J)$ for some $\emptyset \subsetneq J\subsetneq I$. In the first
alternative, if $T$ is a finite code for $I$ so is $T-\{t\}$. And since $I$
is independent, it cannot be coded by the  empty set. So we must consider the
second case. But if $t \in \dcl(I) - \dcl^*(I)$ is in, say, $\dcl^*(J)$, a
permutation switching $a_i$ and $a_j$ for some $a_j\in J$, $a_i\in I-J$ and
fixing the other $a_k$ takes $t$ to some $t' \neq t$. Thus $t\not\in
\sdcl(I)$.
\end{proof}

We use $G_I$ to prove  hypotheses (1) and $G_{\{I\}}$ for (2) of
Theorem~\ref{noei}.  The `proof' below just indicates the organization of the
argument that follows.

\begin{theorem}\label{noei} Let  $T_\mu$ be a Hrushovski construction as in Theorem~\ref{mr} or a strongly minimal
Steiner system as in Theorem~\ref{mrs}.
\begin{enumerate}
\item If  $T_\mu$ triples ($\delta(B) =2$ and $|C|>1$ imply $\mu(C/B)\geq
    3$), then $\dcl^*(I) = \emptyset$.
    \item In any case $\sdcl^*(I) =\emptyset$.
\end{enumerate}
Consequently, $T_\mu$ does not admit elimination of imaginaries.
\end{theorem}

\begin{proof} 1)  By Lemma~\ref{fckey} and Theorem~\ref{mr}, $I$ is not finitely coded.
So by Fact~\ref{wefc}, $T$ does not admit elimination of imaginaries. And by
Lemma~\ref{getun}, $\dcl^*(I) = \emptyset$ implies there is no
$\emptyset$-definable truly $n$-ary  function. 2) Theorem~\ref{no-sym-fun}
(Theorem~\ref{Steiner-no-sym-fun} for the Steiner case)  provides $\sdcl^*(I)
= \emptyset$ without the extra (triplable) hypothesis.
\end{proof}

Importantly, if  the $A$ in a good pair $A/B$ is fixed setwise by $G$ then so
is $B$.

 \begin{observation}\label{fixC}
 Assume  $\bL_0$ consists of all structures $A$ in $\bL^*_0$ such
$\emptyset\leq A$.
 Let $\Ascr, B, C, I \subseteq M $, $M \models T_\mu$ and $G = G_{\{I\}}$ or $G_I$.
Suppose $C$ is $0$-primitive over $\Ascr$ and based on $B\subseteq \Ascr$.
If the automorphism $f\in G$ fixes $\Ascr$ setwise, and fixes $C$ setwise, then it
fixes $B$ setwise.
\end{observation}

Note that the first assumption on $L_0$ in Observation~\ref{fixC} {\em fails}
for the Steiner system case (and also for Proposition 18 of
\cite{Hrustrongmin}). We find a suitable substitute in Lemmas~\ref{primchar}
and \ref{alphpetals}.

\begin{proof} By Definition~\ref{prealgebraic}.3, $B$ is uniquely determined by  $C$ and $\Ascr$.
Thus, if $f \in G$ satisfies
 $B\neq f(B)$, we have a contradiction since $f(C)  = C$ and so some element
 of $C$ is $R$-related to an element not in $B$.
\end{proof}


\section{$G$-Decomposition}\label{decomp}
We continue with the hypotheses of Sections~\ref{HLS}.
Our  original goal was to show $\dcl^*(I) = \emptyset$ for $I =\{a_1, \dots,
a_v\}$ with $d(I) =v\ge 2$. For this we introduced the notion of a
$G_I$-decomposition to analyze the algebraic closure of a finite set. However
we needed an additional hypothesis on $\mu$ to show $\dcl^*(I) = \emptyset$.
In order to eliminate that hypothesis,  we consider decompositions with
respect to two subgroups of $\aut(M)$: $G = G_I$ or $G = G_{\{ I\}}$ that fix
$I$ pointwise or setwise, respectively. Using the  decomposition associated to group we
inductively show the appropriate definable closure is empty. We give a joint
account of the decomposition but by changing the group prove
 Theorem~\ref{mr1} 1) (for $\dcl^*$) or 2) (for $\sdcl^*$).

%

\begin{definition} \label{iclnot} Let $M$ be the generic model of $T_\mu$,
$I = \{a_1, \dots, a_v\}$ be independent with $I\le M$,  and let
$G \in \{
G_I, G_{\{ I\}}\}$. A subset $\Ascr$ is {\em $G$-normal} if it is finite,
contains $I$, $G$-invariant ($G$ fixes $\Ascr$ setwise), and is strong in
$M$.
\end{definition}

We need the following easy observation to prove Lemma~\ref{Gnorex}; finite
$G$-normal sets exist. The forward implication in Observation~\ref{acldcl}
holds for any first order theory.  As in Notation~\ref{*cl} the conditions
are equivalent here  since all models  are $\omega$-homogeneous.

\begin{observation}\label{acldcl} Let $A \subseteq M$. 1) implies 2) and 3); all
are equivalent in an $\omega$-homogenous model.
\begin{enumerate}
\item $c \in \acl(A)$
 \item The orbit of $c$ under $\aut_A(M)$ ($A$ fixed pointwise) is finite.
\item If $A$ is finite then the orbit of $c$ under $\aut_{\{A\}}(M)$ ($A$
    fixed setwise) is finite.
    \end{enumerate}
    \end{observation}

\begin{lemma}\label{Gnorex}
Each finite $U\subseteq \acl(I) \subseteq M$ is contained in a finite
$G$-normal set. If $U$ is a finite $G$-normal set, then $U\subseteq \acl(I)$.
\end{lemma}

\begin{proof} Given any finite set $U$, let $G(U) = \{g(u):u\in U, g \in G\}$,
and $\Ascr^G_U = \icl(I \cup G(U))$. Then $\Ascr^G_U$ is $G$-normal. For
this, note $G(U)$ is finite by Lemma~\ref{acldcl}. The intrinsic closure of a
set is unique, so $\Ascr_U$ is fixed setwise for either $G$. The second part
of the lemma is immediate since $\icl(X) \subseteq \acl(X)$.
\end{proof}

Note that for given $I$ and $U$ in $M$, both the set $\Ascr^G_U$ and the
height (Definition~\ref{decomp1}) of the $G$-decomposition depend on the
choice of $G$.

 We need the following result from \cite[4.2]{Verbovskiy06} to carry out the decomposition.


\begin{lemma}\label{13.4}
Suppose $A_1 \subset A_2 \subset A_3$ are such that $C_i = A_{i+1} \setminus
A_i$ is 0-primitive over $A_i$, for $i=1,2$. If $C_2$ is 0-primitive over $A_1$, then $C_1$ is
0-primitive over $A_1 \cup C_2$.
\end{lemma}
\begin{proof} Let $D\subseteq C_1$. Then
\begin{align*}
\delta (D / A_1 \cup C_2) & =  \delta (D \cup C_2 / A_1) - \delta (C_2 / A_1) =
\delta (D \cup C_2 / A_1)  \\ & =  \delta (D/A_1)
+ \delta (C_2 / A_1 \cup D) = \delta (D / A_1)
\end{align*}
The first three equalities follow easily from the definition of $\delta(X/Y)$
and the conditions of the lemma.    The last equality follows from: $0 =
\delta (C_2 / A_2) \le \delta (C_2 / A_1 \cup D) \le \delta (C_2 / A_1) = 0$.
 So $C_1$ is 0-primitive
over $A_1\cup C_2$.
\end{proof}

The next definitions and theorems provide the tools for the decompositions.
Roughly speaking, capital Roman letters (A, B, C) denote specific components
of the decomposition; Gothic letters $\Ascr$, $\Dscr$ range over initial
segments of the decomposition.  In particular, this means that each of
$\Ascr$, $\Dscr$ contains $I$ and is closed in $M$.

\begin{definition}\label{wpdef}\begin{enumerate}
\item We call a good pair  $A/B$  {\em well-placed} by $\Dscr$, if
    $B\subseteq \Dscr \leq M$ and $A$ is
    $0$-primitive over $\Dscr$, and
\item  $A/B$ is {\em well-realized} in $Y$ if  $\chi_Y(A/B) = \mu(A/B)$. If
    $Y=M$ we omit it and write simply $A/B$ is {\em well-realized}.
    \end{enumerate}
    \end{definition}

Lemma~\ref{getmax} is crucial for the general decomposition
Construction~\ref{decomp1}.

\begin{definition}\label{splitdef} Let $A$ be a subset of $\Dscr$. We say $A$
 {\em splits over} $\Dscr$ if both $ A \cap \Dscr$ and $A-\Dscr$ are non-empty.
\end{definition}


\begin{lemma}\label{getmax} Let $M$ be the countable generic model for
$T_\mu$. Suppose $A/B$ is {\em well-placed} by $\Dscr$ and $\Dscr$ is
$G$-invariant.
%
%
\begin{enumerate}
\item Then $\chi_M(A/B) = \mu(A/B)$.

\item \label{autdist}For each $i < \mu(A/B)$ there is a partial isomorphism
    $h_i$ fixing $B$ pointwise with domain $B\cup A$ and satisfying either i) $h_i(A) \cap
    \Dscr = \emptyset$ or ii) $h_i(A) \subsetneq \Dscr$. Moreover, by
    Definition~\ref{Kmu}.\ref{Kmuitem} the $h_i(A)$ are disjoint over $B$.
    In case i) there is $g \in G $ that fixes $IB$
(and indeed $\Dscr$) pointwise and takes $A$ to $h_i(A)$. That is, $G$ acts
    transitively on
the copies of $A$ that are disjoint from $\Dscr$.
\end{enumerate}
\end{lemma}

While the proof uses that $M$ is generic, the conclusion passes to any model
of $T_\mu$, because models are algebraically closed.

\begin{proof} 1) Part 1  is the translation to this notation of a result
    proved for the Hrushovski case in
\cite[2.25]{Verbovskiy}; for the Steiner case, it is \cite[5.14]{BaldwinPao}.
Since the article \cite{Verbovskiy}  is difficult to access, we repeat the proof here.

Let $\Dscr \le M$ and $A$ be a $0$-primitive extension of $\Dscr$ with the
base $B$. If $\chi_{\Dscr}(A/B) = \mu(A/B)$ we are done, so we assume that
$\chi_{\Dscr}(A/B) < \mu(A/B)$. Let $A_1$ be an isomorphic copy of $A$ over
$\Dscr$ and let $E$ be the canonical amalgam of $\Dscr \cup A$ and $\Dscr
\cup A_1$ over $\Dscr$. By \cite[Lemma 3]{Hrustrongmin}, $E \in \bL_{\mu}$
and there is an embedding $g: E\to M$ such that $g\myrestriction \Dscr\cup A
= id_{\Dscr\cup A}$ and $g(E)\le M$. Then $\chi_E(A/B)=\chi_{\Dscr}(A/B)+1$
and we proceed by induction.

2) By 1) there are partial isomorphisms $h_i$ for $i<\mu(A/B)$, fixing
         $B$,   but not
        necessarily $I$, giving structures  $A_{i} = h_i(A)$ isomorphic to
        $A$ over $B$. Note that  $A_{i}$ cannot split over $\Dscr$ and
				$I \leq M$ (the last holds because $I$ is independent).
				Moreover,
        we have $B A \approx BA_i$, $\Dscr \leq M$, and there
        are no relations between $A$ and $\Dscr - B$.  So,   if
        $A_{i}\cap \Dscr = \emptyset$  there must be no relations between
        $A_{i}$ and $\Dscr- B$.  Else, $\delta(\Dscr A_{i}) < \delta(\Dscr)$.
As $\Dscr A \approx \Dscr A_{i}$,  there is an automorphism of $M$ taking $A
$ to $A_i$ and fixing $\Dscr$ and in particular $I$.
\end{proof}

The following definition and description of the decomposition of a $G$-normal
set is intended to be evident (modulo the references). The next diagram gives
an overall view; Example~\ref{examp1} gives a closer view.

 \begin{figure}[h]\label{example}
\begin{center}

\begin{picture}(400,150)
\thicklines
\put(22,80){\oval(40,20)}
\put(52,80){\oval(100,60)}
\put(82,80){\oval(160,100)}
\put(112,80){\oval(220,140)}
\put(12,76){$\Ascr^0$}
\put(18,56){$\Ascr^1$}
\put(24,36){$\Ascr^2$}
\put(30,16){$\Ascr^3$}
\thinlines
\multiput(222,65)(30,0){6}{\framebox(30,30)}
\put(256, 53){$X_{k}$}
\put(346, 53){$X_{m}$}
\put(72,80){\oval(20,30)}
\put(67,76){$B$}
\put(232, 128){$B$ is the base of both $X_k$ and $X_{m}$}
\put(232, 108){$X_{k}\cup B \cong_B X_{m}\cup B$}
\qbezier(36,89)(47,100)(52,110)
\qbezier(36,71)(47,60)(52,50)
\qbezier(41,84)(70,86)(99,100)
\qbezier(41,76)(70,74)(99,60)
\qbezier(88,109)(99,118)(104,130)
\qbezier(88,51)(99,42)(104,30)
\qbezier(102,90)(130,94)(161,115)
\qbezier(102,70)(130,66)(161,45)
\qbezier(102,80)(130,80)(161,80)
\qbezier(150,129)(161,138)(166,150)
\qbezier(150,31)(161,22)(166,10)
\qbezier(162,100)(190,104)(222,125)
\qbezier(162,60)(190,56)(222,35)
\linethickness{0.075mm}
\multiput(252,65)(3,0){10}{\line(0,1){30}}
\multiput(342,65)(3,0){10}{\line(0,1){30}}
\end{picture}

\bigskip
\begin{picture}(400,150)
\thicklines
\put(22,80){\oval(40,20)}
\put(52,80){\oval(100,60)}
\put(82,80){\oval(160,100)}
\put(112,80){\oval(220,140)}
\put(12,76){$\Ascr^0$}
\put(18,56){$\Ascr^1$}
\put(24,36){$\Ascr^2$}
\put(30,16){$\Ascr^3$}
\thinlines
\multiput(222,65)(30,0){4}{\framebox(30,30)}
\put(72,80){\oval(20,30)}
\put(67,76){$B$}
%
\qbezier(36,89)(47,100)(52,110)
\qbezier(36,71)(47,60)(52,50)
\qbezier(41,84)(70,86)(99,100)
\qbezier(41,76)(70,74)(99,60)
\qbezier(102,80)(130,80)(161,80)
\qbezier(68,110)(79,118)(84,130)
\qbezier(68,50)(79,42)(84,30)
\qbezier(96,105)(120,112)(140,130)
\qbezier(96,55)(120,48)(140,30)
\qbezier(101,96)(130,96)(161,96)
\qbezier(101,64)(130,64)(161,64)
\qbezier(150,129)(161,138)(166,150)
\qbezier(150,31)(161,22)(166,10)
\qbezier(162,100)(190,104)(222,125)
\qbezier(162,60)(190,56)(222,35)
\put(110,86){$A^2_{j,1} = X_k$}
\put(110,70){$A^2_{j,2} = X_m$}
\linethickness{0.075mm}
\multiput(105,64)(3,0){19}{\line(0,1){32}}
\put(232, 148){Since $B\subseteq \Ascr^1$,}
\put(232, 128){$X_k$ and $X_{m}$ were placed into  $\Ascr^2$ and}
\put(232, 108){renamed: $A^2_{j,1} = X_{k}$, $A^2_{j,2} = X_{m}$}
\end{picture}
\caption{From a linear to a tree-decomposition: One Step}
\label{fig:Reinforcement}
\end{center}
\end{figure}

\begin{construction}\label{decomp1}
{\rm Let $\Ascr $ be $G$-normal.
We can linearly decompose $\Ascr$ as the union of $X_n$, $n \le r$, where
$X_0 = I$ and $X_{n+1}$ is $0$-primitive over $X_n$ and good over $Y_{n+1}
\subseteq X_n$ for $n <r$. This is a  cumulative decomposition: $X_n
\subseteq X_{n+1}$.

Since we aim to prove that $\dcl^*(I) \cap X_r = \emptyset$ ($\sdcl^*(I) \cap
X_r = \emptyset$)
 by
induction on $n$, it would be convenient to assume that $X_n$ is
$G$-invariant for each $n<r$. But it is not true. In order to reach an
induction on $G$-invariant  sets,  we create, by grouping  the images of
various partial isomorphisms of the $\hat X_{n+1} =(X_{n+1} - X_n)$ over
$Y_{n+1}$, $G$-invariant strata $\Ascr^{m+1}$ of components that are
independent over $\Ascr^m$. The new tree decomposition creates strata
$\langle \Ascr^m: m< m_0\rangle$; $m_0$ is called the {\em height} of the
decomposition.

We   define the new decomposition of $\Ascr$ into strata $\Ascr^m$  by
 inductively assigning to each $\hat X_{n+1}$
an integer $S(\hat X_{n+1},Y_{n+1})$, the {\em strata} of $\hat X_{n+1}$, the
least $m +1\leq n$ such that $Y_{n+1} \subseteq \Ascr^m$ and
renaming $\hat X_{n+1}= (X_{n+1} - X_n)$
 as an $A^{m+1}_{x,y}$ for an appropriate
$x,y$  (more detail below).  The $Y_{n+1}$ may be omitted when clear from context.
By fiat, $S(X_0, \emptyset) =0$.

\begin{enumerate}[i)]
\item $\Ascr^0$:
\begin{enumerate}
\item Let $D_0 = \Ascr  \cap \acl(\emptyset)$.
\item For $i=1, 2, \dots, v$, let $D_i = \Ascr\cap \acl(a_i)$ and let
    $\Ascr^0= D_1 \cup D_2\cup\dots\cup D_v$.

Note that $D_i \cap D_j = D_0$ for any $1\le i < j \le v$,  $\Ascr^0$ is
finite, $\Ascr^0\leq M$, and so $\delta(\Ascr^0) =v$.
Moreover, since $d(a_i) =1$ the $D_i$ are fully independent over $D_0$.
\end{enumerate}

As we continue the construction we will rearrange the components $\hat X_n$
into a quasi-order  by introducing sets $\Ascr^m$ such that each component
 in $\Ascr^m$ is based on a subset of $\Ascr^{m-1}$.
At the $n$th stage of construction, considering  $(\hat X_{n+1}, Y_ {n+1} )$,
  $\hat X_{n+1}$ is added to $\Ascr^{S(\hat X_{n+1}, Y_{n+1})}$ and given
an appropriate name as described below.	 Each $\Ascr^m$ is
 divided into $q_m$ subsets $\Ascr^m_i$, where $\Ascr^m_i$ consists of
 $\ell^m_i$, disjoint off $\Ascr^{m-1}$, sets $A^m_{i,f}$ which are
$0$-primitive over $\Ascr^{m-1}$ and
pairwise isomorphic over $\Ascr^{m-1}$, and each $A^m_{i,f}$ is based on the same set $B^m_{i} \subseteq \Ascr^{m-1}$. 
We call the $A^m_{i,f}$ {\em petals}.
Lemma~\ref{getmax}.(\ref{autdist}) ensures that $G$ acts $1$-transitively on
$\{A^m_{i,f}: 1\le f\le \ell^m_i\}$. We describe further petals of
$A^{m}_{i,1}/B^m_i$ in the next few paragraphs.

We now give a precise definition of $\Ascr^{m+1}$.   We set $\Ascr^{-1} =
\emptyset$ to allow uniform treatment for all   $m\geq 0$. Note that new
petals  may be added to $\Ascr^{m+1}$ at later stages in the construction.

    \item $\Ascr^{m+1}$: Suppose $S(X_n) = m>1$. We consider
        the good pair $\hat X_{n+1}/Y_{n+1}$ with $Y_{n+1} \subseteq X_n$.
       If $Y_{n+1} = Y_{n'}$ for some $n'\le n$ with $S(X_{n'}) =m$ then
       $Y_{n+1}$ has already been denoted $B^{m+1}_t$ for some $t \leq j$.
        If $\hat X_{n+1} \approx_{Y_{n'}} X_{n'}$, set
        $\hat X_{n+1}$ as ${A}^{m+1}_{t,k}\in \Ascr^{m+1}_t$, where $k$ is
        the least index not previously used with $t$.

If $X_n \not\approx_{Y_{n'}} X_{n'}$ and $Y_{n+1} \neq Y_{n'}$ for any
$n'<n$ with $S(X_{n'}) =m$,
  set      $\hat  X_{n+1}$ as $A^{m+1}_{u,1}$ and set $Y_{n+1}$ as
        $B^{m+1}_u$ for the next available $u$. Then $Y_{n+1} \cap
        (\Ascr^{m} - \Ascr^{m-1}) \neq \emptyset$, $\Ascr^m \leq M$,
        $\Ascr^{m} \hat X_{n+1}\leq M$.
 It is possible that $(Y_{n+1},\hat X_{n+1})  \approx (Y_{n'},\hat X_{n'})$
for some smaller $n'$.

By Lemma~\ref{getmax}.\ref{autdist}, there are partial
         isomorphisms\footnote{It is essential here that each $(X_n/Y_n)$
         is well-placed (Definition~\ref{wpdef}).} $h_i$ for $i\leq\mu(\hat
         X_{n+1}/Y_{n+1})$
           that fix $Y_n$ and the
 $h_i(\hat X_{n+1})$ are independent (and so disjoint) over $Y_n$. Note
that some of these $h_i$ may not extend to automorphisms of $M$ and if so
by Lemma~\ref{getmax}.2
they map $X_{n+1}$ into $\Ascr^m$. Suppose 
that
$\ell^{m+1}_{j+1}$ of these partial isomorphisms extend to
 automorphisms $h_i$ of $M$ that
  fix $I$ and so 
$h_i(\hat X_n) \cap \Ascr^{m} = \emptyset$ for $1
 \leq i \leq \ell^{m+1}_{j+1}$. %

We have relabeled the $h_i(\hat X_{n+1})$ as $A^{m+1}_{j+1,f}$, for $1\leq
 f \leq \ell^{m+1}_j$ and added them to $\Ascr^{m+1}_j$ forming
 $\Ascr^{m+1}_{j+1} = \bigcup_{1 \leq f \leq \ell^{m+1}_{j+1}}
 {A}^{m+1}_{j+1,f}$, which is thus $G$-invariant. Since $\Ascr$ is
 $G$-normal, each of the  $h_i(\hat X_{n+1})$ is an $\hat X_{n'}$ for some
 $n'\geq {n+1}$.

 Let $\mu^{m+1}_{j+1}$ denote $\mu(\hat X_{n+1}/Y_{n+1}) =
  \mu(A^{m+1}_{j+1}/B^{m+1}_{j+1})$.  The other $\mu^{m+1}_{j+1}-
        \ell^{m+1}_{j+1}$ images are subsets of $\Ascr^m$ and are labeled
        as $C^{m+1}_{j+1,k}$ for $ 1 \leq k \leq  \nu^{m+1}_{j+1} =
       \mu^{m+1}_{j+1} - \ell^m_{j+1}$.
\end{enumerate}

Each of the $A^{m+1}_{i,f}$ for $1\le f \le \ell^{m+1}_{j+1}$ and the
$C^{m+1}_{j,k}$ for $ 1 \leq k \leq  \nu^{m+1}_{j+1}$ is a petal.

  Note that  $\hat X_{n+1} = X_{n+1} - X_n$ is based on $Y_{n+1}$, which we
have designated\footnote{We use a single subscript because, while we are
considering several copies of the $\hat X_n$, there is a fixed base.} as
$B^{m+1}_{j+1}\subseteq \Ascr^m$; by the minimality of $m$, $B^{m+1}_{j+1}$
intersects $\Ascr^m -\Ascr^{m-1}$ non-trivially.  Thus as we construct
$\Ascr^{m+1}_j$, we are moving $\ell^{m+1}_{j+1}$ components down so they are
directly above their base. This is possible by Lemma~\ref{13.4}.
We sometimes call the $A^{m+1}_{j,i}$ which have the same base
$B^{m+1}_{j}$ a {\em cluster} $\Ascr^{m+1}_j$.


At the conclusion of the construction  for each $m<m_0$, for some $t_{m+1}<
r$, there will be $t_{m+1}$ ($t_{m+1} = \Sigma_{i<q_m} \ell^{m+1}_i$)
distinct $\hat X_{n+1}$, labeled as $A^{m+1}_{j,f}$ with $S(X_n) = m$; the
$A^{m+1}_{j,f}$ are independent over $\Ascr^m$.  Then $\Ascr^{m+1} = \Ascr^m
\cup \bigcup_{j<q_m} \Ascr^{m+1}_j$ and the union is a partition of
$\Ascr^{m+1}-\Ascr^{m}$. While $\bigcup_{i\le r} X_i$ is a chain, the
$\Ascr^m$ form a tree with the petals $A^{m+1}_{j,f}$ partitioning each
level. More locally $B^{m+1}_j \cup\bigcup_{f=1}^{\ell^{m+1}_j}
A^{m+1}_{j,f}$ looks like a flower with the base $B^{m+1}_j$ and two
collections of petals. $\Ascr^{m+1}-\Ascr^m$ is a collection of petals
$\bigcup_{ 1\le j\le q_m, 1\le f \le \ell^{m+1}_j} A^{m+1}_{j,f}$ on the stem
$\Ascr^m$. But for each $j$, the flower over $B_j$ also contains the
$C^{m+1}_{j,k} \subseteq \Ascr^m$ for $k$ and $j$ with $1\le k \le \nu^{m+1}_j$ and $1\le j\le q_m$.
Further, $\Ascr = \bigcup_{m \leq m_0} \Ascr^m$.}
\end{construction}

Note that any two petals on the same strata, say on $\Ascr^{m+1}$, are
$\delta$-independent over $\Ascr^m$ and in the case of Hrushovski's
construction are fully independent. For Steiner systems we obtain that if
these petals do not belong to the same linear cluster
(Definition~\ref{decomp2}) then they are fully independent.

\begin{remark}\label{lintree} Note that a $G$-decompostion depends on, and is determined by,
 the
original linear decomposition.
\end{remark}

The following observation is key to the proof of the ensuing Lemma~\ref{cl0}
and Lemma~\ref{small-G}.

\begin{observation}\label{obsa} \cite[Note 2.8]{Verbovskiy} {\rm We say that a
   $3$-hypergraph $A$ is {\em disconnected over $B$} if
  there is a partition of $A$  into $A_1 \cup A_2$ such that for every $a\in A_1$
  and $b\in A_2$ there is no $d \in A \cup B$ such that $R(a,b,d)$.  It is easy
  to see that if $A$ is $0$-primitive over $B$, then $A$ is connected over $B$.
  As, $\delta(A_1/A_2B) = 0$, if $A$ is disconnected over $B$; but then $A_2$ is
  $0$-primitive over $B$, contrary to the minimality of $A$.}
\end{observation}

Here are the basic properties of $\Ascr^1$ showing $\dcl^*(I) \cap \Ascr^1 =
\emptyset$; the situation is simpler than the $m>1$ case as there are no maps
of $\hat X_n$ into $\Ascr^m$ over $Y_n$ when $m=0$.

\begin{lemma}[$\Ascr^1$]\label{cl0}
Let $\Ascr$ be $G$-normal and decomposed as  $\langle \Ascr^n:n <m_0\rangle$.
Then for any $i$ and $f < \ell^1_i$, $A^1_{i,f}$, and $B^1_i\subseteq
\Ascr^0$ the following hold: $d(B^1_i)  \ge  2$,
$\Ascr^0  \leq M$, $\ell^1_i = \mu^1_i = \mu(B^1_i,A^1_{i,1})$, and
by $G$-decomposition, there is no copy of $A^1_{i,f}$ over $B^1_i$
 in $\Ascr^0$. So, no $A^1_{i,f}$ is invariant under $G$.
\end{lemma}

\begin{proof}
Note  that for each $i$, $d(B^1_i) \ge 2$; otherwise there exists $b \in B^1_i$
such that $d(b) = d(B^1_i)=1$. Since $b\in \Ascr^0$,  $b \in \acl(a_k)$ for some
$k \in \{1, 2, \dots, v\}$; this implies that $B^1_i \subseteq \acl(a_k)$ and thus
$A^1_{i,j}\subseteq \acl(a_k)$; the last assertion contradicts $A^1_{i,j}
\cap \Ascr^0 = \emptyset $. As noted in Construction~\ref{decomp1}.i), $I\subset
\Ascr^0$ and $\delta(\Ascr^0) = v$ so $d(\Ascr^0) =v$  and $\Ascr^0 \leq M$.

We show there cannot be a copy, $C^1 = C^1_{i,x}$ for some $i,x$, of
$A^1_{i,f}$ with
 base $B^1_i$
 embedded in $\Ascr^0$. Since $d(B^1_i) \ge 2$, $B^1_i$ intersects at
least two $D_k-D_0$ and $D_j-D_0$ for some $k\ne j$. Note that $C^1 \subseteq
\Ascr^0$ is not a singleton $c$, because otherwise $M\models R(c, d_k, d_j)$
for some $d_k\in B^1_i\cap D_k-D_0$ and $d_j\in B^1_i\cap D_j-D_0$,
contradicting full independence of the $D_i$'s over $D_0$. By
Lemma~\ref{primchar1} $C^1$ should intersect both $D_k-D_0$ and $D_j-D_0$. If
not, there would be an $R(c_1,c_2,d)$ with $c_1,c_2 \in D_k$ and $d\in D_j
\cap B^1_i $; this can't happen as the $D_i$ for $i>0$ are fully independent
over $D_0$ (Construction~\ref{decomp1} $\Ascr^0$.b). But then $C^1$ is
disconnected, contrary to Observation~\ref{obsa}. Thus $\ell^1_i = \mu^1_i
\geq \delta(B^1_i) \geq d(B^1_i) \ge 2$.
\end{proof}

We pause to note a distinction between the flat geometries and the locally
modular ones.  \cite{Baldwinautcat} showed that the finite dimensional models
of an $\aleph_1$-categorical theory had either countably many or
$2^{\aleph_0}$ automorphisms, with vector space-like strongly minimal sets on
the first side and algebraically closed fields on the other. We now note:

\begin{corollary}\label{manyaut} If $T_\mu$ is constructed by a Hrushovski
construction (including Steiner systems)  with a flat geometry, each finite
dimensional model $M_n$ has $2^{\aleph_0}$ automorphisms.
\end{corollary}

\begin{proof} Suppose $M$ is prime over the algebraically independent set $X$
with $n$ elements. There are countably many distinct good pairs $(A_n/X)$
(Remark~\ref{narycomgeom}); each has multiplicity at least $2$, and we can
define automorphisms of $M$ that fix or permute the realizations $A_n$ at
will to give $2^{\aleph_0}$ automorphisms.
\end{proof}

The following example shows the situation gets much more complicated with the
second strata.

\begin{example}\label{examp1} {\rm
This example  illustrates i) the shift from a chain to a strata
decomposition, ii) $\acl^*(\Ascr^0)$ may properly extend $\dcl(\Ascr^0)$ and  iii)
that some $A^2_{1,i}$
 may intersect $\Ascr^0$.
Let $M$ be any model of $\hat T_\mu$ with  $\mu(\boldsymbol{\alpha}) = 2$.
Suppose $I = \{a_1,a_2\}$ and let $R$ hold of the triples
$a_1a_2b_1,a_1a_2b_2,c_1c_2b_1,c_1c_2b_2$ and the entire six point diagram be
strong in $M$.

\begin{center}
\begin{minipage}[h]{0.45\linewidth}
\begin{center}
\begin{picture}(200,120)
\put(8, 49){$a_1$} \put(8, 69){$a_2$} \put(22, 50){\circle*{3}} \put(22,
70){\circle*{3}} \put(6, 15){\dashbox{1}(30,90)[c]{~}} \put(8, 0){$X_0$}
\put(22, 50){\line(4,-1){80}} \put(22, 50){\line(1,1){40}} \put(22,
70){\line(2,1){40}} \put(22, 70){\line(2,-1){80}} \put(22,
50){\line(0,1){20}}
\put(88, 26){$b_1$} \put(48, 92){$b_2$} \put(102, 30){\circle*{3}} \put(62,
90){\circle*{3}} \put(84, 15){\dashbox{1}(40,90)[c]{~}} \put(40,
15){\dashbox{1}(40,90)[c]{~}} \put(42, 0){$X_1$} \put(84, 0){$X_2$}
\put(142, 50){\line(-2,-1){40}} \put(142, 50){\line(-2,1){80}} \put(142,
70){\line(-4,1){80}} \put(142, 70){\line(-1,-1){40}} \put(142,
50){\line(0,1){20}}
\put(146, 49){$c_1$} \put(146, 69){$c_2$} \put(142, 50){\circle*{3}}
\put(142, 70){\circle*{3}} \put(128, 15){\dashbox{1}(30,90)[c]{~}} \put(130,
0){$X_3$}
\end{picture}

\end{center}

\captionof{figure}{Chain}\label{Ex0} 
\end{minipage}
\hfill
\begin{minipage}[h]{0.45\linewidth}
\begin{center}
\begin{picture}(120,120)
\put(8, 49){$a_1$} \put(8, 69){$a_2$} \put(22, 50){\circle*{3}} \put(22,
70){\circle*{3}} \put(6, 15){\dashbox{1}(30,90)[c]{~}} \put(8, 0){$\Ascr^0$}
\put(22, 50){\line(2,-1){40}} \put(22, 50){\line(1,1){40}} \put(22,
70){\line(2,1){40}} \put(22, 70){\line(1,-1){40}} \put(22,
50){\line(0,1){20}}
\put(48, 26){$b_1$} \put(48, 92){$b_2$} \put(62, 30){\circle*{3}} \put(62,
90){\circle*{3}} \put(40, 15){\dashbox{1}(44,44)[c]{~}} \put(40,
61){\dashbox{1}(44,44)[c]{~}} \put(42, 0){$\Ascr^1$} \put(66, 19){$A^1_{1,1}$}
\put(66, 95){$A^1_{1,2}$}
\put(102, 50){\line(-2,-1){40}} \put(102, 50){\line(-1,1){40}} \put(102,
70){\line(-2,1){40}} \put(102, 70){\line(-1,-1){40}} \put(102,
50){\line(0,1){20}}
\put(106, 49){$c_1$} \put(106, 69){$c_2$} \put(102, 50){\circle*{3}}
\put(102, 70){\circle*{3}} \put(88, 15){\dashbox{1}(30,90)[c]{~}} \put(90,
0){$\Ascr^2$} \put(100, 19){$A^2_{1,1}$}
\end{picture}
\end{center}
\captionof{figure}{Decomposition}\label{Ex1} 
\end{minipage}
\end{center}

Figure~\ref{Ex0} shows a chain decomposition; Figure~\ref{Ex1} illustrates
the downward embedding in a  strata decomposition (as both $A^1_{1,1}$ and
$A^1_{1,2}$ are based in $\Ascr^0$).

Further $\sdcl^*(\Ascr^0) \cap \Ascr^2 = \dcl^*(\Ascr^0)\cap \Ascr^2  =
\emptyset$. Theorem~\ref{no-sym-fun} implies a stronger result that
$\sdcl^*(\Ascr^0) = \emptyset$.
However,
we will show in Section~\ref{ce} that for an independent pair $I$, there may be elements in
$\dcl^*(I) - \sdcl^*(I)$. $X_1$ and $X_2$ demonstrate that there may be
components $X_n$ and $X_{n'}$, both in strata $m$, such that $(Y_n,\hat X_n)
\approx_I (Y_{n'},\hat X_{n'})$; we provide the tool to study this situation
in Definition~\ref{defJ}.

Finally, in Figure~\ref{Ex2} $A^2_{1,1} \cap \Ascr^1 =\emptyset$ while
$C^2_{1,1}= \Ascr^0 = B^1_1$ although both are based on and isomorphic over $B^2_1$,
as it shown in Figure~\ref{Ex2}.

\begin{center}
\begin{picture}(120,140)
\put(8, 49){$a_1$} \put(8, 69){$a_2$} \put(22, 50){\circle*{3}} \put(22,
70){\circle*{3}} \put(6, 15){\dashbox{1}(30,90)[c]{~}} \put(8, 0){$\Ascr^0$}
\put(18, 10){\dashbox{1}(8,100)[c]{~}} \put(18, 113){$B^1_1$}
\put(22, 50){\line(2,-1){40}} \put(22, 50){\line(1,1){40}} \put(22,
70){\line(2,1){40}} \put(22, 70){\line(1,-1){40}} \put(22,
50){\line(0,1){20}}
\put(48, 26){$b_1$} \put(48, 92){$b_2$} \put(62, 30){\circle*{3}} \put(62,
90){\circle*{3}} \put(40, 15){\dashbox{1}(44,44)[c]{~}} \put(40,
61){\dashbox{1}(44,44)[c]{~}} \put(42, 0){$\Ascr^1$} \put(66, 19){$A^1_{1,1}$}
\put(66, 95){$A^1_{1,2}$}
\put(58, 10){\dashbox{1}(8,100)[c]{~}} \put(58, 113){$B^2_1$}
\put(102, 50){\line(-2,-1){40}} \put(102, 50){\line(-1,1){40}} \put(102,
70){\line(-2,1){40}} \put(102, 70){\line(-1,-1){40}} \put(102,
50){\line(0,1){20}}
\put(106, 49){$c_1$} \put(106, 69){$c_2$} \put(102, 50){\circle*{3}}
\put(102, 70){\circle*{3}} \put(88, 15){\dashbox{1}(30,90)[c]{~}} \put(90,
0){$\Ascr^2$} \put(100, 19){$A^2_{1,1}$}
\end{picture}
\captionof{figure}{$A^2_{2,2} = \Ascr^0$}\label{Ex2} 
\end{center}
}

{\rm Suppose further that $\mu(A^2_{1,1}/B^2_{1,1})= 2$. Then this is a
$G$-decomposition of $\Ascr^0 \cup \{c_2\}$ for either $G$. This shows that
(in the presence of certain good pairs with $\mu( A/B) =2$)  we cannot avoid
$G$-invariant petals.}
\end{example}

\begin{definition}[$J^{m+1}_j$]\label{defJ}
 Let $\Ascr$ be $G$-normal and decomposed by
 $\langle \Ascr^n: n\leq m_0\rangle$.
We let $J^{m+1}_{G,j}$ consist of all indices $j'$
that $g(B^{m+1}_j) = B^{m+1}_{j'}$ for some $g\in G$. Thus we have an
equivalence relation on the $j$'s with $1\le j\le q_m$ enumerating the bases $B^{m+1}_j$; $j
\sim j'$ if $B^{m+1}_{j'} = g(B^{m+1}_j)$ for some $g\in G$.

 If $G$ is fixed we omit it in $J^{m+1}_{G,j}$ and write simply $J^{m+1}_j$.
 Note that $j \sim j'$ implies
 $J^{m+1}_{j'} =J^{m+1}_j$, $A^{m+1}_{j}\approx A^{m+1}_{j'}$, and $\mu^{m+1}_j =
\mu^{m+1}_{j'}$.
  \end{definition}

Immediately,

\begin{observation}\label{inv}
Let everything be as in Definition~\ref{defJ}. Then $B^{m+1}_j$ is
$G$-invariant if and only if $|J^{m+1}_{G,j}| =1$.

Thus, $\Ascr^m$ will consist of $\Sigma_{j } |J^m_{G,j}|\cdot \ell^m_j$
petals.
\end{observation}

We summarise in Notation~\ref{m0def}, which  also depends on the choice of
$G$. In Section~\ref{geq3} we are using $G_I$. In Section~\ref{geq2}, we
employ $G_{\{I\}}$.

\begin{notation}\label{m0def} The {\em height} of $\Ascr$ is the maximal index,
$m_0 \leq r$ of a non-empty strata.

$q_m$ denotes the number of bases $B^{m+1}_j$ that support elements of strata
$\Ascr^{m+1}$.

And, for fixed $G$, (Definition~\ref{defJ}), $|J^{m+1}_j|$ is the number of
those $B^{m+1}_{j'}$ ($j' \in J^{m+1}_j$)  that are isomorphic to $B^{m+1}_j$
over $I$ by some $g\in G$.

For each $m,j$, $\ell^{m+1}_j$ is the number of conjugates of $A^{m+1}_j$
over $I \cup B^{m+1}_j$ under $G$.
Since $\Ascr^m$ is $G$-invariant,
$\ell^{m+1}_j$ is the number of $B^{m+1}_j$-copies of
$A^{m+1}_{j,1}$ that are {\em not} embedded in $\Ascr^m$.

We denote by $\nu^{m+1}_j$ the number of $B^{m+1}_j$-copies of
$A^{m+1}_{j,1}$, labeled as $C^{m+1}_{j,q}$, that are embedded in $\Ascr^m$.

Finally, $\Ascr^m = \bigcup_{i\leq m }\Ascr^i$.
\end{notation}

\begin{lemma}\label{sum}
Let $\Ascr$ be $G$-normal and decomposed by $\langle \Ascr^n:n< m_0\rangle$.
For any positive $m\le m_0$ and $j$ it holds that $\ell^{m}_j +\nu^m_j =
\mu(A^m_{j,1}/B^m_j)$.
\end{lemma}

\begin{proof}
Fact~\ref{getmax}.\eqref{autdist} implies that $\chi_M(A^m_{j,1}/B^m_j) =
\mu(A^m_{j,1}/B^m_j)$. Let $C$ be a copy over $B^m_j$ of $A^m_{j,1}$. Since
$B^m_j \subseteq \Ascr^{m-1}\le M$, the definition of a $0$-primitive
extension implies that either $C\subseteq \Ascr^{m-1}$ or  $C\cap
\Ascr^{m-1}=\emptyset$.
\end{proof}

The following notion is central for analyzing the position of a $G$-invariant
petal in $\Ascr$. As, a $G_I$-invariant singleton is in $\dcl^*(I)$; our goal
is to show there are no such singletons.

\begin{definition}\label{gendef} We say $A^{m+1}_{j,1}$
 {\em determines} $A^m_{i,f}$
if
 $A^m_{i,f}$  is the unique petal based in $\Ascr^{m-1}$ that intersects
  $B^{m+1}_j-\Ascr^{m-1}$.
\end{definition}

$A^{m+1}_{j,1}$ is $G$-invariant and determines $A^m_{i,f}$ then $A^m_{i,f}$
is $G$-invariant, so we normally denote the determined petal by $A^m_{i,1}$.
We now see that a $G$-invariant  singleton determines a petal that contains
$B^{m+1}_{j}$. The following lemmas show that when $|A^{m+1}_{j,1}| >1$,
under appropriate inductive hypotheses, $B^{m+1}_{j}$ is `almost' contained
in $A^m_{i,1}$ (Lemma~\ref{long}.1).

\begin{lemma}\label{omni-G} Let $m\geq 1$ and $B = B^{m+1}_j$ be the base of $A^{m+1}_{j,1}$
over $\Ascr^m$. If $|A^{m+1}_{j,1}| =1$ and $A^{m+1}_{j,1}$ is $G$-invariant then

\begin{enumerate}
\item $A^{m+1}_{j,1}$ determines some $A^m_{i,f}$;

\item \label{B1-G} and if $B$ does not contain a $G$-invariant singleton,
    $B \subseteq A^m_{i,f}$.
\end{enumerate}
\end{lemma}

\begin{proof}
(1) By Observation~\ref{fixC} $B$ is $G$-invariant. Assume the contrary, that
$A^{m+1}_{j,1}=\{c\}$, but $B$ intersects at least two petals $A^m_{i,f}$ and
$A^m_{i',f'}$. Observe that  if the singleton $c$ is primitive over
$\Ascr^m$, then for some $b_1,b_2 \in \Ascr^m$,
    $((b_1,b_2), c)$ realizes the  good pair $\boldsymbol{\alpha}$.
		So, $B =\{b_1,b_2\}$ and $M\models
    R(b_1,b_2,c)$.
By construction $B \cap (\Ascr^{m} - \Ascr^{m-1}) \ne \emptyset$, so at least one of $b_1,
b_2$ is in $A^m_{i,f}$ for some $f$ and $i$, say $b_1$.

Let $C^1 = \{c_1\}$ be an
    isomorphic over $B$ copy of $A^{m+1}_{j,1}$ with $C^1 \subseteq \Ascr^m$.
    As there is no relation $R(b_1,b_2,c_1)$ with the $b_i$ in $A^m_{i,f}$
    and  $A^m_{i', f'}$ (since they are fully independent over $\Ascr^{m-1}$),
    $B - \Ascr^{m-1} \subseteq A^m_{i,f}$.
Since $B$ is
$G$-invariant,  $A^m_{i,f}$ is $G$-invariant.

(2) Suppose for contradiction $b_2 \in \Ascr^{m-1}$. Then since $B$  and
$\Ascr^{m-1}$ are each $G$-invariant both $b_1$ and $b_2$ are fixed by $G$
violating the additional
 assumption for case (2).
 \end{proof}

Now we  investigate the various images contained in $\Ascr^m$ of
$A^{m+1}_{j,1}$. To simplify notation we continue the special notations in
Lemma~\ref{omni-G} and add some more

\begin{notation}\label{pm}{\rm
We write $(A/B)$ for the good pair $(A^{m+1}_{j,1}/B^{m+1}_{j})$ and
$\mu^{m+1}_{j}$ for $\mu(A^{m+1}_{j,1}/B^{m+1}_{j})$. Let $C^d$, for  $1 \leq
d\leq \nu =\nu^{m+1}_j = \mu^{m+1}_j-1$ (since $\ell^{m+1}_j =1$) enumerate
the isomorphic images over $B =B^{m+1}_j $ of $A^{m+1}_{j,1}$ {\em that lie
in $\Ascr^m$}. Let $C^d_+ = C^d \cap \Ascr^{m-1}$, $C^d_- = C^d
-\Ascr^{m-1}$, $B_+ = B \cap \Ascr^{m-1}$, and $B_- = B -\Ascr^{m-1}$.  }
\end{notation}

 The next diagram illustrates Notation~\ref{pm}, where the petal
$A = A^{m+1}_{j,1}$ is a $G$-invariant subset of $\Ascr^{m+1}-\Ascr^m$, its
base is $B = B^{m+1}_j$, which is a subset of $\Ascr^m$;  $A$ has two copies
$C^1$ and $C^2$ over $B$, which is a subset of $\Ascr^m$. In
the diagram the $m$-th strata is $\Ascr^m-\Ascr^{m-1}$. So as not to
overlap the images of the sets $B$, $C^1$, and $C^2$,  we drew only two
petals $A^m_{i,1}$ and $A^m_{i,2}$ from $\Ascr^m$  and separated them with a
dotted line.
\begin{center}
\begin{picture}(300,160)
\put(6, 15){\dashbox{1}(140,140)[c]{~}}
\put(8, 0){$\Ascr^{m-1}$}
\put(146, 85){\dashbox{1}(70,70)[c]{~}}
\put(196, 146){$A^m_{i,2}$}
\put(146, 15){\dashbox{1}(70,70)[c]{~}}
\put(196, 19){$A^m_{i,1}$}
\put(152, 0){$\Ascr^m$}
\put(216, 15){\dashbox{1}(77,140)[c]{~}}
\put(225, 0){$\Ascr^{m+1}$}
\put(140, 85){\oval(70,30)}
\put(56,81.5){$B= B^{m+1}_j$}
\put(120,81.5){{$B_+$}}
\put(153,81.5){{$B_-$}}
\put(225,47){\line(1,0){50}}
\put(275,47){\line(0,1){70}}
\put(265,117){\line(1,0){10}}
\put(265,67){\line(0,1){50}}
\put(225,47){\line(0,1){20}}
\put(225,67){\line(1,0){40}}
\put(228, 36){$A = A^{m+1}_{j,1}$}
\put(120,35){\line(1,0){70}}
\put(190,35){\line(0,1){70}}
\put(180,105){\line(1,0){10}}
\put(180,55){\line(0,1){50}}
\put(120,35){\line(0,1){20}}
\put(120,55){\line(1,0){60}}
\put(126,42){$C^1_+$}
\put(162,42){$C^1_-$}
\put(103,42){$C^1$}
\put(120,135){\line(1,0){90}}
\put(210,65){\line(0,1){70}}
\put(200,65){\line(1,0){10}}
\put(200,65){\line(0,1){50}}
\put(120,115){\line(0,1){20}}
\put(120,115){\line(1,0){80}}
\put(126,122){$C^2_+$}
\put(162,122){$C^2_-$}
\put(103,122){$C^2$}
\end{picture}
\captionof{figure}{Illustrating Notation~\ref{pm}}\label{Notation320}
\end{center}

With this notation we continue to set the stage; now, we assume both
$|A^{m+1}_{j,1}| > 1$ and $|A^m_{i,f}| > 1$ for $i$, $f$ with
$A^m_{i,f} \cap B \ne \emptyset$.
The second assumption follows from the first, when $\mu^{m+1}_j \geq 3$, by
Lemma~\ref{Am-not-1-G} but will be an issue in Section~\ref{geq2}. Recall
Definition~\ref{Rdef} of $R(X,Y)$ and $R(X,Y,Z)$.
\begin{lemma}\label{small-G}
Assume that
  $A^{m+1}_{j,1}$ is $G$-invariant, $|A^{m+1}_{j,1}| > 1$,
	and $|A^m_{i,f}| > 1$ for each $i$, $f$
	such that $A^m_{i,f} \cap B \ne \emptyset$. Then, for any $d$ with $1\leq d \leq \nu =\nu^{m+1}_j$:
\begin{enumerate}
\item \label{Bd-G} For any $i$, $f$ such that $A^m_{i,f} \cap B \ne
    \emptyset$,   $C^d \cap A^m_{i,f}\ne \emptyset$, i.e., $C^d_-  \neq
    \emptyset$.

\item \label{smR-G}
		Using Notations~\ref{pm} and \ref{Rdef}, $R(B_-,C^d_+) = \emptyset$ and $R(B_-,
C^d_+, B_+) = \emptyset$. Thus, $\delta(B_-/B_+ \cup \bigcup_{1 \leq d \leq
\nu} C^d_+) = \delta(B_-/B_+)$.
\item\label{short} If $C^d\cap \Ascr^{m-1}= \emptyset$, that is $C^d_+ =
    \emptyset$, then there is a unique petal $A^m_{i,f}$ that contains both
    $C^d$ and $B_-$. So, $A^m_{i,f}$ is $G$-invariant.
\end{enumerate}
\end{lemma}

\begin{proof}
(1) Lemma~\ref{primchar1} implies  for any $b \in B_-$, and any $d\leq \nu$
there must be $c_1,c_2 \in C^d$ with $R(c_1,c_2,b)$.  Since $b \in B_-$ and
by construction, $b\in B_-\cap A^m_{i,f}$ for some $i$ and $f$. If both $c_1$
and $c_2 \in \Ascr^{m-1}$, then $|A^m_{i,f}| = |\{ b\}| =1$, a contradiction.
So at least one of $c_1, c_2$, say $c_1$, must be in $\Ascr^m- \Ascr^{m-1}$.
Since the petals on the same strata $m$ are freely joined over $\Ascr^{m-1}$,
$c_1$ must be in $A^m_{i,f}$.

(2) Since $|A^m_{i,f}| >1$, for any $b \in B_-$ there do not exist $x_1,
    x_2 \in \Ascr^{m-1}$ such that $M \models R (b, x_1, x_2)$.
	Hence $r(B_- , B_+ \cup \bigcup_{d=1}^{\nu} C_+^d) = r (B_- , B_+) + \sum_{d=1}^{\nu}
    r(B_- , C_+^d)$.
		The conditions $|A^{m+1}_{j,1}| > 1$ and $C^d \cong_{ B}
    A^{m+1}_{j,1}$ imply for any $c \in C^d$ there are no $b_1, b_2 \in B$
such that $M \models R(b_1, b_2, c)$.  Consequently, $r(B_- , C_+^d) = 0$.

(3) Assume that $C^d\cap \Ascr^{m-1}= \emptyset$. Assume also that $C^d\cap
A^m_{s,t} \ne \emptyset$ for some $s$ and $t$, but $C^d \not\subseteq
A^m_{s,t}$. Since petals $A^m_{u,v}$ are free over $\Ascr^{m-1}$ we obtain
that $C^d$ is disconnected over $B$, contradicting Observation~\ref{obsa}.
So, there is a unique petal $A^m_{i,f}$   that contains $C^d$. Obviously,
then $B_- \subseteq A^m_{i,f}$. The assumption $A^{m+1}_{j,1}$ is
$G$-invariant implies (Observation~\ref{fixC}) that $B$ is $G$-invariant. So,
$B_-$,
 and thus, $A^m_{i,f}$ are $G$-invariant.
\end{proof}

We prove a consequence of: $\mu(A^{m+1}_{j,1}/B^{m+1}_j) \geq 3$ and $A^{m+1}_{j,1}$ is
$G$-invariant.

\begin{lemma}
 \label{Am-not-1-G} If $\mu(A^{m+1}_{j,1}/B^{m+1}_j) \geq 3$ and $A^{m+1}_{j,1}$ is $G$-invariant,
  then $\ell^{m+1}_j +1< \mu^{m+1}_j$ and
 $|A^{m+1}_{j, 1}| > 1$  together imply
     $|A^{m}_{i, f}| > 1$ for any $i$, $f$
    such that $A^{m}_{i, f} \cap B \ne \emptyset$.

    \end{lemma}
\begin{proof} Let $b \in B\cap A^m_{i,f}$.
 For $A^m_{i,f}$ there is a unique base\footnote{ We will indicate a
 slight modification of the proof for
 the Steiner case in Lemma~\ref{nonlin}.} $B^m_i$ by Lemma~\ref{fixC}.
Since $\mu^{m+1}_i = \ell^{m+1}_i + \nu^{m+1}_i$, $\ell^{m+1}_i =1$ and
$\mu^{m+1}_j \geq 3$ imply $\ell^{m+1}_j +1 < \mu^{m+1}_j$.
     By the same observation as in Lemma~\ref{omni-G}.2, if  $|A^m_{i,f}| =1$,
      then  $A^m_{i,f} =\{b\}$ and $B^m_i$  is a pair $(c_1, c_2)\in
    \Ascr^{m-1}$ that satisfy $R(c_1,c_2,b)$.
On the other hand, $b$ satisfies $R(\alpha_1, \alpha_2, b)$ for some
$\alpha_1$, $\alpha_2 \in A^{m+1}_{j,1}$ by Lemma~\ref{primchar1}.
But, $  |A^m_{i,f}| =1$ implies there is no pair $x,y$ from $\Ascr^m-
\Ascr^{m-1}$ satisfying $R(x,y,b)$. Since $\ell^{m+1}_j +1 <
\mu^{m+1}_j$ and $\mu^{m+1}_j\geq 3$  there must be at least two  disjoint embeddings
of $A^{m+1}_{j,1}$ in $\Ascr^m$, this implies that some $d  \in \Ascr^{m-1}
-B^m_i$ is in relation with $b$; this contradicts that $A^m_{i,f}$ is related
only to elements of the doubleton $B^m_i$.
\end{proof}

 The hypothesis $C^d_+ \neq 0$ of  the next lemma is verified,
 using the inductive hypotheses,
in the proof of Lemma~\ref{wrapitup} before Lemma~\ref{long} is applied.

\begin{lemma}\label{long} Let $m\geq 1$. Assume that $\delta(B)\ge 2$, $A^{m+1}_{j,1}$ is $G$-invariant,
$|A^{m+1}_{j,1}| > 1$,
	and $|A^m_{i,f}| > 1$ for each $i$, $f$
	such that $A^m_{i,f} \cap B \ne \emptyset$. 
Further, assume that $C^d_+\ne\emptyset$ for each $d$.
Then
\begin{enumerate}[A)]
\item If $\mu^{m+1}_j \ge 3$,   then $A^{m+1}_{j,1}$ determines an
    $A^m_{i,f}$
and $\delta(B_+) \le 1$.
\item If $\mu^{m+1}_j = 2$, then $2 \geq \sum_{i\in \bigI}\ell^m_i +
    \delta(B_+)$,   where $\bigI = \{ i :(\exists t) A^m_{i,t}\cap B \ne
    \emptyset \}$,
    and thus $\delta(B_+)\le 1$.
\end{enumerate}
\end{lemma}	

\begin{proof} Most of the proof is the same for both {\em A)} and {\em B)}; we split
near the end. Let $\bigI = \{ i : A^m_{i,t}\cap B \ne \emptyset$ for some
$t\}$.
By Observation~\ref{fixC} $B$ is
$G$-invariant, so for each $i\in \bigI$ if $A^m_{i,t}\cap B \ne \emptyset$ for some $t$, then it
holds for any $t$, because $G$ acts transitively on the set of
 petals $\{ A^m_{i, t} : 1\le t\le \ell^m_i\}$, as, by construction,
they have the same base $B^m_i$ and they are on the same strata.

Let $\nu = \mu(B^{m+1}_{j},A^{m+1}_{j,1})-1$.
The conditions $|A^{m+1}_{j,1}|>1$ and $|A^{m}_{i,f}|> 1$ for each $i$ and
$f$ such that $A^m_{i,f} \cap B^{m+1}_{j} \ne \emptyset$ imply that we may
apply Lemma~\ref{small-G}\@.\ref{smR-G}  below.

First, we show $\delta(B_-/B_+) \geq \nu \cdot \sum_{i\in \bigI}\ell^m_i$.
For this, using notation~\ref{pm}, by Lemma~\ref{small-G}.\@\ref{Bd-G} and
the last hypothesis, both $C_-^d$ and $C_+^d$ are nonempty.
Since, we noted that  for each $i\in \bigI$ and each $t$ with $1 \le t\le
\ell^m_i$,   $A^m_{i,t}\cap B \ne \emptyset$,    by
Lemma~\ref{small-G}.\@\ref{Bd-G} again, $C^d_-\cap  A^m_{i,t} \ne\emptyset$.
So, invoking 	the definition of good pair we have
 $\delta( C^d_-\cap A^m_{i,f}/B \cup C^d_+) < 0$.
Indeed, $0> \delta( C\cap A^m_{i,f}/B \cup (C-A^m_{i,f}) = \delta( C^d_-\cap
A^m_{i,f}/B \cup C^d_+)$ because all petals on the $m$-strata are
$\delta$-independent over $\Ascr^{m-1}$.
 Consequently, by submodularity of $\delta$
for any  $i\in \bigI$ and
 $f = 1, \dots, \ell^m_i$,
taking into account for the first equality that any two petals on the same strata,
 namely on $\Ascr^{m}$, are
$\delta$-independent over $\Ascr^{m-1}$,
we can conclude:

\begin{equation}\label{eq1}
\delta(C^d_-/B \cup C^d_+) = \sum_{i \in \bigI}\sum_{1 \leq f \leq \ell^m_i}\delta(
C^d_-
\cap A^m_{i,f}/B \cup C^d_+) \le -\sum_{i\in \bigI}\ell^m_i.\end{equation}
Since $\Ascr^{m-1} \leq M$ and by monotonicity

\begin{equation}\label{eq2}
0 \leq \delta(\hspace{-5pt}\bigcup_{1 \leq d \leq \nu}\hspace{-5pt}
C^d_-
\cup B_-/\Ascr^{m-1}) \leq
\delta(\hspace{-5pt}\bigcup_{1 \leq d \leq \nu}\hspace{-5pt} C^d_-
 \cup B_-/\hspace{-5pt}\bigcup_{1 \leq d \leq \nu}
\hspace{-5pt} C^d_+ \cup B_+ ).\end{equation}
But by the definition of $\delta(A/C)$, we can rewrite the last term to obtain

\begin{equation} \label{eq3}
0 \leq \delta(\bigcup_{1 \leq d \leq \nu}\hspace{-5pt} C^d_- / B \cup
\hspace{-5pt}\bigcup_{1 \leq d \leq \nu}\hspace{-5pt} C^d_+ )
+ \delta(B_-/(B_+ \cup\hspace{-5pt} \bigcup_{1 \leq d \leq \nu}\hspace{-5pt} C^d_+)).\end{equation}
The last term of the right hand side of equation~\ref{eq3} equals $\delta(B_-/B_+)$ by
  Lemma~\ref{small-G}\@.\ref{smR-G}. And, the first term satisfies

\begin{equation} \label{eq4}
\delta(\bigcup_{1 \leq d \leq \nu}\hspace{-5pt} C^d_- / B \cup
\hspace{-5pt}\bigcup_{1 \leq d \leq \nu}\hspace{-5pt} C^d_+)  \le
\sum_{i\in \bigI}\sum_{\substack{1 \leq f \leq \ell^m_i\\ 1 \leq d \leq \nu}}
\delta(    (C^d_- \cap A^m_{i,f})/B \cup C^d_+) \leq -\nu \sum_{i\in \bigI}\ell^m_i
\end{equation}
since each $C^d_-$ contributes at most $-\sum_{i\in \bigI}\ell^m_i$ for each
$d = 1, \dots, \nu$ by equation~\eqref{eq1}.
Substituting  our evaluations of the two terms on the right hand side of
  equation~\eqref{eq3} (one from Lemma~\ref{small-G}\@.\ref{smR-G}) and transposing, we have
\begin{equation}\label{eq5}
\delta(B_-/B_+) \geq \nu \sum_{i\in \bigI}\ell^m_i.
\end{equation}

Now, $B = B_+ \cup B_-$ implies $\delta(B) = \delta(B_-/B_+) + \delta(B_+)$.
So
\begin{equation}\label{eq6}
\delta(B) \geq \nu \sum_{i\in \bigI}\ell^m_i + \delta(B_+).
\end{equation}

A) Assume that $\mu^{m+1}_j\ge 3$;
then $\nu \geq 2$. If $\delta(B) = 2$,
then $\nu =2$, $\bigI = \{ i \}$ for some $i$, $\ell^m_i = 1$, and $A^m_{i,1}$ is $G$-invariant.

If $\delta(B) \geq 3$, divide equation~\eqref{eq6} by $\delta(B)$. Then,    since
 $\nu \geq \delta(B) -1$, substituting in equation~\ref{eq6}, we obtain

\begin{equation}\label{eq7}
1 =\frac{\delta(B)}{\delta(B)} \geq \frac{(\delta(B) -1)\sum_{i\in \bigI}\ell^m_i}{\delta(B)}
 +\frac{\delta(B_+)}{\delta(B)}\end{equation}
Since $\delta(B) \geq 3$, $\frac{\delta(B) -1}{\delta(B)} \geq \frac{2}{3}$ so

\begin{equation}\label{eq8}
1 \geq
\frac{2 \sum_{i\in \bigI}\ell^m_i}{3}
 +\frac{\delta(B_+)}{\delta(B)}\end{equation}
 Equation~\eqref{eq8} implies $\bigI = \{ i \}$ for some $i$,
$\ell^m_i =1$, and $A^m_{i,1}$ is $G$-invariant.
Equation~\eqref{eq6} implies
$$\delta(B) \geq \nu + \delta(B_+)\ge (\delta(B)-1) + \delta(B_+).$$
So, $\delta(B_+)\le 1$.

B) Now we assume that $\mu^{m+1}_j =2$. Then, invoking the first hypothesis, $\delta(B) = 2$.
In this case, since $\nu =1$, equation~\eqref{eq6} implies $2 \geq \sum_{i\in \bigI}\ell^m_i + \delta(B_+)$
and so $\delta(B_+)\leq 1$.
\end{proof}

\section{$\dcl$ in Hrushovski's first example $\hat T_\mu$}
\label{main}
    \numberwithin{theorem}{subsection}
In this section we describe $\dcl^*(I)$ in the
main example $\hat T_{\mu}$ (Definition~\ref{defT}) in \cite{Hrustrongmin}.
We find that the values of $\mu$ for good pairs with $\delta(B)=2$ distinguish whether
 $\dcl^*(I)$ may be empty.
Subsection~\ref{geq3} deals with $\dcl^*$ under a stronger hypothesis on $\mu$ to conclude
$\dcl^*(I)=\emptyset$. Subsection~\ref{ce} provides an example that when $\mu(B,C)
= 2$ for certain good pairs, there is a theory where $\dcl^*(I) \neq
\emptyset$.  However, Subsection~\ref{geq2} $\sdcl^*$ shows that such examples still satisfy
$\sdcl^*(I)=\emptyset$ and fail elimination of imaginaries. That proof uses a deeper study of
flowers and bouquets from Subsection~\ref{bouflow}.

The very raw idea is that if some petal $A$ on the $(i+1)$-th strata is $G$-invariant,
and $\mu(A,B) \geq 3$
then it determines a petal on the $i$-th strata, which is also $G$-invariant, for each positive $i$;
but Lemma~\ref{cl0} implies that no petal on the 1st strata is $G$-invariant,
for a contradiction. The hypothesis that $\mu(B,C) \ge 3$ for any good pair $(B,C)$ with
$\delta(B) = 2$  implies that this idea works and $\dcl^*(I) = \emptyset$.
However in the general case $\dcl^*(I)$ may not be empty.
We  consider in Section~\ref{geq2} a more complicated construction in order to prove that
$\sdcl^*(I) = \emptyset$.

Note, however that the family of
theories described in Proposition 18 of \cite{Hrustrongmin} (Steiner quasigroups)
 as well as
the Steiner {\em triple} system of \cite{BaldwinPao}  have a truly binary function
defined by $R$. The definition of $\bK_0$
  from $\delta$ in Hrushovki's Proposition 18
is non-standard; in the linear space case, $\mu(\boldsymbol{\alpha})= 1$,
for $\alpha$ the good pair of a line (Definition~\ref{linelength}). Section~\ref{Steiner} adapts
our main results for Steiner systems.

\subsection{$G_I$: No truly $n$-ary definable functions}\label{geq3}
\setcounter{theorem}{0}
 We
slightly vary Hrushovski's original example by adding a further
adequacy requirement (Definition~\ref{Kmu}.1).

\begin{definition}\label{defade} We say that a function $\mu$ (or theory $T_\mu$) bounding good
 pairs   {\em triples} if
for all good pairs $(A/B)$ with $|A|>1$, $\mu(A/B) \geq \delta(B)$:
$$\delta(B) =2 \Rightarrow \mu(A/B) \ge 3.$$
\end{definition}

In this section, using this triples condition, we $G$-decompose a finite set
using $G_I$ (fixing $I$ pointwise) and show $\hat T_\mu$ is essentially unary and so fails to eliminate
imaginaries.
We give a more refined argument using $G_{\{I\}}$ in Section~\ref{geq2} showing that even
with truly $n$-ary function (i.e. dropping the `triples' hypothesis), $\hat T_\mu$
 must fail to eliminate imaginaries.

\begin{theorem}\label{mr}
Assume  that $\hat T_\mu$ triples.
Let $I$ be a finite independent set that contains at least 2 elements.
Fix a $G$-normal $\Ascr \leq M \models \hat T_\mu$ with height $m_0$.
For every $m \leq m_0$, $\Ascr^m \cap \dcl^*(I) = \emptyset$.

 Thus,
 $\dcl^*(I) \cap \Ascr = \emptyset$; so there is no truly $n$-ary $\emptyset$-definable
 function (Definition~\ref{truen}) for each $n\ge 2$  and
  $\hat T_\mu$ does not admit elimination of imaginaries.
As a corollary, we obtain that $\dcl(J) = \bigcup_{a\in J}\dcl(a)$ for any
  independent set $J\subseteq M$.
\end{theorem}

By the definition of $\Ascr^0$, $\Ascr^0 \cap \dcl^*(I) = \emptyset$.
It suffices to show by induction on $m \geq 1$ that for each $e\in \Ascr^m$,
$d(G_I(e))\geq 2$. As, if $G_I(e) =\{e\}$ then $d(G_I(e)) = 1$. But we must
begin with $m=1$ since elements  $e \in \Ascr^0$ may have $d(G_I(e)) = 1$.
We obtain the conclusion by proving Lemma~\ref{indprop} by double induction. Note that
the truth  of ${\rm dim}_m$ and ${\rm moves}_m$ each depend on the choice of $G$ as $G_I$.
Once we have this decomposition, satisfying moves for $G_I$-normal $\Ascr\leq M$ of
 any height, we can conclude
$\dcl^*_M(I) = \emptyset$.

\begin{lemma}\label{indprop} Assume
 that $\hat T_\mu$ triples.
For $m\geq 1$,
\begin{enumerate}
\item ${\rm dim}_m$: $d(E) \ge 2$ for any $G_I$-invariant set $E\subseteq
    \Ascr^{m}$, which is not a subset of $\Ascr^0$.
  \item ${\rm moves}_{m}$: No $A^{m}_{f,k}$ is $G_I$-invariant.
\end{enumerate}
\end{lemma}

The remainder of this section is devoted to the proof of Lemma~\ref{indprop}.
If $m_0$ is the height of $\Ascr$, then ${\rm dim}_{m_0}$ gives
Theorem~\ref{mr}.

For each  $m < m_0$, each $\Ascr^{m+1}$, and each $j$, $k$ such that $A^{m+1}_{j,
k} \subseteq \Ascr^{m+1}$, Observation~\ref{fixC} implies that if $A^{m+1}_{j,
k}$ is $G_I$-invariant, then $B^{m+1}_j$ is $G_I$-invariant.  However the
converse is false. The main part of the construction in Section~\ref{decomp}
was to describe the family of $G_I$-conjugates $A^{m+1}_{j,k}$ over
$B^{m+1}_j$ of $A^{m+1}_{j,1}$. We now take into account that the $B^{m+1}_j$
need not be $G_I$-invariant. 

How do we use a joint induction?
The $A^{m+1}_{j,i}$ are disjoint.
If either $|J^{m+1}_j| >1$ (Definition~\ref{defJ}) or $\ell^m_j >1$, $A^{m+1}_{j,i}$ moves
 and so
  no element of $A^{m+1}_{j,i}$ is  definable over $I$.
But, in Section~\ref{ce}, we show that when some $\mu^{m+1}_j =2$, it may be that
$|J^{m+1}_j| =1.$ So, in this section we add an hypothesis implying each
relevant $\mu^{m+1}_j \geq 3$.
In order to prove by induction on $m$ that no $A^m_{j,i}$ is $G_I$-invariant
 (Conclusion~\ref{wrapitup}), we need the
dual hypothesis  $dim_m$.
So, the simultaneous induction is organized as follows:
\begin{eqnarray*}
\mathrm{dim}_m \land \mathrm{moves}_{m+1} & \Rightarrow & \mathrm{dim}_{m+1}~~~(\mbox{Lemma}~\ref{petals}) \\
\mathrm{dim}_m \land \mathrm{moves}_{m\hphantom{+1}} & \Rightarrow & \mathrm{moves}_{m+1}~~~(\mbox{Conclusion}~\ref{wrapitup})
\end{eqnarray*}
In the following Corollary~\ref{cl1}, we slightly modify the
proof of Lemma~\ref{petals} to ground the induction by
showing ${\rm dim}_1$ and ${\rm moves}_1$.

We use without further notice the fact that for any $A \subseteq B$,
 $d(A) \leq d(B)$, e.g. $d(X) \leq d(G_I(X))$.
 Note that Construction~\ref{decomp1} shows that any base $B$ arising in
 the construction of $\Ascr^{m+1}$
 satisfies $1 \leq d(B) \leq v$, where $v = |I|$.

Lemma~\ref{petals} is formulated for $G$; it is applied for $G_I$ in
Section~\ref{geq3} while in Section~\ref{geq2}
 we use $\sdim_m$ instead of $\dim^G_m$ with $G = G_{\{I\}}$. The `moves hypothesis'
 in Lemma~\ref{petals} will follow from the inductive hypothesis in the main proof.

\begin{lemma} \label{petals}
 Fix $m$ with $1 \leq m \leq m_0$. If\/\footnote{We do not use the adequacy hypothesis
  (Definition~\ref{defade})  in proving this lemma.}
 $\dim^G_m$ and $G$ moves
 $A^{m+1}_{j,k}$  then $\dim^G_{m+1}$. That is, for
  each $e\in A^{m+1}_{j,k}$, $d(G(e)) \ge 2$.
\end{lemma}

\begin{proof} Fix $m \leq m_0$, $j< q_m$, $k$ and $e$ with $e
\in A^{m+1}_{j,k}$. We show $ d(G(e)) \ge 2$. Let $E = G(e)$.
Clearly $ d(E) \le v$, since $e\in \acl(I)$.
By Lemma~\ref{getmax}.\ref{autdist},  $E\cap A^{m+1}_{j,k} \ne \emptyset$ for each
$k \le \ell^{m+1}_{j}$. And for each $j'\in J^{m+1}_j$ the map taking
$B^{m+1}_j$ to  $B^{m+1}_{j'}$ and   Construction~\ref{decomp1}, guarantee each  $E\cap A^{m+1}_{j',k'} \ne
\emptyset$ for $k' \le \ell^{m+1}_{j'} = \ell^{m+1}_{j}$.

Note that $\icl(E) \subseteq \Ascr^{m+1}$, because $\Ascr^{m+1}\le M$. The proof now
breaks into three cases.

(1) If all the bases $B^{m+1}_{j'}$ are subsets of $\icl(E)$ (for $j' \in
    J^{m+1}_j$), the hypothesis $\dim_m$ and the monotonicity of $d$ implies
    $$ 2 \le d(\bigcup_{g \in G_I}g(B^{m+1}_{j})) \le
    d(\icl(E)) \le v$$
    and we finish.

(2) Suppose no $B^{m+1}_{j'}$ (with $j' \in J^{m+1}_j$) is a subset of $\icl(E)$.
 For each $j' \in
    J^{m+1}_j$ there is a $t'$ such that $A^{m+1}_{j',t'} \cap \icl(E) \neq
    \emptyset$. And since $B^{m+1}_{j'}\nsubseteqq \icl(E)$,
    Definition~\ref{prealgebraic}.3
    of base\footnote{That is, the base $X$ of
    $C/A$ is the least $X$ such that $\delta(A'/X) = 0$ for every $A'\subseteq A$.} implies
    $\delta(\icl(E)\cap A^{m+1}_{j',t}/\icl(E)\cap \Ascr^m) \ge  1$. Thus,

\begin{multline*}
v \geq d(E) = \delta(\icl(E)) = \delta((\icl(E)- \Ascr^m)/ \icl(E)\cap \Ascr^{m})
 + \delta(\icl(E)\cap \Ascr^{m}) \ge \\ \ge    \delta((\icl(E) - \Ascr^m) / \icl(E)\cap \Ascr^{m})
\ge\sum_{j'\in J^{m+1}_j}\sum_{t=1}^{\ell^{m+1}_{j'}}
\delta(\icl(E)\cap A^{m+1}_{j',t}/\icl(E)\cap \Ascr^m) \ge  2.
\end{multline*}

The double summation is at least $2$ because either $|J^{m+1}_{j}| \geq 2$
and there are $2$ non-zero outer summands or $J^{m+1}_j = \{j\}$ and since $G_I$
moves $A^{m+1}_{j,1}$,
$$\sum_{t=1}^{\ell^{m+1}_{j}}\delta(\icl(E)\cap
A^{m+1}_{j,t}/\icl(E)\cap \Ascr^m) \ge  2.$$

(3) Suppose one of the $B^{m+1}_{j'}$ (with $j' \in J^{m+1}_j$) is a subset of $\icl(E)$ and
    another one $B^{m+1}_{j''}$ is not. Clearly, $$d(E) =\delta(\icl(E)) = \delta(\icl(E)\cap
    \Ascr^m) + \delta(\icl(E)-\Ascr^m / \icl(E)\cap \Ascr^m).$$ Say $B^{m+1}_{j_1}
    \subseteq \icl(E)$ and $B^{m+1}_{j_2} \not\subseteq \icl(E)$. The first
summand is at least 1, because $B^{m+1}_{j_1}$ is a subset of $\icl(E)\cap
    \Ascr^m$, but is not a subset of $\acl(\emptyset)$; as, otherwise
    $A^{m+1}_{j_1,1}$ is a subset of $\acl(\emptyset)$. The second summand
    is also at least one, because
$$
\delta(\icl(E)-\Ascr^m / \icl(E)\cap \Ascr^m) = \sum_{j'\in
J^{m+1}_{j_2}}\sum_{t=1}^{\ell^{m+1}_{j'}}\delta(\icl(E)\cap A^{m+1}_{j',t}/\icl(E)\cap
\Ascr^m)$$
and for some
$t \leq \ell^{m+1}_{j_2}$, $\delta(\icl(E)\cap A^{m+1}_{j_2,t}/\icl(E)\cap \Ascr^m) \ge 1$, because
$B^{m+1}_{j_2} \not\subseteq \icl(E)$.
\end{proof}

We establish the $m=1$ step of
Lemma~\ref{indprop} by emulating the proof of Lemma~\ref{petals}; we can't really apply the
result as $\dim_0$ was not defined.

\begin{corollary}\label{cl1} Both ${\rm dim_1}$ and ${\rm moves}_1$ hold.
More precisely, for any $e\in \Ascr^1-\Ascr^0$, $  d(G_I(e)) \ge 2$.
\end{corollary}

\begin{proof} By Lemma~\ref{cl0}, we have ${\rm moves}_1$ and for each $j$,
$d(B^{1}_j)\ge 2$.  Now follow the proof of Lemma~\ref{petals}, noting that it only uses that
$d(B^{1}_j)\ge 2$.
\end{proof}

\begin{lemma} \label{omni}
If ${\rm moves}_{m}$, ${\rm dim}_m$ and
  $A^{m+1}_{j,1}$ is $G_I$-invariant, then
\begin{enumerate}
\item \label{deltaC} For $B = B^{m+1}_j$, $ d(B) \ge 2$.

\item \label{B1}  $|A^{m+1}_{j, 1}| > 1$.

\item \label{Am-not-1} If in addition, $\mu^{m+1}_j \geq 3$, then $|A^{m}_{i,
    f}| > 1$ for any $i$, $f$ such that $A^{m}_{i, f} \cap B \ne
    \emptyset$.
 \end{enumerate}
    \end{lemma}

\begin{proof} By Lemma~\ref{cl1}, we may assume $m \geq 1$.
(1) By Observation~\ref{fixC}, $A^{m+1}_{j,1}$ is $G_I$-invariant  implies
    $B^{m+1}_{j}$  is $G_I$-invariant.
		Then, since $B \nsubseteq \Ascr^0$, ${\rm dim}_m$ implies that
		$d(B) \ge 2$.

(2)  By 1) we may apply  Lemma~\ref{omni-G}.\ref{B1-G} to conclude that
 if $A^{m+1}_{j,1}=\{c\}$, then $A^m_{i,1}$ is $G_I$-invariant for
  some $i$, contradicting $\mathrm{moves}_m$.

(3) Directly follows from Lemma~\ref{Am-not-1-G}.
\end{proof}
With the next result we can  complete the induction. The hypothesis that each $\mu^{m+1}_j \geq 3$ is
essential for the induction.

\begin{conclusion}\label{wrapitup} Under the hypotheses of Theorem~\ref{mr},
 $\mathrm{move}_{m}$ and $\mathrm{dim}_m$ imply
$\mathrm{move}_{m+1}$.
\end{conclusion}
 \begin{proof}  Assume for contradiction that $A^{m+1}_{j,1}$ is $G_I$-invariant.
 The assumption $\mathrm{dim}_m$ and Lemma~\ref{omni}   imply  both $d(B) \ge 2$,
so $\delta(B) \geq 2$, and that
 the hypotheses of Lemma~\ref{small-G}
  hold. 
And Lemma~\ref{small-G} gives the remaining  hypotheses of Lemma~\ref{long}.
 Indeed, if $C^d_+ = \emptyset$ for some $d$,
we obtain that $A^m_{i,f}$ is $G_I$-invariant for some $i,f$ by Lemma~\ref{small-G}.\ref{short};
that contradicts $\mathrm{moves}_m$.
Now, by Lemma~\ref{long} we obtain that $A^m_{i,1}$ is $G_I$-invariant for some $i$;
that contradicts $\mathrm{moves}_m$.
\end{proof}

\medskip

Completing this induction gives Lemma~\ref{indprop} and
so Theorem~\ref{mr}, asserting there are no $\emptyset$-definable truly
$n$-ary functions.
Now we generalize the result by allowing parameters (Corollary~\ref{nodefwparjb}).

\begin{lemma}\label{unique-tree-decompjb}
Let $I$ be a finite independent set and $J\subset I$.
Let a tuple $\cbar \in \acl(J)$ and $d\in \acl(I)$.
Let $\Ascr_I$ be the $G_I$-normal closure of $I \cup \{d, \cbar\}$,
and $\Ascr_J$ the $G_J$-normal closure of $J \cup \{\cbar\}$.
Then there exists a   tree decomposition $\Tscr_I$ of
 $\Ascr_I$ such that the collection
  $\Tscr_J$ of petals
 $A^m_{f,j}$
that are based in $\Ascr_J$ are a
downwardly closed subset of $\Tscr_I$, whose union is $\Ascr_J$. (By {\em downwardly closed}
we mean that if $A^m_{f,j}\subseteq \Ascr_J$
then $B^j_f \subseteq \Ascr_J$.) 
\end{lemma}

\begin{proof} We note $G_{J}(\cbar) = G_I(\cbar)$ is the finite collection of automorphic images
 of $\cbar$ over $J$, as for any automorphism $\rho$ fixing $J$ pointwise $\rho\restriction\Ascr_J$ can be
 extended to one fixing $I$.
By this equality and by monotonicity of the intrinsic closure we obtain
$$\Ascr_J = \icl (J\cup G_{J}(\cbar)) = \icl (J\cup G_{I}(\cbar)) \subseteq \icl(I\cup G_I(d, \cbar)) = \Ascr_I$$
Thus, $\Ascr_J\subseteq\Ascr_I$.
Let $L_J = \langle X_i : i \le r_J\rangle$, where $X_0 \le X_1 \le
 \dots \le X_{r_J}$, be a linear decomposition of $\Ascr_J$;
that is, $X_{i+1}$ is a $0$-primitive extension of $X_i$ for each $i<r_j$
 and $X_{r_J} =\Ascr_J$.
Since $I-J$ is independent over $J$, $I-J$ is  independent over $\Ascr_J$, moreover
$I-J$ and $\Ascr_J$ are fully independent.
Then
$$I = X_0' \le X_1 \cup (I-J) \le \dots \le X_{r_J} \cup (I-J)$$
is a linear chain of $0$-primitive extensions and for each $i\leq r_J$
 the base of $X_{i+1}$ is a subset of
$\Ascr_J$.
Thus, it can be considered as an initial segment of a linear decomposition $L_I$
of $\Ascr_I$ by Lemma~\ref{13.4}.
Let  $\Tscr_K$ be the tree decomposition of $\Ascr_K$ determined by $L_K$ for $K \in \{ I, J\}$.
We now show
that if a petal in $\Ascr_I$ intersects $\Ascr_J$,
it is one of $X_{i+1}-X_i$ and a subset of $\Ascr_J$.
Clearly, $\{\acl(a_i) \cap \Ascr_I: a_i \in J\} \subseteq \{\acl(a_i)
 \cap \Ascr_I: a_i \in I\}$ and $\{\acl(a_i) \cap \Ascr_I: a_i \in J\} \subseteq \Ascr_J^0$
  by Construction~\ref{decomp1}.
So we only have to show the result for petals of $\Tscr_I$ of the form $A^{m+1}_{f,j}$.
Note that any $A^{m+1}_{f,j}$ that nontrivially intersects $\Ascr_J$ is contained in
 $\Ascr_J$,
since $\Ascr_I^m \cup \Ascr_J \leq M$
(We start with closed subsets $I$ and $J$ and obtain $\Ascr_I^m \cup \Ascr_J$
by add a sequence of $0$-primitive extensions)
and $A^{m+1}_{f,j}$ is $0$-primitive over $\Ascr_I^m$.

We need to show $A^{m+1}_{f,j} \in \Tscr_J$.

{\bf Level: {\boldmath   1}} If $m=0$ and
$B^{m+1}_f\nsubseteq \bigcup_{a\in J}\acl(a)$,
 then $A^{m+1}_{f,j} \nsubseteq \Ascr_J$ by Construction~\ref{decomp1}.
So, $A^{m+1}_{f,j} \subseteq \Ascr_J^1$
implies $B^{m+1}_f  \subseteq \Ascr_J$.

{\bf Level:} {\boldmath{$m+1$}} Since $A^{m+1}_{f,j} \subseteq \Ascr_J$ and is some $\hat X_i$ of
the given initial segment of the linear decomposition, $B^{m+1}_f \subseteq \Ascr_J$; so
$A^{m+1}_{f,j} \in \Tscr_J$.

By induction we have Lemma~\ref{unique-tree-decompjb}.
Since $\Ascr_J$  is $G_J$-normal, $\bigcup \Tscr_J = \Ascr_J$.
\end{proof}

We have the following immediate corollary.
 Let $r$ be the height (i.e. the largest index
$k$ of an $A^k_{f,i}$ with $A^k_{f,i}\subseteq \Ascr_J$) of $\Ascr_J$.

\begin{lemma}\label{imcor}
For any $m \leq r$, and any
 $A^{m+s+1}_{p,q}\subseteq
\Ascr_J - \Ascr^m_I$,
and  for any $A^{m+1}_{f,j} \subseteq \Ascr^{m+1}_I-\Ascr_J$, $r(A^{m+s+1}_{p,q},
 A^{m+1}_{f,j}) =0$.
\end{lemma}

 \begin{proof} If $s =0$ the result is clear since the petals over $\Ascr^m_I$
 are fully independent.  If $0<s< r-m$,
we have shown in  Lemma~\ref{unique-tree-decompjb} for any $m \leq r$, any $p,q$ with
 $A^{m+s+1}_{p,q}\subseteq
\Ascr_J - \Ascr_I^m$, that $B^{m+s+1}_p \subseteq \Ascr_J^{m+s}$.
So for any $s\geq 1$, $r(A^{m+s+1}_{p,q},\Ascr_I^{m+s}-\Ascr_J) = 0$.
In particular, $r(A^{n}_{p,q},A^{m+1}_{f,j}) = 0$ for any
 $A^{m+1}_{f,j}\subseteq \Ascr_I - \Ascr_J$, and any $A^{n}_{p,q}\subseteq \Ascr_j$ with
  $m+1 \leq n \leq r$.
  \end{proof}

\begin{corollary}\label{nodefwparjb}
Assume that $\hat T_\mu$ triples.  Then, for $n>1$, no truly
 $n$-ary function is definable in $\hat T_\mu$ {\em even with parameters}.
 \end{corollary}

 \begin{proof} Let $M \models \hat T_\mu$ and suppose $\phi(y, \xbar, \cbar)$ defines a
truly $n$-ary function $y = g(\xbar)$ on $M^n$.
Taking $M$ saturated,
we can choose $a_1, a_2, \dots, a_n$ independent over $\cbar$.
Fix $\cbar'$ maximal independent inside $\cbar$.
Then $I = \{ a_1, \dots a_n\} \cup \{ x\in \cbar'\}$
is independent.
Let $\Ascr = \Ascr_I$ be the $G_I$-normal closure
of $I \cup \{d, \cbar\}$, where $d = g(\abar)$.

For each $q\in \{ 1, \dots, n\}$ we define $I_q = \{a_q\} \cup\{x\in \cbar'\}$.
Clearly, $I_q$ is independent as a subset of the independent set $I$.
Let $\Ascr_q$ be the $G_{I_q}$-normal closure of $\{\cbar\}$, that is,
 $\icl(I_q \cup G_{I_q}(\cbar))$.  We apply Lemma~\ref{unique-tree-decompjb}
 with $J = I_q$.

We now consider two cases.
First, assume that $d$ is not in $\bigcup_{q=1}^n \Ascr_q$.
Let $d \in A^{m+1}_{f, j}$, where $A^{m+1}_{f, j}$ is
a petal in the $G_I$-decomposition of $\Ascr$.
 Applying ${\rm moves}_{m+1}$ to $\Ascr$ we obtain
that there exists $A^{m+1}_{f, j'}$ with $j\ne j'$.
By strata decomposition there exists $\tau \in\aut(M/\Ascr^m)$ such that
$\tau(A^{m+1}_{f,j}) = A^{m+1}_{f, j'}$.

Since $q$ is arbitrary,
Lemma~\ref{imcor} implies
$\sum_{q=1}^n r(A^{m+1}_{f, j}, \Ascr_q-\Ascr^m) = 0$.
But we need more, namely,
$r(A^{m+1}_{f, j}, \bigcup_{q=1}^n \Ascr_q-\Ascr^m) = 0$.
Assume the contrary, that there exist $a\in A^{m+1}_{f,j}$,
$b\in \Ascr_q-\Ascr^m$, and $c\in \Ascr_p-(\Ascr^m\cup \Ascr_q)$ with $R(a, b, c)$.
Not both $b$ and $c$ can be on the same strata,
because they are in different petals but petals on the same strata are fully independent.
Then one of them, say $b$, is in a lower strata, say, in $\Ascr^{m+k}$ for some $k\ge 1$.
Then $\{a, b\} \subseteq \Ascr^{m+k}$ and $(\{a, b\}, c)$ is a good pair.
By $G$-decomposition we obtain that $\{c\} = A^{m+k+1}_{u,v}$ for some $u$ and $v$.
As $A^{m+k+1}_{u,v}\subseteq \Ascr_p$, by Lemma~\ref{unique-tree-decompjb}
its base $B^{m+k+1}_u = \{a, b\}$
is in $\Ascr_p$.
So, both $b$ and $c$ are in $\Ascr_p$, but $r(A^{m+k+1}_{u,v}, \Ascr_p - \Ascr^m) = 0$ by Lemma~\ref{imcor},
for a contradiction.

Now we finish the proof of the first case.
Let $\rho = \tau\restriction(\Ascr^m\cup A^{m+1}_{f, j}) \cup id_{\bigcup_{q=1}^n \Ascr_q-\Ascr^m}$.
Recall, that $\tau\restriction\Ascr^m$ is the identity,
so $\rho$ moves only $A^{m+1}_{f, j}$ to $A^{m+1}_{f, j'}$ and
 fixes $\Ascr^m \cup \bigcup_{q=1}^n \Ascr_q$.
By our last remark, $r(A^{m+1}_{f, j'}, \Ascr^m \cup \bigcup_{q=1}^n \Ascr_q) = r(A^{m+1}_{f, j}, B^{m+1}_{f})$.
Taking into account $A^{m+1}_{f, j} \cup B^{m+1}_{f} \cong_{B^{m+1}_{f}}A^{m+1}_{f, j'}\cup B^{m+1}_{f}$, we obtain
 $\rho$ is a partial isomorphism and can be extended  to an isomorphism $\rho'$ of $M$,
as its domain is closed in $M$.
So, $\rho'$ fixes $\abar$ and $\cbar$,
but, by choice of case, moves $d$, for a contradiction.

Theorem~\ref{nodefwparjb} by doing the second case:
$d \in \bigcup_{q=1}^n \Ascr_q$, say, $d\in \Ascr_1$.
Let $a_n'$ be independent over $\abar \cbar$.
Let $\rho\in \aut(M/\cbar)$ fix $a_q$ for $q<n$ and move $a_n$ to $a_n'$.
Then $\rho\in G_{I_1}$, so $\rho(d) = d$ and $d= g(a_1, \dots, a_{n-1}, a_n')$,
that contradicts $f$ is truly $n$-ary.
\end{proof}

\begin{remark}\label{narycomgeom}{\rm  Note that depending on whether $\hat T_\mu$ triples, there may or
may not be  a truly $n$-ary function. \cite[Theorem 3.1]{EvansFerII}
 show that for any  $\hat T_\mu$ the geometry
  (i.e. of the countable saturated model $M$),
   with any finite set
 $X$   with $|\acl_{M}(X)|$ infinite named,
 is isomorphic to that of the $\omega$-stable version of the
 construction\footnote{Mermelstein
 (personal communication) has shown this result extends to Steiner systems.}. Since
 $\acl(\emptyset)$ is easily made infinite e.g. \cite[Lemma 5.26]{BaldwinPao}, it is easy
 to construct examples with the same geometry.

We show now that for our general situation $\acl_M(X)$ is infinite for any finite $X$.
For this, it is sufficient to show that for any finite set $B$ there are infinitely many
pairwise non-isomorphic good pairs $A/B$. Let $B = \{b_1, \dots, b_k\}$.
Let $n\ge \max\{ 3, k\}$ and $A_n = \{a_1, \dots, a_n\}$. Let $\langle c_1, \dots, c_n\rangle$ be a sequence
over $B$, that contains each element of $B$. We put
$R(a_{i}, a_{i+1}, c_i)$ for each $i\in \{ 1, \dots, n-1\}$ and
$R(a_n, a_1, c_n)$.

In contrast when $B =\emptyset$
 the Hrushovski adequacy condition
can be satisfied when $\mu(A/\emptyset) =0$ for any $A$ primitive over $\emptyset$. And
it is not hard to show that amalgamation still holds \cite{BaldwinsmssII}. But, in such a
case the geometries are not elementarily equivalent
 as the formula $(\forall x) D_1(x)$ holds in the pregeometry with
  $\acl(\emptyset) =\emptyset$
where $D_n$ is a predicate that holds of $n$ independent elements.

{\em Thus the varied behavior of our examples show the coarseness of
 classifying only by geometry.}}
\end{remark}

\subsection{Counterexample} \label{ce}
\setcounter{theorem}{0} Let $M$ be any model of $\hat T_\mu$ with  $\mu(\boldsymbol{\alpha}) = 2$.
The following example satisfies $\mathrm{dcl}^*{I}\ne\emptyset$  for $I=\{ a_1, a_2\}$.
This shows the assumption that $\mu(A/B) \ge 3$ for any good pair
$(A/B)$,
where $\delta(B) = 2$, is essential to show $\mathrm{dcl}^*{I} =\emptyset$ (Theorem~\ref{mr}).

We sketch the motivation for the example.
Recall, that in the decomposition of $\Ascr$ into strata
 we have the 0-strata $\Ascr^0$,
that is obtained as $\Ascr^0 = \bigcup_{i=1}^2(\Ascr \cap \acl(a_i))$.
We have used the hypothesis, $\mu(A/B)\ge 3$ for any good pair
 $(A/B)$ with $\delta(B) = 2$, twice in the proof of Lemma~\ref{long}.
Recall Example~\ref{examp1}, which shows that it is possible
that $\Ascr$ is $G_{\{ I\}}$-invariant and
$A^2_{1,1}$ is $G_{\{ I\}}$-invariant. However, in that example
the elements $c_1$ and $c_2$ are indiscernible over $\{ a_1, a_2\}$.
This happens because the elements $a_1$ and $a_2$ are indiscernible over $\{ b_1, b_2\}$
 and
$\{ a_1, a_2\} \cup \{ b_1, b_2\} \cong_{\{ b_1, b_2\}} \{ c_1, c_2\} \cup \{ b_1, b_2\}$.
Below we modify $A^1_{1,1}$ so that $a_1$ and $a_2$ are no longer
indiscernible over $A^1_{1,1} \cup A^1_{1,2}$.

Let $tp(a_1/ A^1_{1,1}) \ne tp(a_2 / A^1_{1,1})$:  for instance the number of
relations of $a_1$ with $A^1_{1,1}$ and the number of relations of $a_2$ with
$A^1_{1,1}$ are different.
Then $A^2_{1,1}$ contains a copy of $a_1$, or a copy of $a_2$, or both. But
if $a_1$ and $a_2$ are distinguishable over $B^1_1$ inside $\Ascr^1$
(realize different types over $B^1_1$), then their
copies $\alpha_1, \alpha_2$ in $A^2_{1,1}$ are distinguishable, too.
But then the $\alpha_i$ can belong to $\dcl^*(I)$.

\bigskip
The following construction describes an accessible case of
  the  general strategy described in Remark~\ref{exstrat}.

\begin{example}\label{examp2}{\rm
We consider the following example with universe $\{a_1,a_2\} \cup\{c_1,c_2,c_3\}$
and then $9$ more points with $d,\delta,\gamma$ replacing $c$. We define the following relations:

\begin{itemize}
\item[$0$)] $\Ascr^0 = I = \{ a_1, a_2\}$.
\item[$1,1$)] $A_{1,1}^1 = \{ c_1, c_2, c_3\}$ with  $R(a_1, c_1, c_3)$,
$R(a_2, c_1, c_2)$, and $R(a_2, c_2, c_3)$.
\item[$1,2$)] $A_{1,2}^1 = \{ d_1, d_2, d_3\}$ with  $R(a_1, d_1, d_3)$,
$R(a_2, d_1, d_2)$, and $R(a_2, d_2, d_3)$.
\item[$2,1$)] $A_{1,1}^2 = \{ \alpha_1, \alpha_2, \gamma_1, \delta_1, \gamma_3, \delta_3\}$ with the following relations:
$R(\alpha_1, \gamma_1, \gamma_3)$,
$R(\alpha_2, \gamma_1, c_2)$, $R(\alpha_2, c_2, \gamma_3)$, and
$R(\alpha_1, \delta_1, \delta_3)$,
$R(\alpha_2, \delta_1, d_2)$, $R(\alpha_2, d_2, \delta_3)$.
\end{itemize}

Set $\mu(A^1_1) = \mu(A^2_1) =2$.
In the diagrams, we represent a triple satisfying $R$ by a triangle.

\begin{center}
\begin{minipage}[h]{0.65\linewidth}
\begin{center}
\begin{picture}(180,140)
\put(8, 44){$a_1$}
\put(8, 104){$a_2$}
\put(22, 45){\circle*{3}}
\put(22, 105){\circle*{3}}
\put(6, 15){\dashbox{1}(30,120)[c]{~}}
\put(8, 0){$\Ascr^0$}
\put(40, 78){\dashbox{1}(60,57)[c]{~}}
\put(80, 126){$A^1_{1,2}$}
\put(52, 90){\circle*{3}}\put(44, 95){$d_1$}
\put(82, 105){\circle*{3}}\put(70, 108){$d_3$}
\put(76, 126){\circle*{4}}\put(63, 127){$d_2$}
\qbezier(52,90)(64,108)(76,126)
\qbezier(22,105)(49,115)(76,126)
\qbezier(22,105)(37,98)(52,90)
\qbezier(76,126)(79,115)(82,105)
\qbezier(22,105)(52,105)(82,105)
\qbezier(52,90)(60,93)(82,105)
\qbezier(22,45)(37,68)(52,90)
\qbezier(22,45)(52,75)(82,105)
\put(40, 15){\dashbox{1}(60,57)[c]{~}}
\put(80, 19){$A^1_{1,1}$}
\put(52, 30){\circle*{3}}\put(44, 25){$c_1$}
\put(82, 45){\circle*{3}}\put(78, 38){$c_3$}
\put(76, 66){\circle*{4}}\put(63, 66){$c_2$}
\qbezier(52,30)(64,48)(76,66)
\qbezier(22,105)(49,86)(76,66)
\qbezier(22,105)(37,67)(52,30)
\qbezier(76,66)(79,55)(82,45)
\qbezier(22,105)(52,75)(82,45)
\qbezier(52,30)(67,37)(82,45)
\qbezier(22,45)(37,38)(52,30)
\qbezier(22,45)(52,45)(82,45)
\put(42, 0){$\Ascr^1$}
\put(104, 15){\dashbox{1}(70,120)[c]{~}}
\put(105, 0){$\Ascr^2$}
\put(155, 19){$A^2_{1,1}$}
\put(158, 45){\circle*{3}} \put(159, 46){$\alpha_1$}
\put(158, 105){\circle*{3}}\put(159, 106){$\alpha_2$}
\put(132, 45){\circle*{3}} \put(135, 48){$\gamma_3$}
\put(132, 105){\circle*{3}}\put(130, 108){$\delta_3$}
\put(112, 30){\circle*{3}}\put(113, 25){$\gamma_1$}
\put(112, 90){\circle*{3}}\put(108, 95){$\delta_1$}
\qbezier(158,45)(145,75)(132,105)
\qbezier(158,45)(135,68)(112,90)
\qbezier(132,105)(122,97)(112,90)
\qbezier(158,45)(145,45)(132,45)
\qbezier(158,45)(135,38)(112,30)
\qbezier(132,45)(122,37)(112,30)
\qbezier(158,105)(117,115)(76,126)
\qbezier(158,105)(135,97)(112,90)
\qbezier(112,90)(94,108)(76,126)
\qbezier(158,105)(117,115)(76,126)
\qbezier(158,105)(145,105)(132,105)
\qbezier(132,105)(104,115)(76,126)
\qbezier(158,105)(117,85)(76,66)
\qbezier(158,105)(135,67)(112,30)
\qbezier(112,30)(94,48)(76,66)
\qbezier(158,105)(117,85)(76,66)
\qbezier(158,105)(145,75)(132,45)
\qbezier(76,66)(105,55)(132,45)
\end{picture}
\end{center}
\captionof{figure}{$\mathrm{dcl}^*(I)\ne\emptyset$}\label{Ex4} 
\end{minipage}
\hfill
\begin{minipage}{0.30\linewidth}
\begin{center}
\begin{picture}(104,140)
\put(8, 44){$a_1$}
\put(8, 104){$a_2$}
\put(22, 45){\circle*{3}}
\put(22, 105){\circle*{3}}
\put(6, 15){\dashbox{1}(30,120)[c]{~}}
\put(8, 0){$\Ascr^0$}
\put(40, 78){\dashbox{1}(60,57)[c]{~}}
\put(80, 126){$A^1_{1,2}$}
\put(52, 90){\circle*{3}}\put(44, 95){$d_1$}
\put(82, 105){\circle*{3}}\put(70, 108){$d_3$}
\put(76, 126){\circle*{4}}\put(63, 127){$d_2$}
\qbezier(52,90)(64,108)(76,126)
\qbezier(22,105)(49,115)(76,126)
\qbezier(22,105)(37,98)(52,90)
\qbezier(47,93)(57,107)(67,122.3)
\qbezier(42,95.2)(50,107)(58,118.7)
\qbezier(37,98)(43,107)(49,115.3)
\qbezier(32,100.5)(36,106)(40,112)
\qbezier(27,103)(29,105.5)(30.5,108)
\qbezier(76,126)(79,115)(82,105)
\qbezier(22,105)(52,105)(82,105)
\qbezier(52,90)(60,93)(82,105)
\qbezier(72,105)(69,114)(67,122)
\qbezier(62,105)(60,112)(58,118.7)
\qbezier(52,105)(50,110)(49,115)
\qbezier(42,105)(41,108)(40,112)
\qbezier(32,105)(31.5,106.5)(31,108)
\qbezier(22,45)(37,68)(52,90)
\qbezier(22,45)(52,75)(82,105)
\qbezier(47,82.5)(59.5,88.5)(72,95)
\qbezier(42,75)(52,80)(62,85)
\qbezier(37,67.5)(44.5,71.5)(52,75)
\qbezier(32,60)(37,62.5)(42,65)
\qbezier(27,52.5)(29.5,53.75)(32,55)
\put(42, 0){$\Ascr^1$}
\end{picture}
\end{center}
\captionof{figure}{$\Ascr^1$}\label{Ex5}    
\end{minipage}
\end{center}

Figure~\ref{Ex5} shows by shaded triangles the $R$-triples in $I\cup \{ d_1, d_2, d_3\}$.
The petals $A^1_{1,1}$ and $A^1_{1,2}$ are isomorphic over $I$.

Clearly, $G_I(c_2) = \{ c_2, d_2\}$, because there is no relation of either of these elements
with $a_1$ and there are two relations of each one with $a_2$.
By mapping the point with the Greek label to the corresponding Roman one, we show
that $A^2_{1,1} = \{ \alpha_1, \alpha_2, \gamma_1, \gamma_3, \delta_1, \delta_3\}$ is
  isomorphic
to $\Ascr^1-B^2_1 = \{ a_1, a_2, c_1, c_3, d_1, d_3\}$  over $B^2_1= \{ c_2, d_2\}$.

It is  routine to check that
$\Ascr^1-B^2_1$ is $0$-primitive over $B^2_1$.
Obviously, $A^2_{1,1}$ is $G_I$-invariant.
The element $\alpha_2$ is a unique element in $A^2_{1,1}$ which is in $4$ relations in $\Ascr^2$,
so $\alpha_2\in \mathrm{dcl}^*(I)$.}
\end{example}

\begin{remark}\label{exmoral} {\rm
\begin{enumerate}\item In Example~\ref{examp2} ${\rm moves}_1$ and ${\rm dim}_1$
one hold; but $\mu^2_{1,1} =2$   so we cannot apply Lemma~\ref{long}.A to
conclude ${\rm moves}_2$. In fact, $A^2_{1,1}$ is $G_I$-invariant.
\item
Note that this example is not a linear space (Section~\ref{Steiner}); if it satisfied the
linear space
axiom each of $A^1_{1,1}$ and $A^1_{1,2}$ would be a clique.
\item $\alpha_2$ is in $\dcl^*(I)$ but not in $\sdcl^*(I)$,
because an automorphism which swaps $a_1$ and $a_2$ cannot
preserve $A^1_{1,1}\cup A^1_{1,2}$, since in $\Ascr^1$, $a_1$ is in two relations
and $a_2$ is in four relations.  Thus this structure is $G_I$-invariant but  not $G_{\{I\}}$-invariant. In order
to build $\Ascr^{G_{\{I\}}}$ we add new copies of $A^1_{1,1}$ and $A^1_{1,2}$:
\begin{enumerate}
\item[$2,1$)] $A_{2,1}^1 = \{ c'_1, c'_2, c'_3\}$ with  $R(a_2, c'_1, c'_3)$,
$R(a_1, c'_1, c'_2)$, and $R(a_1, c'_2, c'_3)$.
\item[$2,2$)] $A_{2,2}^1 = \{ d'_1, d'_2, d'_3\}$ with  $R(a_2, d'_1, d'_3)$,
$R(a_1, d'_1, d'_2)$, and $R(a_1, d'_2, d'_3)$.
\end{enumerate}
Now there is an $ f \in G_{\{I\}}$  with  $f(a_0) = a_1$, $f(a_1) = a_0$ that maps $A^1_{1,i}$ to  $A^1_{2,i}$.
One can construct an $A^2_{2,1}$ containing $\alpha'_2$ that is the image of $\alpha_2$ under $f$.
\item
Note that $\{ A^1_{1,1}, A^1_{1,2}\}$ is called
a flower of $A^1_{1,1}$ over its base $B^1_1$ in Definition~\ref{flowerdef}.
Also note that $\{ A^1_{2,1}, A^1_{2,2}\}$ is another flower of $A^1_{1,1}$ over its base
$B^1_1$. The difference is that if we arrange $B^1_1$ for the first considered flower
as $\langle a_1, a_2\rangle$, then for the second flower the arrangement of $B^1_1$
must be $\langle a_2, a_1\rangle$.
\item  In Definition~\ref{bouquetdef}, we call the collection
$\{ \{ A^1_{1,1}, A^1_{1,2}\}, \{ A^1_{2,1}, A^1_{2,2}\}\}$ of all of these flowers a bouquet.
\end{enumerate}}
\end{remark}

We now explain here the methodology and motivation  for  constructing a set
with non-empty $\dcl^*(I)$. It may be useful for further examples.

 \begin{remark}\label{exstrat} {\rm
Let $E_i$ be a subset of $\acl(a_i)$, for $i\in \{1, 2\}$
such that $\delta(E_1) = \delta(E_2) =1$ (that is, $E_i \le M$).
The most simple case is $E_i = \{ a_i\}$.
Let $B^1_1 = E_1 \cup E_2$.
Then $\delta(B^1_1) = 2$. Let $A^1_{1,1}$ be any set that is good over $B^1_1$. We
put $\mu(A^1_{1,1}/B^1_1)=2$.
So, $\Ascr^1 = \icl(I \cup E_1 \cup E_2) \cup A^1_{1,1} \cup A^1_{1,2}$.

We choose one element $b_i$ from $A^1_{1,i}$ for $i = 1,2$.
Let $B^2_1 = \{ b_1, b_2\}$.
Then $\delta(B^2_1) = \delta(\Ascr^1) = 2$ and
there is a chain $B^2_1 = X_0 \le X_1 \le \dots \le X_r =
\Ascr^1$ such that $X_{i+1}$ is a $0$-primitive extension of $X_i$.
 So, $X_1$ is a $0$-primitive extension of $B^2_1$ and is a subset of $\Ascr^1$.
We must choose $A^1_{1,1}$ and $B^2_1$ so that $X_1$ is good over $B^2_1$.
This is not true in general\footnote{Here is  a counterexample.  To begin with
 we find a $0$-primitive extension of a one element  set, $\{b\}$.
 We consider $A'$ as  four points
$c_1,c_2, d_1, d_2$ satisfying  $R(c_1,c_2,b)$, $R(d_1,d_2,b)$, $R(c_1,d_1,b),
R(c_2,d_2,b)$. There are 5 points 4 edges and any subset has larger $\delta$.

We would like to make this structure $0$-primitive over $a_1,a_2$. It needs one  more trick.
Replace $A'$ by
$A$ by adding a point $c_3$ to $A'$  and replacing the edge $R(c_1,c_2,b)$ by
two edges $R(c_1,c_2,c_3), R(c_2,c_3,b)$. Then, $A$
is $0$-primitive over $\{b\}$.
Now consider two new elements $A_0 =\{a_1,a_2\}$; we want $A$ $0$-primitive over $A_0$. For this
let the new relations be $R(c_1,c_2,a_1),R(d_1,d_2,a_2)$.

So while the discussion here is fine for motivating the example it doesn't suffice
 to show that $\Ascr^1$ with  pair $A^1_{1,1},A^1_{1,1}$ must contain an $A^2_{1,1}$
  good over one point from each.}, but Example~\ref{examp2} shows it can be done.

We are going to find $A^2_{1,2}$ inside $\Ascr^1$ in order to make
$A^2_{1,1}$ $G_I$-invariant.
Suppose there are $m$ copies of $X^1$ over $B^2_1$ that are inside $\Ascr^1$;
put $\mu(X_1/ B^2_1) = m+1$.
Let $A^2_{1,1}$ be the (m+1)-th copy of $X_1$ over $B^2_1$. Obviously,
$A^2_{1,1}$ is not in $\Ascr^1$. We put $\Ascr^2 = \Ascr^1 \cup A^2_{1,1}$.
 If $B^2_1$ is fixed pointwise by $G_{I}$
 (that is, $b_1$ is definable
in $A^1_{1,1}$ over $I$ and $b_2$ its copy in $A^1_{1,2}$),
then $\ell^2_1 = 1$ and $|J^2_1| = 1$.}
\end{remark}

The following is not essential to achieving Remark~\ref{exstrat} but is mandated by the construction.
\begin{claim}\label{smclaim}
The intersection $X_1 \cap B^1_1$ is not empty.
\end{claim}

\begin{proof}
This follows from the fact that $A^1_{1,1}$ and $A^1_{1,2}$ are free over
$\Ascr^0$.
Indeed, assume that $X_1\cap B^1_1 = \emptyset$.
Then
$$0=\delta(X_1/B^2_1) = \delta (X_1\cap A^1_{1,1}/B^2_1)
+ \delta (X_1\cap A^1_{1,2}/B^2_1) + \delta (X_1 - (A^1_{1,1}\cup A^1_{1,2})/ B^2_1)$$
because there are no relations between $A^1_{1,1}$, $A^1_{1,2}$ and $\Ascr^0-B^1_1$.
Then each of these predimensions
is equal to $0$; that contradicts the definition of a good pair.\end{proof}

\subsection{Bouquets and Flowers}\label{bouflow}

In Remark~\ref{exmoral}.3, we noted that to make a $G_{\{I\}}$-normal structure
we required not only
an image of a $0$-primitive $A$ with base $B$ but
an image  $\pi(A)$ for a $\pi$ in $G_{\{I\}}$ that fixes $B$ setwise but
not pointwise. The analysis of the case where there are good pairs $A/B$
with $\delta(B) =2$ and $\mu(A/B) =2$ requires a much finer analysis of the
second realization of $A/B$.  We introduce here some further notation to
describe the situation and illustrate them in
Example~\ref{overlap}. 

\begin{definition}[flower]\label{flowerdef}
Let $A/B$ be a good pair.
A {\em flower} $\Fscr$ of $A/B$ in a set $\Dscr$ is the set of all images of
isomorphisms
    of $A$ over $B$ into   $\Dscr$ which fix $B$ {\em pointwise}.
    The elements of the flower
    are called {\em petals}\footnote{In Construction~\ref{decomp1} there were $\mu(A/B)$ petals,
    the $A^{m}_{i,f}$ and $C^{m}_{i,k}$. We no longer assume that $A/B$ is well-placed and we allow the
    petals to intersect  so we have
    less control over the number of petals; in particular it will vary with $\Dscr$.}.

Suppose a flower $\Fscr$ of $A/B$ is a subset of $\Ascr$.
 A {\em
   certificate}  $\Cscr$ of $A/B$   (witnessing $\Ascr \in \bK_\mu$)
     is a {\em maximal disjoint} set of   $\chi_{\Ascr}(A/B)= \mu(A/B)$ images of
isomorphisms
    of $A$ over $B$ into $\Ascr$ that fix $B$ {\em pointwise}.
    \end{definition}

      When $A/B$ is well-placed,
   a flower $\Fscr$ contains at least one certificate  $\Cscr$ for
    $\chi_M(A/B) =\mu(A/B)$ and,
   since each intersection decreases $\delta$, $|\Fscr| \leq \mu(A/B) + \delta(B)$.
Moreover any pair of petals from distinct certificates (or flowers) that
intersect are in $\icl(B) \subseteq \Ascr^m$,  for the least $m$ such that $B\subseteq \Ascr^m$.

Of course each petal $C \in \Fscr$ is isomorphic to $A$ over $B$. $\Dscr$
will usually be fixed in context as either the generic $M$ or a
$G_{\{I\}}$-decomposable $\Ascr$ (e.g. an $\Ascr^m$).
There are only finitely many certificates of $A/B$ in $M$; an upper bound
is $\binom{\mu(A/B)}{\mu(A/B) + \delta(B)}$.

Note that in the description of the class $\bL_{\mu}$
one put the upper bound on the cardinality of a certificate of a
good pair $A/B$---it does not exceed $\mu(A/B)$.

When we write two structures $C$ and $D$ are equal
we mean they have both the same domain and
each symbol in the vocabulary has the same interpretation in each.
For a substructure  $X$ of $M$, $diag_X(\xbar)$ denotes the diagram of $X$,
with respect to a {\em fixed enumeration}  $\xbar$ of the domain of $X$.

\begin{notation}\label{nameenum} Let $\bbar = \langle b_1, \dots, b_n\rangle$  enumerate $B$
and   $\gamma \in \aut(B)$; write $\bbar^\gamma$ for
$\langle \gamma(b_1), \dots, \gamma(b_n)\rangle$.

Any sequence $\cbar $ that satisfies
$diag_{A\cup B}(\xbar, \bbar^\gamma)$ determines an {\em enumeration of a   petal} of the
flower  of $A/B$. The set enumerated by  this sequence  is a petal
$F^\gamma_i$.
 Each $F^\gamma_i$ may have multiple enumerations that satisfy the fixed diagram.
  A flower
$\Fscr^\gamma$ of $A/B$ is a  maximal set $\{F^\gamma_i:i< r^\gamma\} $ of such petals.
\end{notation}

Note that for fixed $\gamma$ there may be different certificates.  Any two
such certificates
 must have at least
one pair of intersecting petals (by maximality).
But distinct flowers
$\Fscr, \Fscr'$ generated by
$A/B$ and $A'/B$ with $A,A'$ non-isomorphic over $B$ cannot have a common petal.
 If $f$ and $g$ map $A$ and $A'$ to some $A''$ while
fixing $B$ pointwise, then $g^{-1}\circ f$ is an isomorphism from $A$ to $A'$
fixing $B$.

However, a different problem appears when we allow automorphisms that fix the
base setwise but not pointwise. We must do this when considering $G_{\{I\}}$
since $I$ itself can be the base.

For simplicity of reading we
denote $G_{\{I\}}$ by $G^*$ and $G^*_{\{B\}}$  ($G^*_{B}$) denotes the
elements of $G^*$ that fix $B$ {\em setwise} ({\em pointwise}).

We now have a subclass of the $0$-primitive extensions  $A$ where $A/B$ is
well-placed by $\Ascr^m$: the orbit of {\em the flower of $A/B$} under
$G^*_{\{B\}}$.
\begin{definition}[bouquet]\label{bouquetdef}
Let $A/B$ be a good pair.
     The bouquet $\Bscr$ of $A/B$ is the collection of all images $\{\pi(F_i):F_i \in \Fscr\}$ of each flower
      $\Fscr$ as $\pi$ ranges through elements of
				$G^*_{\{B\}}$.
\end{definition}

Can two flowers in  a bouquet contain a common petal? When does a bouquet
contain more than one flower? Lemma~\ref{different-flowers} and
\ref{flower=bouquet1} answer these questions.

\begin{lemma}\label{different-flowers}
Let $A/B$ be a good pair and $\{ \{ F_i^\gamma : i < r_\gamma\} : \gamma \in
G^*_{\{B\}} \}$ list its bouquet $\Bscr$.
If $i\neq j$ then $F_i^\gamma \ne F_j^\delta$ for each
$\gamma, \delta \in G^*_{\{B\}}$ unless
$\Fscr_\gamma=\{ F_t^\gamma : t < r_\gamma\}$ and $\Fscr_\delta =
 \{F_t^\delta : t < r_\delta\}$ are the same flower.
\end{lemma}

\begin{proof}
Assume   that $F_i^\gamma = F_j^\delta$ for some $\gamma$, $\delta \in
\aut(G^*_{\{B\}})$. We will show $\Fscr_\delta = \Fscr_\gamma$.

Let $\langle f_1, \dots, f_k\rangle$ be an enumeration of $F_i^\gamma$,
such that
$$M\models diag_{A\cup B} (f_1, \dots, f_k, \gamma(b_1), \dots, \gamma(b_n))$$
Since $|F_i^\gamma| = |F_j^\delta|$ there is $\varepsilon \in
S_k$ such that $\langle f_1, \dots, f_k\rangle= \langle f_{\varepsilon(1)},
\dots, f_{\varepsilon(k)}\rangle$ and
$$M\models diag_{A\cup B} (f_{\varepsilon(1)}, \dots,
f_{\varepsilon(k)}, \delta(b_1), \dots, \delta(b_n))$$

Let $s <r_\gamma$
 and let  $\langle d_1, \dots, d_k\rangle$ enumerate the petal $C_s^\gamma$ of
the flower $\Fscr_\gamma$
 of $A$ over $B$.
That is,
$$M\models diag_{A\cup B} (d_1, \dots, d_k, \gamma(b_1), \dots, \gamma(b_n)).$$
By the property of $\varepsilon$ noted above and the
definition of diagram, we have
$$M\models diag_{A\cup B} (d_{\varepsilon(1)}, \dots, d_{\varepsilon(k)},
 \delta(b_1), \dots, \delta(b_n))$$
but $D = \{d_{\varepsilon(1)}, \dots, d_{\varepsilon(k)}\}$ is
a petal of the flower $\Fscr_\delta$.
 Obviously, $F^\gamma_s =D$.
So, each petal of the flower $\{ F^\gamma_t : t < r_\gamma\}$ is also a
petal of $\{ F^\delta_t : t < r_\delta\}$.

The inverse inclusion is similar. Hence, the flowers are equal.
\end{proof}

We can now conclude:

\begin{lemma}\label{flower=bouquet1}
 Assume that $A/B$ is well-placed by some $\Dscr \supseteq B$ and $A$ is $G_{\{
I\}}$-invariant. Then
\begin{enumerate}

\item the bouquet of $A/B$ consists of a single flower;
\item the bouquet of $A/B$ is $G_{\{ I\}}$-invariant.
\end{enumerate}
\end{lemma}
\begin{proof}

1) Assume to the contrary that the bouquet of $A^{m+1}_{j,1}$ over
$B^{m+1}_{j}$ consists of at least two flowers. Let $\pi \in G_{\{ I\}}$ be
an automorphism which moves one flower of the bouquet of $A^{m+1}_{j,1}$ over
$B^{m+1}_{j}$ to another one. Since $A^{m+1}_{j,1}$ is $G_{\{
I\}}$-invariant, $\pi (A^{m+1}_{j,1}) = A^{m+1}_{j,1}$; so, these two flowers
have a common petal. By Lemma~\ref{different-flowers} these flowers are
equal, for a contradiction.

2) By Lemma~\ref{omni-G}.1,   $B$ is $G_{\{ I\}}$-invariant, so each $g \in
G_{\{ I\}}$ fixes $B$ setwise. But then the $C/B$-bouquet is just the $G_{\{
I\}}$-orbit of the unique flower $\Fscr$ of $A/B$, namely $\Fscr$.
\end{proof}

We now give several examples to clarify the relationship
among these concepts.

\begin{example}\label{overlap} \rm
1) Two certificates in the same flower:
Let $A/B$ be $0$-primitive and $C^i_j$ for $i<3,j<2$ be isomorphic with $A$ over
$B$.
 For each $i$, $|C^i_0\cap  C^i_1| =1$; these are the only
intersections. Let $\Dscr = B\cup \bigcup_{i<3,j<2}C^i_j $.
 $A/B$ is
well-placed by $\Dscr$. $\{A\} \cup \{C^i_j: i<3,j<2\}$ is the flower of $A/B$.
But each of $\{A\} \cup \{C^i_j: i<3\}$ for $j=0$ and $j=1$ is a certificate
(Actually, there are 8 certificates.)

2) Two flowers in the same bouquet: Let $B = \{ b_1, b_2\}$ and $C_i = \{
c^i_1, c^i_2, c^i_3\}$ with $R(b_1, c^i_1, c^i_2)$, $R(b_2, c^i_2, c^i_3)$,
$R(b_2, c^i_3, c^i_1)$, for $i = 1,2$, and let $\mu(C_i/B) = 2$.

Let $D_i = \{ d^i_1, d^i_2, d^i_3\}$ with $R(b_2, d^i_1, d^i_2)$, $R(b_1,
d^i_2, d^i_3)$, $R(b_1, d^i_3, d^i_1)$, for $i = 1,2$. There is a $\pi \in
G^*_{\{I\}}$ that swaps $b_1$ and $b_2$ and takes $\{C_1,  C_2\}$   to
$\{D_1, D_2\}$. Recall that flowers are given by maps that fix $B$ pointwise.
Note $C_1$ and $D_1$ are in the same orbit under $G^*_{\{B\}}$ but not
$G^*_{B}$.

There are two flowers: $\{C_1,C_2\}$ over $\langle b_1, b_2\rangle$ and
$\{D_1,D_2\}$ over
 $\langle b_2, b_1\rangle$. (They are distinct because $b_1$ occurs in two relations in
 the $D_i$ and one in the $C_i$.)
\end{example}

\medskip
\subsection{$G_{\{ I\}}$: elimination of imaginaries fails}\label{geq2}
\setcounter{theorem}{0}

{\bf Context}
We showed that $\dcl^*(I) = \emptyset$ and so $\sdcl^*(I) = \emptyset$,
 provided that $\mu$ triples: $ \ \ \mu(C/B) \geq 3$ for
$\delta(B) =2$ with $|C|>1$. So $\hat T_\mu$ does not admit elimination of imaginaries.
Now we are going to show that the
symmetric $\sdcl^*(I)$ is empty for any $\mu$ satisfying Hrushovski's original
 conditions and so elimination of imaginaries fails.  That is,
we now omit the adequacy hypothesis that governed Section~\ref{geq3}.
There may now be definable truly
binary functions but elimination of imaginaries still fails.
The innovation is to  consider the action of $G_{\{ I\}}$ rather than $G_I$, $\sdcl$ rather than
$\dcl$.

Recall that  in Example~\ref{examp2}
 $d(G_I(a_1)) = 1$, since $G_I(a_1) = \{ a_1\}$.
The situation differs when we consider $G_{\{I\}}$.
In this case, working in a $G_{\{I\}}$-normal set,
$G_{\{I\}}(a_1) = \{ a_1, a_2\}$, so, $d(G_{\{I\}}(a_1)) = 2$. Similarly,
while $\alpha_2$ is in $\dcl(I)$,  $\sdcl(I) = \emptyset$.
In general, the $G_I$-invariant set generated by a set $U$ is contained
in the $G_{\{I\}}$-invariant set $U$ generates.

While in the proof of Theorem~\ref{mr} we showed ${\mathrm{dim}_m}$ for $m\ge 1$,
here we shall prove $\mathrm{sdim}_m$ for $m\ge 0$. Allowing $m=0$
 has a crucial role for application
of Claim~\ref{inv2} in the proof of Theorem~\ref{no-sym-fun},
showing  any $G_{\{I\}}$-invariant subset
of $\Ascr^2$ is safe.
In Example~\ref{examp2} one can see the difference between a flower and
a bouquet
and how the notion of bouquet works for the proof of Theorem~\ref{no-sym-fun}
(Remark~\ref{exmoral}(3)--(5)).

\begin{theorem}\label{no-sym-fun}
 If $\hat T_\mu$ is as in Definition~\ref{defT}, then there is no symmetric
 $\emptyset$-definable truly $n$-ary function for $v\ge 2$,
i.e., $\sdcl^*(I) = \emptyset$ for any $v$-element independent set $I$.  That is,
 there is no $\emptyset$-definable truly $n$-ary function whose value does not depend
 on the order of the arguments.
	Thus, $\hat T_\mu$ does not admit elimination of imaginaries. (See Theorem~\ref{noei}.)

As a corollary, we obtain that $\sdcl(J) = \bigcup_{a\in J}\sdcl(a)$ for
 any independent set $J$.
\end{theorem}

In contrast to Section~\ref{geq3}, we work now with a {\em global induction} on the height $m_0$ of
$G_{\{I\}}$-decompositions  of   finite, $G_{\{I\}}$-invariant subsets $\Ascr$ of $\acl(I)$
 with $I \subseteq \Ascr \leq M$.
We show for each $m_0$,  for all such
decompositions of height $m_0$, for all $m \leq m_0$, $\sdim_m$ holds.
While we analyze a specific $G_{\{ I\}}$-normal $\Ascr$ containing $I$
 and
 a $G_{\{ I\}}$-decomposition
of $\Ascr$ into strata $\Ascr^n$ as in Section~\ref{decomp}, the contradiction will result
in most involved case a second normal subset of $\Mscr$.
The analysis takes into account that the resulting $\Ascr^n$ are
 now  $G_{\{I\}}$-invariant.
For this we need to introduce the induction hypothesis in Lemma~\ref{d(orbit)=2} on the dimension of $G_{\{I\}}$-invariant
sets.

\begin{definition}[Safe]\label{sdimdef}
Let $X$ be contained in  a finite $G_{\{I\}}$-invariant set $ \Ascr$.  We say $X$ is {\em safe}
if $ d(E) \ge 2$ for any $G_{\{I\}}$-invariant set $E\subseteq
  X$ that is not a subset of $\acl(\emptyset)$.

The $G_{\{ I\}}$-decomposition $\Ascr^m$ of $\Ascr$
 satisfies $\sdim_m$ if every $G_{\{I\}}$-invariant subset of $\Ascr^m$ is safe.
\end{definition}
In addition to changing the group, the requirement $E \nsubseteq \Ascr^0$ has been
replaced by  $E \nsubseteq\acl(\emptyset)$.
So, the main differences between Theorem~\ref{mr} and
Theorem~\ref{no-sym-fun} are the following:
\begin{itemize}
\item There may be cases where $A^{m+1}_{f,k}$ is $G_{\{ I \}}$-invariant,
because there is no longer the restriction that $\mu(C/B) \ge 3$
(In Example~\ref{examp1}, Figure~\ref{Ex2}, $A^2_{1, 1}$ is $G_{\{I\}}$-invariant);
\item Different $A^{m+1}_{j,i}$ may be shown safe for different reasons. (Lemma~\ref{backward})
\item For any $e \in \Ascr -\acl(\emptyset)$, $d(G_{ I }(e)) $ may be $1$, but we show
$d(G_{\{ I \}}(e)) \ge 2$.
\end{itemize}

Note that $X$ is $G_{\{I\}}$-invariant implies $X$ is $G_I$-invariant.
Analogously to Notation~\ref{m0def} we write $s\ell^{m+1}_j$ for the number of images under $G_{\{I\}}$ of $A^{m+1}_j$ that
 do not intersect  $\Ascr^{m}$.
Since $G_{\{I\}} \supseteq G_I$, $s\ell^{m+1}_j \geq \ell^{m+1}_j$.
 Results from Section~\ref{geq3} for $G_I$ decompositions do not automatically extend.
We will now prove $\sdim_{m_0}$ holds not by a
dual induction but by distinct
arguments depending on whether $\mathrm{move}_m$ holds at a given stage, which requires
 an even more global induction on all $G$-normal decompositions rather that
 the length of a fixed decomposition. The main lemma becomes:

\begin{lemma}\label{d(orbit)=2}
Let $\hat T_\mu$ be as in Theorem~\ref{no-sym-fun}.
Then for every finite $G$-normal $\Ascr \subseteq \acl(I)$
and every $G_{\{I\}}$-decomposition
$\langle \Ascr^i: i<  m_0^{\Ascr}  \rangle  $ of $\Ascr$:
$$\mbox{for every } m\leq m_0\   \sdim_m \mbox{ holds of } \Ascr.$$
\end{lemma}

Theorem~\ref{no-sym-fun} follows from Lemma~\ref{d(orbit)=2},
because if there were a $u \in \sdcl(I)$, then $G_{\{I\}}(   u) = \{ u \}$
and so $d(G_{\{I\}}(u) )\le \delta (G_{\{I\}}(u)) =1$.

We cannot prove that every $A^{m+1}_{j,k}$ is moved by $G_{\{ I\}}$ (Remark~\ref{exmoral}).
Rather, we show that
if $A^{m+1}_{j,k}$ is $G_{\{ I\}}$-invariant then each $s \in A^{m+1}_{j,k}$
satisfies $\dim(G_{\{ I\}}(s)) \ge 2$.
 If
$\mu^{m+1}_j = 2$ (recall that $\mu^{m+1}_j = \mu(A^{m+1}_{j,k}/B^{m+1}_{j})$),
the argument turns out to be a short argument (Lemma~\ref{one-copynew}).
So we assume below that $\mu^{m+1}_j \geq 3$.
The global induction for Theorem~\ref{no-sym-fun} obtains from a  failure of $\sdim_{m+1}$
 another $G_{\{ I\}}$-invariant set $\tilde \Ascr^* =\Ascr^{m-1}\cup \tilde A^m_{1,j}$.
The height of $\tilde \Ascr^*$ is $m$,
but $\tilde A^m_{j,1}$
contains an element $e'$ such that $d(G_{\{ I\}}(e))\le 1$,
which violates the inductive hypothesis, $\mathrm{sdim}_m$, for $\tilde \Ascr^*$.
However, $\tilde \Ascr^*$ need not be contained in $\Ascr$.

 The proof is a lengthy induction. We start with the following claim which
 is blatantly false for $G_I$.

\begin{claim}\label{dim0}
The statement $\sdim_0$ holds: every $G_{\{I\}}$-invariant subset of $\Ascr^0$ has
dimension at least 2 provided that this set is not a subset of $
\acl(\emptyset)$.
\end{claim}

\begin{proof}
Without loss,  since $e \in \acl(\emptyset)$ implies
$d(G_{\{I\}}(e)) = 0$, let $e \in \Ascr^0 \setminus \acl(\emptyset)$.
Then, since $d(e) \leq 1$, $e \in \acl(a_1)\cup \acl(a_2)\cup \dots \cup \acl(a_v)$,
say, $e\in \acl(a_1)$.
Since $e \in \acl(a_1)\setminus \acl(\emptyset)$, we obtain $a_1 \in \acl(e)$.

Let $g_i \in G_{\{ I\}}$ be such that $g_i(a_1) = a_i$. Such a $g_i$ exists because
 the $a_i$ are independent and
strong minimality implies there is a unique non-algebraic type over the empty set.
Then $a_i \in \acl(g_i(e))$.
Thus $\{ a_1, a_2, \dots, a_v\} \subseteq \acl(\{g_i(e) :
i = 1, \dots, v\}) \subseteq \acl(G_{\{ I\}}(e))$.
So, $d(G_{\{ I\}}(e)) = v \ge 2$.
\end{proof}

By Lemma~\ref{petals} (the inductive step, `$moves_m$ implies $dim_m$, from Section~\ref{geq3})
 applied to a $G_{\{I\}}$-decomposition, when $A^{m+1}_{j,i}$ moves,  we have:

\begin{claim}\label{invm}{If $m\geq 1$} and $A^{m+1}_{j,i}$ is
not $G_{\{I\}}$-invariant,
$\sdim_m$ implies    $G_{\{I\}}(A^{m+1}_{j,i})$  is safe.  
\end{claim}

This is the first divide;  it tells us two things. 1) We can prove $\sdim_{m+1}$ by
 showing individual $G_{\{I\}}$-invariant petals
are safe (Definition~\ref{sdimdef}) and 2) $\sdim_1$ is true as in Corollary~\ref{cl1}.

{\em For the remainder of Section~\ref{geq2} we assume $\sdim_m$ holds for each
$G_{\{I\}}$-normal $\Ascr$.  We show that
for any such $\Ascr$, any $G_{\{I\}}$-invariant $A^{m+1}_{j,i}$ is safe.}

We now establish some tools used below as well as show
 in Claim~\ref{one-copynew} that for each $m \ge 1$ and for $G_{\{I\}}$-invariant $A^{m+1}_{j,1}$,
$\sdim_m$  and $\mu^{m+1}_j=2$  imply $A^{m+1}_{j,1}$ is safe.
 In order to explain the main idea of the rest of the  proof we review
Example~\ref{examp1}, Figure~\ref{Ex2}, where $A^2_{1, 1}$ is $G_{\{I\}}$-invariant.
Clearly, there is an isomorphism $\rho_0$ of $A^2_{1, 1}$ to $C^2_{1,1} =\{a_1, a_2\}= I$ over $B^2_{1} = \{ b_1, b_2\}$.
Since we put $I\le M$, we have $B^2_{1}\leq M$, so $\rho_0$ can be extended to an automorphism $\rho$ of $M$.
Thus, we have found an automorphism which takes the $G_{\{I\}}$-invariant petal $A^2_{1,1}$ into $\Ascr^1$.
Moreover, $\rho(A^2_{1,1})  = C^2_{1,1} = I$ is obviously $G_{\{I\}}$-invariant.
Thus, $A^2_{1, 1}$ has a $G_{\{I\}}$-invariant copy inside $\Ascr^1$ and by the inductive hypothesis, $\sdim_1$,
this copy is safe. Now the key points are Observation~\ref{grptriv} and Lemma~\ref{one-copynew}, which allow
the transfer of safeness of $\rho(A^2_{1,1})$ to $A^2_{1,1}$.
In general, given a  $G_{\{I\}}$-invariant petal $A^{m+1}_{j,1}$ we find an automorphism $\rho$
and prove that $\rho$ takes $A^{m+1}_{j,1}$ into $\Ascr^m$ or possibly into another $G$-normal set $\tilde\Ascr^m$ of height $m$.
We show that $\rho(A^{m+1}_{j,1})$ is $G_{\{I\}}$-invariant and then by the induction hypothesis is safe.
Finally, we apply Lemma~\ref{one-copynew} to show that $A^{m+1}_{j,1}$ is safe.
Finding  $\rho$ is easy, showing that $\rho(A^{m+1}_{j,1})$ is in a $G$-normal set of height $m$
is quite simple in Lemmas~\ref{omni2} and \ref{quickstop}, but is more difficult in
Lemma~\ref{slowstop}. Much of the argument, including Subsection~\ref{bouflow}, is aimed
at proving that $\rho(A^{m+1}_{j,1})$ is $G_{\{I\}}$-invariant.

\begin{notation}\label{morenot}  Extending Notation~\ref{pm}
we write $A$ and  $B$ for $A^{m+1}_{j,i}, B^{m+1}_{j}$. $C$
 represents a
 $ C^{m+1}_{j,q}$ for arbitrary $q$,
where $ C^{m+1}_{j,q}$ for $q = 1, \dots, \nu = \mu^{m+1}_j-1$ list the isomorphic over $B$
copies of $A^{m+1}_{j,1}$ in $\Ascr$ that are subsets of $\Ascr^m$.
We may write $C^1$, \dots, $C^\nu$,
when the stratum $m$ and $j$ are fixed.
 Recall that $\mu^x_j$ abbreviates $\mu(A^x_{j,1}/B^x_j)$.
\end{notation}

We state the next observation for $G$  as $G_I$ or $G_{\{I\}}$ to emphasize it
holds for either group. Our application will be to $G_{\{I\}}$.

\begin{observation}\label{grptriv}
Consider the action of $G$ on $M$. Suppose $A$  and $C$ are $G$-invariant
subsets of $\Ascr \leq M$ with $A\cap \Ascr^m = \emptyset$  and $C\subseteq \Ascr^m$,  
and  $\rho$  is an automorphism of $M$ that takes $A$ onto $C$.
If  for an arbitrary $\alpha \in G$,
 $$(*) \ \ \hat \alpha \myrestriction (\Ascr^m \cup A)  = \alpha\myrestriction \Ascr^m \cup (\rho^{-1} \myrestriction
C \circ \alpha \myrestriction C \circ \rho  \myrestriction A).$$
extends to an element (also denoted $\hat \alpha$) of $G$ then $\rho^{-1}$
injects each orbit of $G$ on $C$ into an orbit of $G$ on $A$ as follows. For any $e \in A$ and
$\alpha$; if $\alpha(\rho(e)) = e' \in C$ then   $\hat \alpha (e)=\rho^{-1}
(\hat \alpha(\rho(e))) = \rho^{-1}(e') $ defines an injection from the
$G$-orbit of $\rho(e)$ to that of $e$. Consequently, the image $\rho(X)$ of any $G$-invariant
subset $X$ of $A$ is a union of $G$-orbits and hence $G$-invariant.
\end{observation}

Part 1) of Lemma~\ref{one-copynew} tells us that the special case where $\delta(B) =2$
we can extend an {\em isomorphism} $\rho$ from $A$ to $C$ to an automorphism of $M$
and then deduce the safety of $A$ from the safety of $C$. Part 2) asserts that the
deduction of safety is fine provided  $\rho$ extends to an automorphism. A major task in
this section will be establishing that $\rho$ has such an extension.

\begin{lemma}\label{one-copynew} Let $A=A^{m+1}_{j,1}$ be $G_{\{I\}}$-invariant
and
\begin{enumerate}
\item
 $\rho$ an isomorphism fixing $B = B^{m+1}_j$ pointwise and taking $A^{m+1}_{j,1}$
  to $ C \subseteq \Ascr^m$.

\begin{enumerate}
\item \label{one-copy-1new}
If $C$ is $G_{\{I\}}$-invariant  and there is an automorphism  $\alpha' \in G_{\{I\}}$
moving $\rho(e)$ to $e'$ for some $e \in A^{m+1}_{j,1}$, $e'\in C$,
then there exists $\hat \alpha \in G_{\{I\}}$ moving
$e$ to $\rho^{-1}(e')$. So, if $C$ is safe so is $A^{m+1}_{j,1}$.

\item In particular,
if  $A^{m+1}_{j,1}$ is $G_{\{I\}}$-invariant, $\sdim_m$ holds,
and $\mu^{m+1}_j = 2$, then
$d(G_{\{I\}}(s)) = 2$ for each $s \in A^{m+1}_{j,1}$.

\end{enumerate}
\item More generally,  let $A=A^{m+1}_{j,1}$ be $G_{\{I\}}$-invariant 
and
 $\rho$ is an automorphism of $M$ moving  $B = B^{m+1}_j$ inside $\Ascr^m$ and
  taking $A^{m+1}_{j,1}$ to a $G_{\{I\}}$-invariant $D=\rho(A^{m+1}_{j,1})$ with $D \subseteq  \Ascr^m$.
Then if $D$ is safe so is $A^{m+1}_{j,1}$.
\end{enumerate}
\end{lemma}

\begin{proof} 1a) We must show that the
$\hat \alpha$ on $\Ascr^m \cup A$  defined at (*) in Observation~\ref{grptriv}
  extends to an element of $G_{\{I\}}$.
Note $\hat \alpha$
is
well-defined on $\Ascr^m \cup A$, since $\rho$ fixes $B$ pointwise and $C$ is
fixed setwise by $\alpha$. Since $R(A, \Ascr^m) = R(A, B)$ (Definition~\ref{Rdef}) and $\rho$ is a
$B$-isomorphism from $A$ to $C$, $\hat \alpha \myrestriction \Ascr^m \cup A$ is an
 automorphism of $\Ascr^m
\cup A$. And, since $\Ascr^m \cup A \leq M$, $\hat \alpha$ extends to the
required map in $G_{\{I\}}$.

1b) Since 
$B \subseteq \acl(I)$ but $B \not\subseteq \acl(\emptyset)$, $\sdim_m$, and
the conditions on $\mu$ (in Definition~\ref{Kmu}) imply $2  \leq d(B) \leq
\delta(B) \leq 2$. Thus, $B\leq M$ and $\rho$ extends to an automorphism of $M$.
Note that all petals over $B$ are disjoint, because $B\leq
M$, so the total number of petals that are isomorphic to $A$ over
$B$ is equal to $\mu^{m+1}_j = 2$, namely, they are $A$ and $C$. Since
$A=A^{m+1}_{j,1}$ is $G_{\{I\}}$-invariant, $B^{m+1}_j$ is
$G_{\{I\}}$-invariant. Now we prove that $C$ is $G_{\{I\}}$-invariant.
By Lemma~\ref{flower=bouquet1} the bouquet on
$A^{m+1}_{j,1}/B^{m+1}_j$ is equal to the flower  on $A^{m+1}_{j,1}/B^{m+1}_j$
and has only two elements. The
global $G_{\{I\}}$-isomorphism $\rho$ guarantees the same holds for $C$.
Consequently,
$G_{\{I\}}$ fixes each of $A$, $B$, $C$ setwise. By 1a) $\rho^{-1}\myrestriction C$
induces a $G_{\{I\}}$-isomorphism from
$G_{\{I\}}(e)$ into $G_{\{I\}}(\rho(e))$. The induction hypothesis gives
$d(G_{\{I\}}(c)) =2$ for any $c\in C$, e.g., $\rho(e)$; so $d(G_{\{I\}}(e))
=2$.

2) Let $\alpha \in G_{\{I\}}$ fix $D$ setwise. Now consider
 $$\hat \alpha \myrestriction (\Ascr^m \cup A)    = \alpha \myrestriction \Ascr^m \cup (\rho^{-1} \myrestriction
D \circ \alpha \myrestriction D \circ \rho  \myrestriction A).$$
$\hat \alpha$ is well-defined and fixes $I$ as in case 1a). Since $R(A, \Ascr^m) = R(A, B)$ and
$\rho$ is  an isomorphism of
$BA$ to $\rho(B)D$,  $\rho(A)$ is good over $\rho(B)$.  But since $\Ascr^m \leq M$,
this implies $R(\rho(A), \Ascr^m) = R(\rho(A), \rho(B))$. So,
$\hat \alpha \myrestriction \Ascr^m \cup A $ is an automorphism of $\Ascr^m
\cup A $. And, since $\Ascr^m \cup A \leq M$, $\hat \alpha$ extends to the
required map in $G_{\{I\}}$. By Observation~\ref{grptriv} and since $\rho$ is an
 automorphism, if $D$ is safe, so is $A$.
\end{proof}

The following notation will be used to study the relationship between a $G_{\{I\}}$-invariant set and
the set it determines (Definition~\ref{gendef}).
If $\mu(A/B) \geq 3$ we will have the following situation.

\begin{notation}\label{hatnot} We extend Notation~\ref{morenot} to consider two levels.
We will let $\Dscr$ range over subsets of the $G$-decomposable $\Ascr$; in applications they
will usually be initial segments of the decomposition.
Let $A,B$ denote a good pair well-placed by $\Dscr \leq M$
such that $A$  is $G_{\{I\}}$-invariant.
$C$ denotes an arbitrary petal of the flower of $A/B$.
We write $\hat A, \hat B, \hat\Dscr $ for a similar triple determined (Definition~\ref{gendef}) by the first.
\end{notation}

Here is the way in which this situation arises. Suppose a $G_{\{I\}}$-invariant $A^{m+1}_{j,1}$
with base $B^{m+1}_j$ determines $A^m_{i,1}$. Then each of the $C^{m+1}_{j,k}$
intersects $A^m_{i,1}$. When $A^m_{i,1}$ is also $G_{\{I\}}$-invariant, then we get a new iteration.
We call the first level $A,B,C$ and the second $\hat A,
\hat B, \hat C$.  Similarly $\Dscr$ and $\hat \Dscr$ refer to (are instantiated as)
$\Ascr^m$ and $\Ascr^{m-1}$. We introduce this notation to avoid the distraction of the multiple super/sub scripts and
 focus on certain relationships which will appear several times in the sequel. In the crucial case where $\mu(\hat A/\hat B) =2$,
we will be able to extend the partial isomorphism
  $\rho$ over $\hat B$ taking
 $\hat A$  to its unique copy
 $\rho(\hat A) =\hat C\subseteq \hat \Dscr$ to an automorphism of $\Mscr$ also called $\rho$.

At this stage we must invoke our induction hypothesis.
\begin{lemma}\label{backward} Suppose $\Ascr$ satisfies $\sdim_m$.
\begin{enumerate}
\item $A^{m+1}_{j,1}$ is safe if either
\begin{enumerate}
\item   $A^{m+1}_{j,1}$ is not $G_{\{ I\}}$-invariant   or
\item   $A^{m+1}_{j,1}$ is   $G_{\{ I\}}$-invariant and $\mu^{m+1}_{j} =2$.
\end{enumerate}
\item If $A^{m+1}_{j,1}$ is $G_{\{ I\}}$-invariant, $|A^{m+1}_j| >1$ and $\mu^{m+1}_j\ge 3$ then
$A^{m+1}_{j,1}$ determines $A^m_{i,1}$   for some $i$.
 Moreover,  $B^{m+1}_{j}\cap A^m_{i,1} \neq \emptyset$ and
$B_+ =B^{m+1}_{j} - A^m_{i,1}\subseteq\acl(\emptyset)$.
\end{enumerate}
\end{lemma}

\begin{proof} Case 1) follows from Lemmas~\ref{invm} and \ref{one-copynew}.\ref{one-copy-1new}.
For Case 2), since $|A^{m+1}_{j,1}| >1$  the  hypotheses of Lemma~\ref{Am-not-1-G}
 hold,  we may apply Lemma~\ref{small-G} and
then Lemma~\ref{long}. Thus, $\delta(B_+)\le 1$.
 Since $\Ascr^m$ witnesses $\Ascr$ satisfies $\mathrm{sdim}_m$, and $B_+ - \acl(\emptyset)$
 is $G_{\{ I\}}$-invariant, if $B_+ - \acl(\emptyset)$ were
 nonempty it would have  dimension 2.
Thus, $B_+ \subseteq  \acl(\emptyset)$.
\end{proof}

Claim~\ref{omni2}.1  shows a stronger form of case 2) ($B_+ =\emptyset$)
  when $|A^{m+1}_{j,1}| =1$.

\begin{lemma}\label{rhoAinv} In the situation of Notation~\ref{hatnot},
suppose $\mu(\hat A/\hat B) =2$.  Let $\rho$ be a partial isomorphism over $\hat B$ from
 $\hat A$  to its unique copy
 $\rho(\hat A) = \hat C\subseteq \hat \Dscr$. Then $\rho(B)$ is $G_{\{I\}}$-invariant.
\end{lemma}

\begin{proof} By Lemma~\ref{one-copynew}.1b, $\rho$ extends to an
automorphism of $M$.
Suppose $\pi \in G_{\{I\}}$ fixes $\hat B$ setwise.
By Lemma~\ref{backward}.2,
$B  \subseteq \hat A \cup(\acl(\emptyset)\cap \Ascr^0))$
and so
$$\rho(B ) \subseteq \rho(A  \cup (\acl(\emptyset)\cap \Ascr^0)) = \hat C \cup{}
(\acl(\emptyset)\cap \Ascr^0).$$

Obviously, $\rho(B \cap \acl(\emptyset)) = B \cap
\acl(\emptyset)$ is $G_{\{I\}}$-invariant, because $B$ is
$G_{\{I\}}$-invariant. By Lemma~\ref{backward}.2, $B
- \acl(\emptyset) = B \cap A^m_{i,1}$ is $G_{\{I\}}$-invariant.
We know that both $\hat A$ and $\Hat C$ are $G_{\{I\}}$-invariant,
and that $\rho(B\cap \hat A) \subseteq \Hat C$.
Assume for contradiction that $\rho(B\cap \Hat A)$ is not $G_{\{I\}}$-invariant,
witnessed by
$\pi \in G_{\{I\}}$  such that
$$\pi(\rho(B\cap \Hat A)) \ne \rho(B\cap \Hat A).$$
Then we put
$$
\tau = \pi\myrestriction \Ascr^{m-1} \cup (\rho^{-1}\circ \pi\circ \rho)\myrestriction \Hat A
$$
Obviously, $\tau$ can be extended to an automorphism of $M$ and
$\tau(B\cap \Hat A) \ne B\cap \Hat A$, contradicting
 $G_{\{I\}}$-invariance of $B$.
\end{proof}

We continue to rely on our induction hypothesis, $\sdim_m$; we show a $G_{\{I\}}$-invariant
  $A^{m+1}_{j,1}$ with only one element determines an
$A^m_{i,1}$ with at least two elements and $\mu^m_i \geq 3$.
Parts 2) and 3) foreshadow the main argument below.

\begin{claim}\label{omni2} Assume $\sdim_m$.
If $|A^{m+1}_{j,1}| = 1$ and is $G_{\{ I\}}$-invariant then
\begin{enumerate}
\item
 $B=B^{m+1}_j \leq M$, and each $C^k = C^{m+1,k}$ is contained in $A^m_{i,1}\cup B^m_i$
and $A^{m+1}_{j,1}$ determines $A^m_{i,1}$ for some $i$.
\item Moreover, $\mu^{m+1}_j\geq 3$;
\item and $\mu^m_i \geq 3$.
\end{enumerate}
\end{claim}

\begin{proof}
1) Lemma~\ref{omni-G}.1 asserts $A^{m+1}_{j,1}=\{e\}$ determines some
$A^m_{i,1}$  for some $i$ and by $\sdim_m$, Lemma~\ref{omni-G}.2 yields $B \subseteq A^m_{i,1}$.
By $\sdim_m$ again,
$d(B)=\delta(B) = 2$; so $B\leq M$.
Since $B\subseteq A^m_{i,1}$   has relations in $\Ascr^m$ only with
elements of $A^m_{i,1}$ and its base $B^m_i$ and each $C^k$ is a singleton,
each $C^k \subseteq A^m_{i,1}\cup B^m_i$.

2) Assume to the contrary that $\mu^{m+1}_j = 2$.
By Lemma~\ref{flower=bouquet1}.2, $A^{m+1}_{j,1}$, $B^{m+1}_j$,  and
$C^{m+1}_{j,1}$ are $G_{\{ I\}}$-invariant.
Since $C^{m+1}_{j,1} \subseteq \Ascr^m$
 and is $G_{\{ I\}}$-invariant,
 $C^{m+1}_{j,1}$ is safe by induction.
But $|C^{m+1}_{j,1}| = 1$, so $C^{m+1}_{j,1} = \{ c\}$.
Then $d(G_{\{ I\}}(c)) = d(\{ c\}) \le
\delta(\{ c\}) = 1$, for a contradiction.

3) Assume to the contrary that $\mu^m_i = 2$.
Using the notation and result of Lemma~\ref{rhoAinv}, we are given
 a partial isomorphism $\rho$ taking $A^m_{i,1}$ to $C^{m}_{i,1}\subseteq \Ascr^{m-1}$.
Moreover,
$A^{m}_{i,1}, B^m_{i},C^{m}_{i,1},
\rho(A^{m+1}_{j,1})$
are all $G_{\{ I\}}$-invariant.  By $\sdim_m$,
$d(B^m_i)\ge 2$
and so $2 \le \delta (B^m_i) \le \mu^m_i = 2$ and $A^m_{i,1}$
is a $0$-primitive extension of $B^m_i$; thus,  $B^m_i\cup A^m_{i,1} \le M$.
So $\rho$ can be extended to an automorphism $\hat \rho$ of $M$.
The automorphism $\hat  \rho$ is not
 in $G_{\{ I\}}$ as it doesn't respect strata. Indeed, it may not fix $\Ascr$ setwise.

Clearly, $\hat \rho(B)$, $\hat\rho(C^1)$, \dots, $\hat\rho(C^\nu) \subseteq \hat\rho(A^m_{i,1}\cup B^m_i) = C^{m}_{i,1}\cup B^m_i
\subseteq \Ascr^{m-1}$.
Since $B \leq M$, $\hat\rho(B) \leq M$ so
by Lemma~\ref{getmax} $\mu(\rho(B), \rho(C^1)) = \mu(B, A^{m+1}_{j,1}) = \nu +1$, so
$\hat\rho(A^{m+1}_{j,1}) = \{e'\}$ is
a
$(\nu+1)$th copy of $\hat\rho(C^1)$ over $\hat\rho(B)$.
Note that $\hat\rho(A^{m+1}_{j,1})$
is the  unique such  copy which is not in $C^{m}_{i,1}$.

As $C^m_1$ is $G_{\{I\}}$-invariant,
by the `consequently' of Observation~\ref{grptriv}, $\rho(B)$
is a $G_{\{ I\}}$-invariant set and so is $\{ \hat\rho(C^1), \dots, \hat\rho(C^\nu)\}$.
By Lemma~\ref{flower=bouquet1} the bouquet of $A^{m+1}_{j,1}/B^{m+1}_j$  consists of one flower.
Again by Observation~\ref{grptriv}, the bouquet of $\hat\rho(A^{m+1}_{j,1})/\hat\rho(B^{m+1}_j)$ consists of one flower
$\{ \hat\rho(C^1), \dots, \hat\rho(C^\nu), \hat\rho(A^{m+1}_{j,1})\}$.
Since $\hat\rho(C^1), \dots, \hat\rho(C^\nu)$ are in the $G_{\{I\}}$-invariant set $C^m_{i,1}$,
 $\hat\rho(A^{m+1}_{j,1})$ is a $G_{\{ I\}}$-invariant set, because it is
disjoint from $C^{m}_{i,1}$.
Thus, $e'$ as a unique element of $\hat\rho(A^{m+1}_{j,1})$, is fixed  by $G_{\{I\}}$.

If $\hat\rho(A^{m+1}_{j,1}) \subseteq \Ascr$, by the definition of decomposition
 it is contained in $\Ascr^m\leq M$.
But this means $e'\in \sdcl^*(I)$ contradicting the induction assumption
 that $\Ascr$ satisfies $\sdim_m$,
since $G_{\{I\}}(e') = \{e'\}$, implies that
$$d(G_{\{I\}}(e')) = d(\{ e'\})\le \delta(\{ e'\}) = 1.$$

The final possibility is that $\hat\rho(A^{m+1}_{j,1}) \subseteq M -\Ascr$.	
Now we use our `global induction'.
Let    $\tilde \Ascr = \Ascr^{m-1} \cup \{e'\}$. Then $\tilde \Ascr \leq M$ (since
$\delta(\{e'\}/\rho(B)) =0$) and
$\tilde \Ascr$ is $G_{\{ I\}}$-invariant.
Thus $\tilde \Ascr$ admits a decomposition of height $m$
and $\{e'\}= \tilde A^m_{1,1}$. 
But $G_{\{I\}}(e') = \{e'\}$ contradicting
the inductive hypothesis that
$\mathrm{sdim}_m$ hold for all decompositions.
\end{proof}

The argument for Lemma~\ref{omni2}.3 shows the main idea of
 the proof of (Lemmas~\ref{quickstop},
\ref{slowstop}) for $m \geq 2$.
We cut an intermediate strata  out, preserving the top,
in order to obtain a counterexample with  smaller height.
There are three possibilities for $A^{m+1}_{j,i}$: item 1) of Lemma~\ref{backward} details that
we have finished the proof for two of them and item 2) specifies the conditions for further analysis.

So we need only study case 2 of Lemma~\ref{backward}.
We describe the case $m+1=2$ to illuminate a major issue in the remainder of the proof.

\begin{claim}\label{inv2}
 Assume that
$A^2_{j,1}$ is $G_{\{I\}}$-invariant, then $\mu^{2}_j = 2$ and $A^2_{j,1}$ is safe.
\end{claim}

\begin{proof}
If some $A^2_{j,1}$ is $G_{\{I\}}$-invariant and $\mu^{2}_j \ge 3$, Claim~\ref{backward}.2  gives that
 $A^1_{i,1}$ is $G_{\{ I\}}$-invariant for some $i$;
this contradicts Lemma~\ref{cl0} (${\rm moves}_1$). So, $\mu^{2}_j \le 2$.
Since $\delta(B)\ge d(B)\ge  2$ by $\sdim_1$, we obtain that $\mu^{2}_j = 2$
and ${\rm sdim}_2$ follows from Lemma~\ref{one-copynew}.
\end{proof}

The difficulty is that this argument depended on every $A^1_{j,i}$ being moved; not merely being
safe. In order to deal with this, we
introduce a new system of indexing which is expounded more fully in Notation~\ref{desseq}.
Note that a $G_{\{ I\}}$-invariant $\overleftarrow{A}^0$ engenders by
 Lemma~\ref{long}
 a decreasing sequence of $G_{\{ I\}}$-invariant petals  $\overleftarrow{A}^i$ such that
 $\overleftarrow{A}^i$ determines $\overleftarrow{A}^{i+1}$ which continues as long as
$\overleftarrow{\mu}{^i} = \mu(\overleftarrow{A}^i/\overleftarrow{N}^i) \geq 3$. However,
  we know that no petal $A^1_{u,v}$ is $G_{\{ I\}}$-invariant.
 So this sequence must terminate with an $s\leq m-1$ such that $\overleftarrow{\mu}{^s} =2$.
 We begin the study of such
 sequences with the case  $\overleftarrow{\mu}^1 = 2$, where the chain has only
  two levels:{\em The next five Lemmas, \ref{B+0} through \ref{quickstop}, complete the proof when $\mu^m_i =2$.}
We apply the technical Lemma~\ref{B+0} in the
 proof of Lemmas~\ref{intersection-empty} and ~\ref{controlC}.

\begin{lemma}\label{B+0} Let $D,E \subseteq_\omega M$ satisfy $\delta(D) =0$; then $\delta(D/E) \leq 0$.
Thus, if $E\leq M$ then $\delta(D/E) =0$.
\end{lemma}
\begin{proof}
Monotonicity of $\delta$ implies the first inequality
$$ 
 \delta (D /E)
 \le \delta (D / E \cap D) =\delta
(D) - \delta (E \cap D) = -\delta (E \cap D)
\le 0
$$
and the second equality holds since any subset of $M$ has non-negative dimension.
\end{proof}

Lemma~\ref{quickstop}, Claim~\ref{relationarg1}, Lemma~\ref{seqred} and Lemma~\ref{slowstop}
rely indirectly on
the following ostensibly technical claim about the location of $A$,  with $A$
in Claim~\ref{intersection-empty} getting different interpretations.
It is the crucial point
that allows us to anchor (Definition~\ref{desseq}.4) our inductive analysis in $\acl(\emptyset) \cup B$ where $B$ is the base of
good pair  rather than
the $\Ascr_r$, where the sequence in Definition~\ref{desseq} stops.

\begin{claim}
\label{intersection-empty}
Let $A$ and $B$ be disjoint finite subsets of $M$,
with $A$  good over $B$ and $\delta(B) = d(B) \ge 1$.
Then $A\cap \acl(\emptyset) = \emptyset$.
\end{claim}

\begin{proof}
Let $d \in A \cap \acl (\emptyset)$ and
$D= \icl(d)$.
Then $\delta (D) = 0$.
Suppose for contradiction that $D \subseteq A$.
Since $B \leq M$, Lemma~\ref{B+0} implies $0 \le \delta (D / B) \le \delta
(D / \emptyset) = 0$;
 this contradicts the definition of a good pair, as $D$ must equal
 $A$ and then,  since $\delta(A)- \delta(D) =0$,
$A$ is $0$-primitive over $\emptyset \subsetneq B$.

Let $D_0 = D\cap B$ and $D_1 = D\cap (B\cup A)$.  Since,
$D,B$ and $B\cup A$ are all strong in $M$, so are $D_0$ and $D_1$.
So, $\delta(D) = 0$ implies $ \delta(D_0)
 = \delta(D_1) = 0$.
By Lemma~\ref{B+0} $0 = \delta(D_1/B)$. Repeating the reasoning of the first
paragraph with $D_1$ playing the role of $D$, we obtain the same contradiction; so,
$A$ is $0$-primitive over $\emptyset$.
\end{proof}

\begin{lemma}\label{controlC}  Suppose $\mu^{m+1}_j\ge 3$,
 $A^{m+1}_{j,1}$ is $G_{\{I\}}$-invariant
and  determines $A^m_{i,1}$.
Then $C^{m+1, q} \subseteq A^m_{i,1} \cup \icl(B^m_i)$ for each $q$ with $1\leq q \leq \nu^{m+1}$, i.e. $C^{m+1, q}_+ \subseteq \icl(B^m_i)$.
\end{lemma}

\begin{proof}
Let $B$ denote $B^{m+1}_j$, $\hat B$ denote $B^m_{i}$, $\hat A$ denote $A^m_{i,1}$. 
 We write $C$, for a fixed but arbitrary $q$,
 $C =C^{m+1, q}$, and $C_- = C^{m+1, q}_- = (C^{m+1, q} - \Ascr^{m-1})$.
First we show $C_-=  C^{m+1,q} \cap (\Ascr^m- \Ascr^{m-1})$  is contained in $\hat A$.

Assume that $C_-$
intersects some other petal $A'$ on strata $m$.
By monotonicity and since $C$ is 0-primitive over $B$, taking into account $B\cup C \subseteq \Ascr^m$,
$$(*)\ \ \delta ((C\cap A')/ \Ascr^{m} - A') \leq \delta ((C\cap A') / (B  \cup (C - A'))) < 0.$$

But, also
$$ (**)\ \ \delta((C\cap A')/ \Ascr^{m} -A') = \delta((C\cap A')/ \Ascr^{m-1})  \ge 0.$$
(**) holds because $R(A',\Ascr^{m}-A') = R(A',\Ascr^{m-1})$
as all petals in $\Ascr^{m}-\Ascr^{m-1}$ are fully independent over $\Ascr^{m-1}$
and because $\Ascr^{m-1} \leq M$.
But (*) and (**) are contradictory, so $C_- \subseteq A^m_{i,1}= \hat A$.

\medskip
Our goal is to show   $C^{m+1, q}_+ \subseteq \icl(\hat B)$, where $\hat B = B^m_i$.
So, suppose for contradiction that for some $q$ with  $1 \leq q \leq \nu^{m+1}$
where $\nu^{m+1} = \mu^{m+1}-1$, and with
 $C_+ = C^{m+1,q}_+ = C^{m+1,q} \cap \Ascr^{m-1}$, we have
$ C_+ - \icl(\hat B)\neq \emptyset$.
Since  $C/B$  is a good pair and
$C = (C_+ - \icl(\hat B)) \cup (C_+\cap \icl(\hat B)) \cup C_- $:
\begin{equation}\label{cm*-1}
0  >  \delta ((C_+ - \icl(\hat B))/B\cup C_-
   \cup (C_+\cap \icl(\hat B))).
\end{equation}


\begin{claim}\label{relationarg1}
 Let $\overline{B}_+ = \icl(B_+)$. Inequality~\eqref{cm*-1} simplifies to:

\begin{equation}\label{cm*-2}
\delta ((C_+- \icl(\hat B))/B\cup C_-
   \cup (C_+\cap \icl(\hat B)))  =
\delta (C_+-\icl(\hat B)/B_+ \cup (C_+\cap \icl(\hat B)))
\end{equation}
and
\begin{equation}\label{over-icl-B}
0 >
\delta (C_+-\icl(\hat B)/\overline{B}_+ \cup (C_+\cap \icl(\hat B)))
\end{equation}	
\end{claim}

\begin{proof}
Since $B_- \cup C_- \subseteq A^m_{i,1}$,  each relation
between $B_-\cup C_-$ and $\Ascr^{m-1}$ is a relation on $B_-\cup C_-$ as a subset of $\hat A$
and the  base, $\hat B$, of $\hat A$. So we can delete $B_-\cup C_-$ from the
base of Equation~\eqref{cm*-1} and obtain Equation~\eqref{cm*-2}.
 By Lemma~\ref{long} $\delta(B_+)\le 1$.
Then Lemma~\ref{backward}.2 implies that $B_+\subseteq\acl(\emptyset)$,
so $\delta(\overline{B}_+) = \delta(\icl(B_+)) = 0$.
By Claim~\ref{intersection-empty}, $C\cap \acl(\emptyset) = \emptyset$, then $C_+\cap \overline{B}_+$ = 0.
Monotonicity of $\delta$ implies that
$$
\delta (C_+-\icl(\hat B)/B_+ \cup (C_+\cap \icl(\hat B))) \ge \delta (C_+-\icl(\hat B)/\overline{B}_+ \cup (C_+\cap \icl(\hat B)))
$$
The last inequality and inequality~\eqref{cm*-1} yield inequality~\eqref{over-icl-B}.
\end{proof}

\begin{remark}  In the special case that $B_+  = \emptyset$, the Lemma~\ref{controlC} is easy.
By monotonicity of $\delta$ and by $\icl(\hat B)\le M$ we obtain
\begin{equation*}
\delta (C_+-\icl(\hat B)/C_+\cap \icl(\hat B)) \ge
\delta (C_+-\icl(\hat B)/\icl(\hat B))\ge 0
\end{equation*}	
The last contradicts \eqref{cm*-1}.
\end{remark}

Verification of Equation~\eqref{cm*-1} in the general situation of
 Claim~\ref{relationarg1} requires a further technical calculation.

 More generally,
we only know  from the proof of Claim~\ref{relationarg1} that
 $\delta(\overline{B}_+) = 0$ so we must consider more carefully the connections of
$C_+$ and $\overline{B}_+ $.

We apply the identity
($\delta(X/Y\cup Z) = \delta(X\cup Y/Z) - \delta (Y/Z)$)
by putting $X= C_+-\icl(\hat B)$, $Y = (\overline{B}_+-\icl(\hat B))$, and
$Z =  (\overline{B}_+\cap \icl(\hat B)) \cup (C_+\cap \icl(\hat B)) = (\overline{B}_+ \cup C_+) \cap \icl(\hat B)$.
Thinking of $\overline{B}_+$ as $(\overline{B}_+-\icl(\hat B)) \cup(\overline{B}_+\cap \icl(\hat B))$,
we rewrite the right hand side of \eqref{cm*-2} as follows:
\begin{eqnarray}
\label{cm*-3}
\hspace{16mm}
\delta (C_+-\icl(\hat B)/\overline{B}_+ \cup (C_+\cap \icl(\hat B)))= \\\label{12}
=\delta ((C_+ -\icl(\hat B))\cup (\overline{B}_+-\icl(\hat B))/ (C_+ \cup \overline{B}_+)\cap
     \icl(\hat B)) - \\\label{13} -
\delta (\overline{B}_+- \icl(\hat B)/(C_+\cup \overline{B}_+)\cap\icl(\hat B)).
\end{eqnarray}

Now we show the subtracted term, $\delta(Y/Z)$ (Line~\eqref{13}) is $0$.
We apply monotonicity of $\delta$ in Line~\eqref{3.6} and Lemma~\ref{B+0} in Line~\eqref{3.66}.

\begin{eqnarray}\label{cm*-3.5}
\delta (\overline{B}_+-\icl(\hat B) /(C_+ \cup \overline{B}_+)\cap\icl(\hat B))
& \le & \delta (\overline{B}_+-\icl(\hat B)/\overline{B}_+\cap\icl(\hat B))=\label{3.6} \\
& = & \delta (\overline{B}_+/\overline{B}_+\cap \icl(\hat B))  \le 0 \label{3.66}
\end{eqnarray}

On the other hand, applying monotonicity of $\delta$ and $\icl(\hat B)\le M$ we obtain
$$
\delta (\overline{B}_+-\icl(\hat B) /(C_+ \cup \overline{B}_+)\cap\icl(\hat B))
 \ge  \delta (\overline{B}_+-\icl(\hat B)/\icl(\hat B)) \ge 0
$$
So,  Line~\eqref{13} is $0$. By Line~\eqref{cm*-2} and Claim~\ref{relationarg1},
 Line~\eqref{cm*-3} is negative, so
Line~\eqref{12} is negative, too.
Below we sequentially apply the distributive law, monotonicity of $\delta$ and $\icl(\hat B)\le M$ to Line~\eqref{12}.
\begin{multline*}
\delta ((C_+ -\icl(\hat B))\cup (\overline{B}_+-\icl(\hat B))/ (C_+ \cup \overline{B}_+)\cap \icl(\hat B)) = \\ =
\delta ((C_+ \cup \overline{B}_+)-\icl(\hat B) / (C_+ \cup \overline{B}_+)\cap \icl(\hat B)) \ge \\ \ge
\delta ((C_+ \cup \overline{B}_+)-\icl(\hat B) / \icl(\hat B)) \ge 0
\end{multline*}
The  contradiction obtained with Equation~\eqref{cm*-1}	 completes
the proof of Lemma~\ref{controlC}.
\end{proof}

We avoid the subscripts and isolate in Lemma~\ref{rhoAinv1} the connections imposed by determination,
which drive the proof, and  to emphasize that these results do not require any inductive hypotheses.

In combination with
 Lemma~\ref{intersection-empty} (which is used in
Lemma~\ref{quickstop}), Claim~\ref{relationarg1}, Lemma~\ref{seqred} and Lemma~\ref{slowstop},
we now extend
Lemma~\ref{rhoAinv} from petals to flowers. While this larger set being $G$-invariant is {\em a priori} weaker,
we recover the result for petals by a substantial induction.

\begin{lemma}\label{rhoAinv1}
Let $A$ be well-placed over $\Dscr$ by $B$ and
 $\hat A$ be well-placed over $\hat \Dscr \leq \Dscr $ by $\hat B$. Further
 suppose that $A$ is $G_{\{I\}}$-invariant, $A$ determines $\hat A$, and
 $\mu(\hat A/\hat B) =2$.  Further, let $A,C^1, \ldots C^\nu$ list {\em the}
(by Lemma~\ref{flower=bouquet1})
flower associated with  $A/B$.
Let $\rho$ be a partial isomorphism from
 $\hat A$ over $\hat B$ to its unique copy
 $\hat C\subseteq \hat \Dscr$.  Then
\begin{enumerate}
\item $\rho$ extends to an automorphism of $M$.
\item there is a unique $G_{\{I\}}$-invariant flower
over $\rho(C)/\rho(B)$ and $\rho(A)$ is $G_{\{I\}}$-invariant.
\end{enumerate}
\end{lemma}

\begin{proof}
i) In fact, we will make the extension of $\rho$ fix
 $\acl(\emptyset) \cap \Ascr^0$.
Let $\Wscr_1 = (\Ascr^{0}\cap \acl(\emptyset))\cup \hat B \cup \hat A)$
and $\widetilde \Wscr = (\Ascr^{0}\cap \acl(\emptyset)) \cup \hat B \cup \hat C) \subseteq \Dscr$.
Then, by Claim~\ref{intersection-empty}, with $\hat C= \rho(\hat A)$
 playing the role of $A$,
$\hat C \cap \acl(\emptyset) = \emptyset$. Now, since $\mu(\hat A/\hat B) =2$, $\hat A$ and $\hat C$
are isomorphic by $\rho$ not only over $\hat B$ but over
 $\Wscr_2 =(\Ascr^{0}\cap \acl(\emptyset))\cup
\hat B$. (We know $R(\hat A , \hat \Dscr) = R(\hat A,\hat  B)$. So if
the isomorphism is not over $({ \Ascr}^{0}\cap \acl(\emptyset)) \cup \hat B$,
  there is a  relation between $\hat C$ and
$( {\Ascr}^{0}\cap \acl(\emptyset))-   \hat B$. But then $\delta( (\Ascr^{0}\cap
\acl(\emptyset))/ \hat  C  \cup\hat  B) < 0$, contradicting $\hat C  \cup
\hat B \leq M$.) Note that $\Wscr_2\le M$ because
\begin{multline*}
2 = d(\hat B) \le d(\hat B \cup(\Ascr^0\cap \acl(\emptyset))) \le
\delta(\hat B \cup(\Ascr^0\cap \acl(\emptyset)))  \\
\le
\delta(\hat B ) - \delta(\hat B \cap\Ascr^0\cap \acl(\emptyset))
+ \delta (\Ascr^0\cap\acl(\emptyset)) = 2-0+0 = 2
\end{multline*}
We obtain the first zero because $\hat B \cap\Ascr^0\cap \acl(\emptyset) \le M$.

 As  $\Wscr_2\le M$,
  $\rho$ extends
 to an automorphism of $M$ fixing $\Wscr_2$, also denoted $\rho$.

ii) By Lemma~\ref{rhoAinv}, $\rho(B)$ is $G_{\{I\}}$-invariant.
Suppose $\pi \in G_{\{I\}}$, and so fixes $B$ setwise.
By Lemma~\ref{long}.A, $B  \subseteq \hat A \cup(\acl(\emptyset)\cap \Ascr^0))$
and so
$$\rho(B ) \subseteq \rho(\hat A  \cup (\acl(\emptyset)\cap \Ascr^0)) = \hat C \cup{}
(\acl(\emptyset)\cap \Ascr^0).$$
Put
$$\check \pi = \pi\myrestriction  
\hat \Dscr
\cup
((\rho^{-1}\myrestriction \Hat C)\circ(\pi\myrestriction \Hat C)\circ(\rho\myrestriction \hat A)).$$
Since $
{\hat {\Dscr}}
\cup \hat A \le M$,
we can extend $\check \pi$   to
$\pi'\in G_{\{ I\}}$.
By Lemma~\ref{flower=bouquet1}, there is a unique flower $\Fscr$ over $A/B$.
 $\pi'$ maps it to a flower $\rho(\Fscr)$ over $\rho(C)/\rho(B)$, which must also be unique. As,
 $\pi'^{-1}$ of a second flower  over $\rho(C)/\rho(B)$ would contradict the uniqueness of $\Fscr$.

Since $A, C^1, \ldots C^\nu$ enumerate the $G_{\{I\}}$-invariant flower $\Fscr$,  its
$G_{\{I\}}$-invariant-image (by Lemma~\ref{flower=bouquet1}.2) under $\rho$ is
$\{\rho(A), \rho(C^1), \ldots \rho(C^\nu)\}$.
Since $B^m_i \leq M$,  Lemma~\ref{controlC} implies all the
 $C^q \subseteq \hat C \cup \hat B$.
But then, since $\pi$  fixes both $\hat B = B^m_i$ and $\hat C$
 (since $\mu(\hat A/ \hat B) =2$),
 each $\rho(C^q)$ is contained in $\hat C \cup \hat B$ which is $G_{\{I\}}$-invariant and $\rho(A) \cap
(\hat C \cup \hat B) = \emptyset$. So $\rho(A)$ is $G_{\{I\}}$-invariant as  the complement within
the $G_{\{I\}}$-invariant flower $\rho(\Fscr)$ of the   set of the $\rho(C^q)$ that are contained in $\hat C \cup \hat B$.
\end{proof}

We continue the inductive proof of $\sdim_m$ for all $m < m_0$.
The idea is similar to Claim~\ref{omni}.2, where with $|A^{m+1}_{j,1}| =1$
 we have constructed $\tilde \Ascr$,   a counterexample with
  smaller decomposition height,
  but the situation is more complicated.
We have reduced to the case where $\mu^{m+1}_j \ge 3$ and some  for $m'< m$, $\mu^{m'}_i= 2$.
Building on Lemma~\ref{rhoAinv},
we first consider the special case when
$\mu^m_i = 2$.
Note that by Lemma~\ref{omni2}.3 the hypothesis $|A^{m+1}_{j,1}| >1$  is essential.

\begin{claim}\label{quickstop} Suppose $A^{m+1}_{j,1}$ is $G_{\{I\}}$-invariant,
determines $A^{m}_{i,1}$,
and in addition that
$|A^{m+1}_{j,1}| >1$ and $\mu^m_i =2$. Then, $A^{m+1}_{j,1}$ is safe.
\end{claim}

\begin{proof} Recall that we are doing a global induction to show $\Ascr$ satisfies $\sdim$. The next
constructions will allow us to show $A^{m+1}_{j,1}$ is safe by finding an isomorphic copy of it
with lower height.

Lemma~\ref{long} and Lemma~\ref{backward} imply that since
$A^{m+1}_{j,1}$ is not a singleton and is $G_{\{ I\}}$-invariant,
$\mu^{m+1}_j \ge 3$ implies $B^{m+1}_j \subseteq A^m_{i,1} \cup
(\acl(\emptyset)\cap \Ascr^0)$.
By Lemma~\ref{rhoAinv} the
bouquet of $A^m_{i,1}$ over $B^m_i$ is $G_{\{ I\}}$-invariant.
Since both $A^{m+1}_{j,1}$ and $A^{m}_{i,1}$ are $G_{\{ I\}}$-invariant,
Lemma~\ref{flower=bouquet1} implies
the bouquets  of $A^{m+1}_{j,1}$ over $B^{m+1}_{j}$ and
of $A^m_{i,1}$ over $B^m_i$ are each $G_{\{ I\}}$-invariant.
The second of these consists of two petals:
$A^m_{i,1}$ and $C^{m,1}$. Since the bouquet and  $A^m_{i,1}$ are $G_{\{ I\}}$-invariant
so is $C^{m,1}$.

  We now apply Lemma~\ref{rhoAinv1}, taking  $A$ as
$A^{m+1}_{j,1}$,  $\hat A$ as
$A^m_{i,1}$, $B$ as $B^{m+1}_j$, $\hat B$ as $B^m_i$, $C^i$ as $C^{m+1}_{j,i}$,  $\hat C$ as
 $C^{m,1}$, $\Dscr$ as $\Ascr^{m-1}$, and $\rho$ as constructed in Lemma~\ref{rhoAinv1}
  to conclude:  $\rho(B^{m+1}_j)$ is $G_{\{I\}}$-invariant,
the flower over $\rho(C)/\hat B$  is $G_{\{I\}}$-invariant, and $\rho(A)$ is $G_{\{I\}}$
invariant.

Using the notation of Lemma~\ref{rhoAinv1} and Claim~\ref{quickstop},
 we illustrate the location of $\rho(B^{m+1}_j) = \rho(B)$
 and $\rho(C^{m+1}_{j,x}) = \rho(C^x)$, for $x = 1, 2$,  in Diagram~\ref{Claim4419}.
 So, $\rho(B) = B_+ \cup \rho(B_-)$.
We denote $\rho(B_-)$ by $\widetilde{B}_-$ in Figure~\ref{Claim4419}.
   As in Figure~\ref{Notation320} we draw only two petals
$A^m_{i,1}$ and $A^m_{i,2}$ from $\Ascr^m$  and $A^{m+1}_{j,1}$
from $\Ascr^{m+1}$.
 While, for simplicity, $B^m_i\subseteq \Ascr^{m-1}$ is omitted from the diagram, we
 clarify its properties. Since $B^m_i$ is $G_{\{ I \}}$-invariant and is safe
  by $\sdim_m$,  general
properties of the  construction yield
$$2 \le d(B^m_i) \le \delta (B^m_i) \le \mu(A^m_{i,1}/B^m_i) = 2.$$
This implies $B^m_i \le M$  and $B^m_i = \icl(B^m_i)$.
Lemma~\ref{controlC} implies, writing $C^x_+$ for
$C^{m+1}_{j,x} \cap \Ascr^{m-1}$, that
$C^x \subseteq A^m_{i,1} \cup \icl(B^m_i)$; more precisely,
 $C^x_+ \subseteq \icl(B^m_i) = B^m_i$ and $C^x_- \subseteq A^m_{i,1}$.
 Since $\rho$ fixes $B^m_i$ pointwise,
$\rho(C^x) = \rho(C^x_+ \cup C^x_-) = C^x_+ \cup \rho(C^x_-)$.
We denote $\rho(C^x_-)$ by $\widetilde{C}^x_-$.
By construction, $\rho$ moves $A^m_{i,1}$ to $C^{m,1}$. So, any subset of
$A^m_{i,1}$ is moved to $C^{m,1}$, in particular $C^x_-$.

Obviously, $\rho(B)$, $\rho(C^1)$, and $\rho(C^2)$ are subsets of $\Ascr^{m-1}$, so the third petal over $\rho(B)$ will be on at most $m$-th strata
of some $G_{\{I\}}$-normal set, as we show it below.

\begin{center}
\begin{picture}(370,160)
\put(6, 15){\dashbox{1}(40,140)[c]{~}}
\put(8, 0){$\Ascr^{0}$}
\put(46, 15){\dashbox{1}(20,140)[c]{~}}
\put(50,81.5){$\cdots$}
\put(50, 0){$\cdots$}
\put(66, 15){\dashbox{1}(80,140)[c]{~}}
\put(75, 0){$\Ascr^{m-2}$}
\put(146, 85){\dashbox{1}(70,70)[c]{~}}
\put(190, 146){$A^{m-1}_{1,2}$}
\put(146, 15){\dashbox{1}(70,70)[c]{~}}
\put(190, 19){$A^{m-1}_{1,1}$}
\put(152, 0){$\Ascr^{m-1}$}
\put(216, 15){\dashbox{1}(70,140)[c]{~}}
\put(225, 0){$\Ascr^{m}$}
\put(266, 19){$A^{m}_{i,1}$}
\put(286, 15){\dashbox{1}(70,140)[c]{~}}
\put(295, 0){$\Ascr^{m+1}$}%
\put(26, 85){\oval(30,30)}
\put(20,81.5){$B_+$}
\put(251, 85){\oval(30,30)}
\put(245,81.5){{$B_-$}}
\put(295,47){\line(1,0){50}}
\put(345,47){\line(0,1){70}}
\put(335,117){\line(1,0){10}}
\put(335,67){\line(0,1){50}}
\put(295,47){\line(0,1){20}}
\put(295,67){\line(1,0){40}}
\put(298, 36){$A = A^{m+1}_{j,1}$}
\put(180,35){\line(1,0){80}}
\put(190,55){\line(0,1){50}}
\put(180,105){\line(1,0){10}}
\put(180,35){\line(0,1){70}}
\put(260,35){\line(0,1){20}}
\put(190,55){\line(1,0){70}}
\put(197,42){$C^1_+$}
\put(227,42){$C^1_-$}
\put(263,42){$C^1$}
\put(195,135){\line(1,0){65}}
\put(205,65){\line(0,1){50}}
\put(195,65){\line(1,0){10}}
\put(195,65){\line(0,1){70}}
\put(260,115){\line(0,1){20}}
\put(205,115){\line(1,0){55}}
\put(197,122){$C^2_+$}
\put(227,122){$C^2_-$}
\put(263,122){$C^2$}
\put(86, 35){\dashbox{5}(40,100)[c]{~}}
\put(106, 85){\oval(24,24)}
\put(100,81.5){{$\widetilde{B}_-$}}
\put(98, 110){\dashbox{2}(28,20)[c]{$\widetilde{C}^2_-$}}
\put(98, 40){\dashbox{2}(28,20)[c]{$\widetilde{C}^1_-$}}
\put(90,140){$C^{m,1}$}
\end{picture}
\captionof{figure}{Illustrating Claim~\ref{quickstop}}\label{Claim4419}
\end{center}
In Figure~\ref{Claim4419} we abbreviate as follows: $A = A^{m+1}_{j,1}$, $B = B^{m+1}_{j}$,
$C^x= C^{m+1,x}$, $\widetilde C^x_- = \rho(C^{m+1,x}_-)$ and $\widetilde B_- = \rho(B_-)$.

 Clearly
$\chi_M(\rho(A^{m+1}_{j,1})/\rho(B^{m+1}_{j}) \leq \mu^{m+1}_j =
 \mu(A^{m+1}_{j,1}/B^{m+1}_{j}) $.
The $\rho(C^{m+1,q})$ gives us $\mu^{m+1}_j-1$ witnesses.
Since $\rho(B^{m+1,q})   \subseteq \Ascr^{m-1}\leq M$, $\rho(A^{m+1}_{j,1})$
cannot split over $\Ascr^{m-1}$ (Definition~\ref{splitdef}). Similarly $\Ascr \leq M$ implies  $\rho(A^{m+1}_{j,1})$
cannot split over $\Ascr^{m-1}$.  We now have three cases depending on the exact
location of $\rho(A^{m+1}_{j,1})$.

Case 1. $\rho(A^{m+1}_{j,1}) \subseteq \Ascr^{m-1}$.
Immediately,
the induction hypothesis $\sdim_m$ (in fact, $\sdim_{m-1}$) implies
$\rho(A^{m+1}_{j,1})$ is safe.

To complete the proof, we show  an extension of Lemma~\ref{one-copynew}.

Case 2: $\rho(A^{m+1}_{j,1}) \subset \Ascr$ and  $\rho(A^{m+1}_{j,1}) \cap \Ascr^{m-1}=
\emptyset$:
Since $\Ascr^{m-1} \leq M$, we must have
$\delta(\rho(A^{m+1}_{j,1})/\Ascr^{m-1}) = 0$.  But
$(\rho(A^{m+1}_{j,1})/\rho(B^{m+1}_{j}))$ is a good pair.
 We know ${(B^{m+1}_{j})}_{-} \subseteq A^m_{1,1}$ so
$(\rho(B^{m+1}_{j}))_- \subseteq C^{m,1} \subseteq \Ascr^{m-1}$.
And $(\rho(B^{m+1}_{j}))_+ \subseteq \Ascr^0$. So $\rho(B^{m+1}_{j}) \subseteq
\Ascr^{m-1}$.
Since $\Ascr$ is $G_{\{I\}}$-normal and $\rho(A^{m+1}_{j,1})$ is
well-placed by $\rho(B)$ over $\Ascr^{m-1}$ the construction  places
$\rho(A^{m+1}_{j,1})$ in $\Ascr^m$. So by $\sdim_m$, $\rho(A^{m+1}_{j,1})$ is
 safe.

Case 3. $\rho(A^{m+1}_{j,1}) \subset M -\Ascr$:
Then we put $\tilde \Ascr^m = \Ascr^{m-1}\cup \rho(A^{m+1}_{j,1})$.
Note that $\tilde \Ascr^m$ is a $G_{\{I\}}$-normal with height $m$.
Applying the global induction hypothesis $\sdim_m$ to $\tilde \Ascr^m$,
we see $\rho(A^{m+1}_{j,1})$ is safe.

Thus, in each case $\rho(A^{m+1}_{j,1})$ is safe.
So, by Lemma~\ref{one-copynew}.3, $A^{m+1}_{j,1}$ is safe.
\end{proof}

We have finished if  the descending sequence described before Lemma~\ref{B+0} stops immediately;
we now consider the alternative.
Because `$A^{m+1}_{j,1}$ determines $A^{m}_{i,1}$' produces
 a decreasing chain of complicated
sub/superscripts, we introduce a notation
for  a descending sequence, which is
 relative to a given $G_{\{I\}}$-invariant petal $A^{m+1}_{j,1}$,
and describes the `root' below $A^{m+1}_{j,1}$ that controls its intersection with $\sdcl^*(I)$.
Recall that capital Roman letters ($A,B$) denote petals, while script letters $\Ascr$ denote initial segments of a tree-decomposition.

\begin{definition}\label{desseq}[Determined Sequences]
We write $\Upsilon(A)$ for the petal determined (Definition~\ref{gendef}) by $A$.
Then $\Upsilon^k(A)$ denotes the kth iteration of this operation.
\begin{enumerate}
\item As usual, $\Ascr^m = \bigcup_{i\leq m}\Ascr^i$.

\item Fix  $\Aback^0, \Aback^1$ such that  $\Aback^0 = A^{m+1}_{j,1}$ determines
   $\Upsilon(\Aback^0) = \Aback^1 =A^m_{i,1}$ (Definition~\ref{gendef}).
 For fixed $q$, $\overleftarrow {C}^{0,q} = C^{m+1,q}$, $\overleftarrow {B}^0 = B^{m+1}_j$.

The crucial inductive definition is
\[
    \Aback ^{k+1}=
    \begin{cases*}
      \Upsilon(\Aback^k), \hspace{.18in} \mbox{if }  \mu(\Aback^k, \overleftarrow {B}^k) \geq 3  \\
      stop, \hspace{.35in} \mbox{if } \mu(\Aback^k, \overleftarrow {B}^k) =2
      \end{cases*}
\]

So, $\Aback^k = A^{m+1-k}_{t_k,1}$ for some $t_k$ for each  $k\le s$.
Increment indices for $B, C, \mu$ in the same way.  E.g. $\overleftarrow{C}^{k,q}$ is
$C^{m+1-k}_{j,q}$ in the notation for decompositions.
\item The {\em order} of $\Aback^0 = A^{m+1}_{j,1}$
 is the least index $s$ such that
$\mu(\Aback^{s}, \overleftarrow {B}^{s}) = 2$.
\item Suppose the order of $\Aback^0$ is $s$.  We work from the bottom to define the {\em root}
that supports $\Aback^0$.
$$\mathcal{W}_{s+1} = (\Ascr^0\cap\acl(\emptyset))\cup \Bback^s.$$

For $k\leq s$, we define $\mathcal{W}_{k}$ by downward induction.
$$\mathcal{W}_{k} = \mathcal{W}_{s+1}\cup \Aback^{s} \cup \dots \cup\Aback^{k}= \mathcal{W}_{k+1} \cup \Aback^{k}
\subseteq \overleftarrow {\Ascr}^{k}.$$
\end{enumerate}
\end{definition}

Recall that increasing the superscript of an $\overleftarrow {\Ascr}$ moves to
lower strata. Since we are analyzing
$A^{m+1}_{j,1}$, for any $k < m+1$, $\overleftarrow {\Ascr}^{0}= \Ascr^{m+1}$,
$\overleftarrow {\Ascr}^{k+1}$ is the initial
segment
preceding $\overleftarrow {\Ascr}^{k}$ in the original decomposition. In particular,
 $$\overleftarrow {\Ascr}^{s+1} =\Ascr^{m+1-(s+1)} = \Ascr^{m-s}= \bigcup_{k\leq m-s}\Ascr^{k}.$$
Also, the $\Cback^{k,q} = C^{m+1-k,q}_{t_k} \subseteq \overleftarrow
 \Ascr^{k+1}$ are isomorphic over $\Bback^{k}$
	copies of $\Aback^{k}$.
	
Since $\mu(\Aback^s, \Bback^s)=2$ and $\Bback^s$ is safe,
$2\le d(\Bback^s) \le \delta(\Bback^s) \le \mu(\Aback^s, \Bback^s)=2$. So the next lemma
is easy.
\begin{lemma}\label{W-is-2} Suppose the sequence $\langle \Aback^k:  0 \leq k \leq s\rangle$ stops
with $\mu(\Aback^s/\Bback^s) =2$, then
$\delta(\mathcal{W}_k)=2$ for each $k\le s+1$.
\end{lemma}

\begin{proof}
Since $\mathcal{W}_{s+1} = (\Ascr^0\cap\acl(\emptyset))\cup \Bback^s\le M$
and $\delta(\Bback^s)=2$,  $\delta(\mathcal{W}_{s+1})=2$.
Then we can finish by induction,
because at each step we consider a 0-primitive extension.
\end{proof}

\begin{lemma}\label{seqred} Suppose the sequence  $\langle \Aback^k:  0 \leq k \leq s\rangle$ stops
with $\mu(\Aback^s/\Bback^s) =2$ then

\begin{enumerate}
\item each $\Wscr_k \leq M$ and is $G$-invariant;
\item $\Cback^{k,q} \subseteq  \Wscr_{k+1}$ for every
$q \in \{ 1, \dots, \mu(\Aback^0/\Bback^0)-1\}$.
\end{enumerate}
\end{lemma}

\begin{proof}
1)  To start the induction, note $\Bback^s \leq M$ since $\mu(\Aback^s/\Bback^s) =2$.
$G_{\{I\}}$-invariance follows from the definition of determined, noting that $\Aback^k$
is $G_{\{I\}}$-invariant by Lemma~\ref{long} as $\mu(\Aback^{k-1}/\mu(\Bback^{k-1})\geq 3$.
But each $\Wscr_{k+1} \leq \Wscr_{k}$ since all   have dimension $2$.

2) By Lemma~\ref{controlC}, for each $k\leq s$, $q \leq \mu(\Aback^k/\Bback^k)$,
$\Cback^{k,q} \subseteq \Aback^{k+1} \cup \icl(\Bback^{k+1})$. Since each
$\Bback^{k+1} \subseteq \Wscr_k \leq M$, this implies $\Cback^{k,q} \subseteq \Wscr_{k+1}$.
\end{proof}

Having dealt with the case $\mu^m_i = 2$,  we consider the general case of
Lemma~\ref{quickstop}.
The key difficulty is that we cannot deduce $\rho(B)$ is
$G_{\{I\}}$-invariant in one step as in Lemma~\ref{rhoAinv}.
We have a sequence that stops with an $\Aback^s$ such that $\mu(\Aback^s/\Bback^s) =2$
so that there is an automorphism  $\rho$ mapping $\Aback^s$ into $\Cback^s$. With this $\rho$ fixed
we argue inductively that each $\Aback^k$ for $s \geq k \geq 0$ is safe.

But, we must perform a
dual induction with the proof that $\rho(A)$ is $G_{\{I\}}$-invariant.

\begin{lemma}\label{slowstop}
Suppose the sequence
$\langle \Aback^k:  0 \leq k \leq s\rangle$ stops
with $\mu(\Aback^s/\Bback^s) =2$.
Then $\Aback^k$ is safe for each $k\leq s$. In particular, when $k=0$, we see
$A^{m+1}_{j,1}$ is safe.
\end{lemma}

\begin{proof}
We use Definition~\ref{desseq} of $\mathcal{W}_k$.
 By Lemma~\ref{intersection-empty}, as used in Lemma~\ref{quickstop}, fix
 an automorphism $\rho$
   of $M$,
that sends $\Aback^{s }$ to its unique copy $\Cback^{s}$
and which fixes $\mathcal{W}_{s+1}= (\Ascr^{0}\cap \acl(\emptyset))\cup \Bback^s$
pointwise.
Recall  $\mathcal{W}_s = \mathcal{W}_{s+1} \cup \Aback^s$ and
that $\Cback^s \subseteq \overleftarrow {\Ascr}^{s+1} = \Ascr^{m-s}$.
Let  $\widetilde{\mathcal{W}}_s = \mathcal{W}_s \cup \Cback^s =
 (\Ascr^{0}\cap \acl(\emptyset))\cup \Bback^s \cup \Cback^s$.
Then $\rho(\mathcal{W}_s) = \widetilde{\mathcal{W}}_s$.
For $k \leq s+1 $, we build on Definition~\ref{desseq}.4 of
$\mathcal{W}_{k}$.
We define
$$
\widetilde{\mathcal{W}}_{k} = \widetilde{\mathcal{W}}_{s} \cup \rho(\Aback^{s-1}) \cup
\dots \cup \rho(\Aback^k).$$
Note that $\widetilde{\mathcal{W}}_{k}$ need not be contained in $\Ascr$.
In particular, $\mathcal{W}_{s} \subseteq \overleftarrow {\Ascr}^{s}$ while
$\widetilde{\mathcal{W}}_{s} - \overleftarrow {\Ascr}^{s+1} = \rho(\Aback^{s-1})$.

Finally, to obtain a $G_{\{I\}}$-normal structure with a well-defined height, we define:
 $$\widetilde{\mathcal{R}}_{k} = \overleftarrow {\Ascr}^{s+1} \cup \rho(\Aback^{s}) \cup
\dots \cup \rho(\Aback^k).$$
Note that the height of $\widetilde{\mathcal{R}}_{s}
 = \overleftarrow {\Ascr}^{s+1} \cup \rho(\Aback^{s})$ is $m-s$ because $\rho(\Aback^{s}) \subseteq \overleftarrow {\Ascr}^{s+1}$ and $\overleftarrow\Ascr^{s+1} = \Ascr^{m-s}$. Moving from $\widetilde{\mathcal{R}}_{k+1}$ to $\widetilde{\mathcal{R}}_{k}$
increases the height at most by $1$; that is why
the height of $\widetilde{\mathcal{R}}_{0}$ is at most $m$.
Since $\Aback^k$ determines $\Aback^{k+1}$,
Lemma~\ref{backward}.2 implies that $\Bback^{k} \subseteq (\Ascr^0\cap\acl(\emptyset)) \cup \Aback^{k+1}$:
obviously, then,
$$ (*) \
 \rho(\Bback^{k}) \subseteq \widetilde{\mathcal{W}}_{k+1} =
\widetilde{\mathcal{W}}_{s} \cup \rho(\Aback^{s}) \cup
\dots \cup \rho(\Aback^{k+1}) \subseteq \widetilde{\mathcal{R}}_{k}.$$

By Claim~\ref{seqred} $\Cback^{k,q} \subseteq \mathcal{W}_{k+1}$.
Thus, $\rho(\Cback^{k,q}) \subseteq \widetilde{\mathcal{W}}_{k+1}$.
Obviously, $\rho(\Aback^k) \cap \widetilde{\mathcal{W}}_{k+1} = \emptyset$
because $\Aback^k \cap \mathcal{W}_{k+1} = \emptyset$.
By Lemma~\ref{flower=bouquet1} the bouquet of $\rho(\Aback^k)$ over $\rho(\Bback^k)$
consists just of one flower. 

We conclude Lemma~\ref{slowstop} from the following,  which we show  below for each $k \leq s$:
\begin{enumerate}

\item $\rho(\Bback^k)$ is $G_{\{I\}}$-invariant;
      \item
      $ \Aback ^{s+1}\cup \widetilde{\mathcal{W}}_{k}$, $\widetilde{\mathcal{W}}_{k}$,
and $\rho(\Aback^k)$ are $G_{\{I\}}$-invariant.
\end{enumerate}

We prove these two assertions by simultaneous induction on $k$.
The induction is downward from $s$ and the base step is the third paragraph of
the proof of Lemma~\ref{quickstop}. So, we assume that (1)--(2) hold for $k+1$
and show that they hold for $k$.

(1) Recall that by Lemma~\ref{long}.1 and Lemma~\ref{backward}.2, $\Bback^k \subseteq \Aback^{k+1}
 \cup (\Ascr^0\cap\acl(\emptyset))$.
So, $$\rho(\Bback^{k}) \subseteq \rho(\Aback^{k+1}
 \cup (\Ascr^0\cap\acl(\emptyset))) = \rho(\Aback^{k+1}) \cup (\Ascr^0\cap\acl(\emptyset)).$$
We consider  $\rho(\Bback^{k})\cap (\Ascr^0\cap\acl(\emptyset))$   and
$\rho(\Bback^{k})\cap \rho(\Aback^{k+1})$ separately.
Since $\rho$ fixes $\Ascr^0\cap\acl(\emptyset)$ pointwise
and $\Bback^k$ is $G_{\{I\}}$-invariant,
$\rho(\Bback^{k}\cap \Ascr^0\cap\acl(\emptyset))$ is $G_{\{I\}}$-invariant.

We show an arbitrary $\pi \in G_{\{I\}}$ fixes $\rho(\Bback^{k})\cap  \rho(\Aback^{k+1})$ setwise.
By the induction hypothesis, $\rho(\Aback^{k+1})$ is $G_{\{I\}}$-invariant,
so $\pi(\rho(\Bback^{k})\cap  \rho(\Aback^{k+1})) \subseteq \rho(\Aback^{k+1})$.
Now we put
$$\tau = (\pi\myrestriction \overleftarrow\Ascr^{k+2})
\cup (\rho^{-1}\circ \pi \circ \rho)\myrestriction \Aback^{k+1}.$$
Obviously,  this isomorphism extends to an automorphism from $G_{\{I\}}$.
Since $\Bback^{k}$ is $G_{\{I\}}$-invariant,
$\tau(\Bback^k) = \rho^{-1}\circ \pi \circ \rho(\Bback^k) = \Bback^k$;
so $\pi(\rho(\Bback^{k})) = \rho \circ \pi(\Bback^{k}) = \rho(\Bback^{k})$.

(2) By the induction hypotheses
$\overleftarrow \Ascr^s \cup \widetilde{\mathcal{W}}_{k+1}$, $\widetilde{\mathcal{W}}_{k+1}$,
and $\rho(\Aback^{k+1})$ are $G_{\{I\}}$-invariant. Whence, by (1) $\rho(\Bback^k)$
is $G_{\{I\}}$-invariant.
Since $\widetilde{\mathcal{W}}_{k} = \widetilde{\mathcal{W}}_{k+1}\cup \rho(\Aback^k)$,
it is sufficient to prove that $\rho(\Aback^k)$ is $G_{\{I\}}$-invariant
 to deduce that $\overleftarrow \Ascr^s \cup \widetilde{\mathcal{W}}_{k}$
and $\widetilde{\mathcal{W}}_{k}$  are $G_{\{I\}}$-invariant.
So, we consider $\rho(\Aback^k)$. In fact, we repeat some reasoning from Lemma~\ref{quickstop}.
We put $\overleftarrow \nu^{k} = \mu(\Aback^k/\Bback^k) - 1$.

Case 2a) $\rho(\Aback^k)\subseteq \overleftarrow \Ascr^s\cup \widetilde{\mathcal{W}}_{k+1}$:
Since $\rho(\Bback^k)$
is $G_{\{I\}}$-invariant, so is
$$ (**)\   
\bigcup_{q=1}^{\overleftarrow \nu^{k}} \rho(\Cback^{k,q}) \cup
\rho (\Aback^k)$$
because  it is the flower of $\rho (\Aback^k)$ over
$\rho_0(\Bback^k)$.
Clearly, each $\Cback^{k,q}$ intersects $\Aback^{k+1}$;
{\em otherwise}, there must be
$c_1, c_2 \in \Cback^{k,q} \subseteq \overleftarrow \Ascr^s\cup{\mathcal{W}}_{  k+2}$
and $b\in \Bback^k \cap \Aback^{k+1}$ with $R(b, c_1, c_2)$. The last implies that
$|\Aback^{k+1}| =1$, for a contradiction. Thus, each $\rho(\Cback^{k, q})$
intersects $\rho(\Aback^{k+1})$, which is $G_{\{I\}}$-invariant by the induction hypothesis.
Let $\tau \in G_{\{ I\}}$ be arbitrary. Since $\rho(\Bback^{k})$ is $G_{\{I\}}$-invariant
 $\tau (\rho(\Bback^{k})) =
\rho(\Bback^k)$ and so
$$\tau(\rho(\Aback^{k})) \in
\{ \rho(\Cback^{k, q}) : q = 1, \dots, \overleftarrow {\nu}^{k}\}  \cup \{\rho(\Aback^{k})\}.$$
By construction, $\Aback^{k} \cap \Aback^{k+1} = \emptyset$;
so $\rho(\Aback^k)  \cap \rho(\Aback^{k+1}) = \emptyset$
and $\tau(\rho(\Aback^k))$ does not intersect $\tau(\rho(\Aback^{k+1})) = \rho(\Aback^{k+1})$.
But, we showed in the last paragraph $\rho(\Cback^{k,q}) \cap \rho(\Aback^{k+1}) \ne \emptyset$, so
$\tau(\rho (\Aback^{k}))$ cannot be equal to any of the $\rho(\Cback^{k,q})$.
Hence, using (**), $\tau(\rho (\Aback^k)) = \rho (\Aback^k)$ and $\rho (\Aback^k)$
 is $G_{\{I\}}$-invariant.

 Since $\rho(\Aback^k)\subseteq \overleftarrow \Ascr^s\cup \widetilde{\mathcal{W}}_{k+1}$,
 by the global induction hypothesis $\rho(\Aback^k)$ is safe.

Case 2b) $\rho(\Aback^k)\not\subseteq \overleftarrow  \Ascr^s\cup \widetilde{\mathcal{W}}_{k+1}$:
As $\rho(\Aback^k)$ is a $0$-primitive extension of
$\overleftarrow  \Ascr^s\cup \widetilde{\mathcal{W}}_{k+1}$,
$\rho(\Aback^k)\cap (\overleftarrow  \Ascr^s\cup \widetilde{\mathcal{W}}_{k+1}) = \emptyset$.
By (*), $\rho(\Bback^k) \subseteq \overleftarrow  \Ascr^s\cup \widetilde{\mathcal{W}}_{k+1}$;
moreover $\rho(\Bback^k) \cap  \rho(\Aback^{k+1}) \ne \emptyset$.

\begin{claim}\label{notnew} In case 2b, for each $k <s$, $\rho(\Aback^{k })\subseteq \overleftarrow  \Ascr^{k+1} -
\overleftarrow  \Ascr^{k+2}$ or $\rho(\Aback^{k})\cap \Ascr = \emptyset$.
In either case $\overleftarrow  \Ascr^{k+1} \cup \widetilde{\mathcal{W}}_{k+2}$ is
$G_{\{I\}}$-normal.
\end{claim}

\begin{proof} We have $\rho(\Bback^{k }) \subseteq \rho(\Aback^{k+1}) = \overleftarrow {C}^{k+1}
 \subseteq \overleftarrow {\Ascr}^{k+2}\cup \widetilde{\mathcal{W}}_{k-1} $,
$\rho(\Aback^{k})$ is $0$-primitive over $\rho(\Bback^{k})$, and $\overleftarrow {\Ascr}^{k+1}\leq M$.
Thus, $\rho(\Aback^{k})$ is $0$-primitive over $\overleftarrow {\Ascr}^{k+1}$ and based on
$\rho(\Bback^{k})\subseteq \overleftarrow{\Ascr}^{k+1}$.  If $\rho(\Aback^{k}) \subseteq \Ascr$
by, construction, $\rho(\Aback^{k})\subseteq \overleftarrow {\Ascr}^{k+1 }-\overleftarrow {\Ascr}^{k+2}$.  If not,
 since $\rho(\Aback^{k })$ cannot split (Definition~\ref{splitdef}) over $\Ascr$,
 $\Ascr \cap \rho(\Aback^{k }) =\emptyset$ and so $\overleftarrow{\Ascr}^{k+1}
 \cup \{\widetilde{\mathcal{W}}_{k}\}$
 is $G_{\{I\}}$-normal.
\end{proof}

Since $\overleftarrow {\nu}^{k}$ copies of $\rho(\Aback^k)$ over its base
are inside $\widetilde{\mathcal{W}}_{k+1}$,
 $\rho(\Aback^k)$ is $G_{\{I\}}$-invariant.
This completes the proof of Case 2b.

Since $\rho(A^{m+1}_{j,1}) = \rho(\Aback^{0})\subseteq
\widetilde{\mathcal{W}}_0 \subseteq \widetilde{\mathcal{R}}_{0}$
 and the height of $\widetilde{\mathcal{R}}_{0}$ is at most $m$,
  by the global induction $\rho(A^{m+1}_{j,1})$ is safe; by
Lemma~\ref{one-copynew}.2, we conclude $A^{m+1}_{j,1}$ is safe.
We finish Lemma~\ref{slowstop}. \end{proof}

This completes the proof of Lemma~\ref{d(orbit)=2}  showing
 $\sdim_m$ for $m\leq m_0$; thus we have
the main conclusion, Theorem~\ref{no-sym-fun}.

\section{Steiner Systems}\label{Steiner}
\numberwithin{theorem}{section}

In this section we study the strongly minimal $k$-Steiner systems discovered in \cite{BaldwinPao}.
A $k$-Steiner system is a collections of points and lines so that two points determine
a line and all lines have the same finite length $k$.  A quasigroup (binary
operation with unique solutions of $ax =b$ and $xa =b$)  such
that every $2$-generated sub-quasigroup
has $k$ elements determines a $k$-Steiner system where the lines are the two generated subalgebras.
Our interest in the existence of definable truly binary functions arose from
 the discovery that while
a Steiner system with  line length three admits  a quasigroup operation definable in
the vocabulary of the ternary collinarity
predicate and Steiner systems with prime power length admit the imposition of
quasigroups  that preserve lines (e.g. \cite{GanWer2}), it seemed very unlikely in the second case
 that
those quasigroups were {\em definable from $R$} (\cite{BaldwinsmssII}).

There are two examples of strongly minimal $3$-Steiner systems in \cite{BaldwinPao} and  \cite[Section 5]{Hrustrongmin}.
By explicitly adding
multiplication to the vocabulary, \cite{BaldwinsmssII} constructs strongly minimal
 quasigroups which determine
$k$-Steiner systems for each prime power $k$.
We   show  below that this separate operation is essential.
 The following problem/example inspired this research and is solved here.

\begin{problem}\label{defqg} {\rm We can impose a quasigroup structure on any
$4$-Steiner system.   There are two obvious
ways: one commutative, one not \cite{BaldwinsmssII}.  In fact, \cite{GanWer}
a quasi-group can be imposed in any Steiner $k$-system when $k= p^n$ for a prime $p$.
1) Prove the operations
of these quasi-groups are not $R$-definable in a strongly minimal $4$-Steiner system $(M,R)$.
2) More generally, is there an $\emptyset$-definable truly binary function?}
\end{problem}

We   now use  $\bK$ rather than $\bL$ to emphasize the distinctions from Section~\ref{main}.
Having said that, $\bK^*=\bL^*$ and $\bK^*_0=\bL^*_0$ while $\bK _0\neq \bL_0$ .
We work in a vocabulary $\tau$ with one ternary relation $R$, and
assume always that $R$ can hold only of three distinct elements and then in
any order (a $3$-hypergraph). The basic definitions are in Section~\ref{HLS}.
In the language
of $*$-petals,
$\mu$
{\em triples} if for every {\em non-linear}
(Definition~\ref{decomp2}) $*$-petal
 $(C/B)$ with $\delta(B) = 2$ and $|C|>1$, $\mu(C/B)\geq 3$.

\begin{theorem}\label{mrs} Let $M \models T^S_\mu$ be a strongly minimal Steiner system
 described in Notation~\ref{defT}. Then
 \begin{enumerate}
 \item  The naturally imposed quasigroups
  (\cite{GanWer}) on $M$ are not $\emptyset$-definable in $M$.
\item If $\mu(\boldsymbol{\alpha}) \geq 2$ and $\mu$
{\em triples},
 then there is no $\emptyset$-definable truly binary function in $T_\mu^S$.
\item There is no symmetric $\emptyset$-definable truly $v$-ary function for $v\ge 2$,
i.e., $\sdcl^*(I) = \emptyset$ for any $v$-element independent set $I$.
 \end{enumerate}
\end{theorem}

After a short introduction establishing 1), we prove 2) and 3). The major
obstacles to adapting the  earlier proofs of these results in the Hrushovski case
are a) the need to modify the notion of base (Lemma~\ref{primchar}) and b)
 the analysis of distinct occurrences of $R$ (e.g. Lemma~\ref{small-G}).

The following example (Figure~\ref{fig:steinerce}) shows that as in the Hrushovski case, we must pass
to $G_{\{I\}}$ and strengthen the hypothesis to get $\sdcl^*(I) = \emptyset$. Definable truly binary functions
may appear when $\mu(A/B)=2$ and $d(B) =2$ is allowed.
We put the following lines: $\{a_1, d_2, d_1\}$, $\{a_1, d_4, d_5\}$,
$\{a_2, d_5, d_3, d_1\}$, and $\{d_2, d_3, d_4\}$. The elements $c_i$
is the isomorphic copy of $d_i$ over $\{a_1, a_2\}$, for each $i$.
In order to construct $A^2_{1,1}$ we make $\alpha_i$ a copy of $a_i$,
 $\delta_i$ a copy of $d_i$ and each
$\gamma_i$ a copy of $c_i$ for each appropriate $i$,
where the isomorphism under consideration is over $\{d_3, c_3\}$. Then $\alpha_1\in\dcl^*(a_1, a_2)$.

\begin{figure}[h]\label{example-S}
\begin{center}
\begin{minipage}[h]{0.34\linewidth}
\begin{center}
\begin{picture}(151,185)
\put(43, 6){$a_1$} \put(133, 6){$a_2$} \put(45, 15){\circle*{3}} \put(135, 15){\circle*{3}} \put(0,0){\dashbox{1}(150,30)[c]{~}}
\put(0,32){\dashbox{1}(89,60)[c]{~}}
\qbezier(45,15)(30,45)(15,75)
\qbezier(45,15)(60,43.2)(75,71.4)
\qbezier(135,15)(52,112)(15,75)
\qbezier(23,60)(29.5,110)(69,60)
\put(15, 75){\circle*{3}}\put(9, 65){$d_1$}
\put(75, 71.4){\circle*{3}}\put(72, 62){$d_5$}
\put(23, 60){\circle*{3}}\put(17, 50){$d_2$}
\put(37, 85){\circle*{3}}\put(31, 76){$d_3$}
\put(69, 60){\circle*{3}}\put(67, 50.6){$d_4$}
\end{picture}
\end{center}
\end{minipage}
\hfill
\begin{minipage}[h]{0.56\linewidth}
\begin{center}
\begin{picture}(196,185)
\put(58, 6){$a_1$} \put(148, 6){$a_2$} \put(60, 15){\circle*{3}} \put(150, 15){\circle*{3}} \put(15,0){\dashbox{1}(180,30)[c]{~}} \put(0,8){$\Ascr^0$}
\put(15,32){\dashbox{1}(89,60)[c]{~}}
\put(106,32){\dashbox{1}(89,60)[c]{~}}
\qbezier(60,15)(45,45)(30,75)
\qbezier(60,15)(75,43.2)(90,71.4)
\qbezier(150,15)(67,112)(30,75)
\qbezier(38,60)(44.5,110)(84,60)
\put(30, 75){\circle*{3}}\put(24, 65){$d_1$}
\put(90, 71.4){\circle*{3}}\put(87, 62){$d_5$}
\put(38, 60){\circle*{3}}\put(32, 50){$d_2$}
\put(52, 85){\circle*{3}}\put(46, 76){$d_3$}
\put(84, 60){\circle*{3}}\put(81, 50.6){$d_4$}
\put(0,40){$\Ascr^1$}
\qbezier(60,15)(120,43.2)(180,71.4)
\qbezier(60,15)(95,45)(130,75)
\qbezier(150,15)(222,94)(130,75)
\qbezier(112.4,60)(174,97)(156,60)
\put(180, 71.4){\circle*{3}}\put(182, 68){$c_5$}
\put(130, 75){\circle*{3}}\put(132, 80){$c_1$}
\put(112.4, 60){\circle*{3}}\put(114.5, 57){$c_2$}
\put(154, 78.5){\circle*{3}}\put(151, 71.8){$c_3$}
\put(156, 60){\circle*{3}}\put(162, 57){$c_4$}
\put(15,94){\dashbox{1}(180, 90)[c]{~}}
\put(58, 174){$\alpha_1$} \put(148, 174){$\alpha_2$} \put(60, 169){\circle*{3}} \put(150, 169){\circle*{3}} \put(0,102){$\Ascr^2$}
\qbezier(60,169)(45,140)(30,111)
\qbezier(60,169)(75,137)(90,104.8)
\qbezier(150,169)(45,39)(30,111)
\qbezier(38,126)(45,45)(80.1,126)
\put(30, 111){\circle*{3}}\put(22, 116){$\delta_1$}
\put(90, 104.8){\circle*{3}}\put(78, 107){$\delta_5$}
\put(38, 126){\circle*{3}}\put(30, 131){$\delta_2$}
\put(80.1, 126){\circle*{3}}\put(78, 131){$\delta_4$}
\qbezier(60,169)(120,146)(180,123)
\qbezier(60,169)(95,140)(130,111)
\qbezier(180,123)(156,40)(130,111)\qbezier(150,169)(190,159)(180,123)
\qbezier(116.6,122)(164,33)(172,126)
\put(180, 123){\circle*{3}}\put(182, 123){$\gamma_5$}
\put(130, 111){\circle*{3}}\put(132, 111){$\gamma_1$}
\put(116.6, 122){\circle*{3}}\put(118.6, 122){$\gamma_2$}
\put(172, 126){\circle*{3}}\put(161, 121){$\gamma_4$}
\put(20, 170){$A^2_{1,1}$}
\end{picture}
\end{center}
\end{minipage}
\end{center}
\captionof{figure}{Example of  $\dcl^*(\{a_1,a_2\})\ne \emptyset$}\label{fig:steinerce} 
\end{figure}

We defined linear spaces and the appropriate $\delta$ for studying
 them in Definition~\ref{twodelta}.
In \cite[Lemma 3.7]{BaldwinPao} we showed that this $\delta$ is flat, submodular, and computes
 exactly on free products defined as in Definition~\ref{defcanam}. Thus, the notion of decomposition
 and the arguments for  the basic properties of the standard Hrushovski construction
 in earlier sections go through below with minor changes.
 However, Lemma~\ref{alphpetals} shows
 some significant differences in the resulting decomposition.  This finer analysis of the
 decomposition, which is the chief novelty of this section,
 powers the understanding of definable closure in these Steiner systems.

Recall from Conclusion~\ref{conclude2}:
 for each $3 \leq k < \omega$, there are continuum-many strongly minimal
  infinite linear spaces in the vocabulary $\tau$ that are Steiner $k$-systems.
A crucial invariant for these systems is `line length'. The length of each line in a model of the Steiner system
 is $\mu(\boldsymbol{\alpha}) + 2$ where $\boldsymbol{\alpha}$ is from Notation~\ref{linelength}. However,
 there may be maximal cliques in a substructure $A$ with smaller cardinality.
We refer to such configurations
 as {\em partial lines}; a line of length $\mu(\boldsymbol{\alpha}) + 2$ may be called {\em full} for emphasis.
Following a convention established in \cite{BaldwinPao}, we think of two independent points as
a {\em trivial (and therefore partial) line}.

 The Hrushovski restraint in defining the $\mu$-function:
an integer $\mu(\beta) = \mu(A/B) \geq \delta(B)$
was relaxed in \cite{BaldwinPao} to hold  only when  $|A-B|\geq 2$. To
allow lines of length
three, we required only  $\mu(\boldsymbol{\beta)} \geq 1$,
if $\boldsymbol{\beta} = \boldsymbol{\alpha}$.
Thus for the case when $\mu(\boldsymbol{\alpha})=1$ we got a strongly minimal Steiner system
with lines of length three. Obviously, there is a definable symmetric truly binary function $H$ on
pairs of distinct elements; $H(x,y)$ is the third point on the line determined by $x$ and $y$ and
$H(x,x) =x$. So we restrict here to lines of length at least four.

\begin{assumption} $\mu(\boldsymbol{\alpha})\geq 2$.
\end{assumption}

With longer line length $k$ one can always introduce a $k$-ary {\em partial function}
saying its value on $k-1$
distinct elements is the remaining point on the line. But, now there is no clear
 way to give a uniform definition of
 a $k$-ary {\em function}
on sequences with repetition.  With the following variant on the results
in Section~\ref{main},
we show there is no such truly binary  function in the vocabulary: $\{R\}$.

As noted in Remark~\ref{minmax} the original Hrushovski construction
 supports a happy coincidence.
The minimal subset $B$ of $D$ (the base: Definition~\ref{primchar1})
 such that $A$ is $0$-primitive over $B$ is also
the maximal subset such that
 every element of $B$ is $R$ related to some element of $A$.
But for linear spaces,
the two notions diverge. 
Allowing for and exploiting  this difference is one of the two  major changes
 from the proof
for the Hrushovski construction in Section~\ref{main}.
We recast \cite[Lemma 4.8]{BaldwinPao} (Lemma~\ref{primchar1}) in case 2) of the next
lemma.

\begin{lemma} \label{primchar}
 Let $D \leq D \cup A \in \bK_0$ be a $0$-primitive extension with $D \cap A =\emptyset$.
 Then there are two cases:
\begin{enumerate}
\item  If $A=  \{ a \}$ there is a unique  line $\ell$ with $\ell\cap D\geq 2$.
    In that case, any $B
   \subseteq (\ell\cap D)$ with $|B'|=2$ yields
    a good pair $(B,a)$. Furthermore, $d \in D$ is in the relation $R$
  with  the  element $a$ if and only if $d$ is on $\ell$.
\item  If $|A| \geq 2$ then there is a unique maximal subset $B$ of
    $D$ with every point $b$ in $B$ incident with a line $\ell_b$ with
    $|\ell_b \cap A| \geq 2$ containing $b$.

 \end{enumerate}
\end{lemma}

On the basis of Lemma~\ref{primchar} we add the new notion of {\em extended base}.

\begin{definition}\label{basedef}  Let $A$ be a $0$-primitive extension of $D$ (in $M$),
 where we assume that $D\cap
A = \emptyset$.
If $A =\{a\}$, then
the {\em extended base} for $A$ is the maximal set $\check B =\ell \cap D$
    where $\ell$ is the line through $b_1,b_2$ for any elements $b_1,
    b_2 \in D$ such that $R(b_1,b_2, a)$.
Note that if $A = A^{m+1}_{j,i}$ and $D = \Ascr^m$ the extended base for $A$ is
 $B^{m+1}_j = \{d\in \Ascr^{m }- \Ascr^{m-1} : R(b_1,b_2,d)\}$
for any $b_1,b_2 \in \Ascr^m$ with $R(b_1,b_2,a)$.
\end{definition}

If $A^{m+1}_{j,i}=\{a\}$ is $0$-primitive over $\Ascr$ with extended base $B=B^{m+1}_{j}$,
any two element subset $B_0$ of $B$ can act as a base. If $\ell_{m+j} < \mu(\boldsymbol\alpha)$, the $C^{m+1,q}$
must be mapped into $B-B_0$.

Definition~\ref{decomp2}
will be clarified by  Lemma~\ref{alphpetals} showing that all types of $*$-petals
have been described.
\begin{definition}\label{decomp2}
Let $G\in \{G_I, G_{\{I\}}\}$, and let $\Ascr$ be a $G$-normal set.
Fix a decomposition  of 
 $\Ascr$ into strata $\Ascr^m$ constructed inductively as in Construction~\ref{decomp1}.
\begin{enumerate}
\item We say $A= A^{m+1}_{j,1} = \{a\} \in \Ascr - \Ascr^m$ is an {\em $\boldsymbol{\alpha}$-point} if there  exist
$b_1,b_2 \in \Ascr^m$ with $R(b_1,b_2,a)$.

\item A set $A$ is a {\em linear cluster} if $A = \{a\in
\Ascr^{m+1}-\Ascr^{m }:R(b_1,b_2,a)\}$ for some $b_1,b_2 \in \Ascr^m$.
We denote the linear cluster
with extended base $\check B =B^m_f\subseteq \Ascr^{m}$ as $A^{m+1}_f = \bigcup  A^{m+1}_{f,i}$
 where the $A^{m+1}_{f,i}$ are the $\alpha$-petals over $\check B$.

\item A $*$-petal is either an $A^{m+1}_{f,i}$ with cardinality greater than $1$ (called a {\em non-linear petal})
or a linear cluster.

\item We write $\Smoves_m$ if every non-linear petal $A^{m }_{f,k}$ is moved by some $g \in G_I$.

\item Recall that we say $X$ is safe
if $ d(E) \ge 2$ for any $G_{\{I\}}$-invariant set $E\subseteq
  X$ which is not a subset of $\acl(\emptyset)$. The $G_{\{ I\}}$-decomposition $\Ascr^m$ of $\Ascr$
 satisfies $\Ssdim_m$ if every $G_{\{I\}}$-invariant subset of $\Ascr^m$ is safe.
\end{enumerate}
\end{definition}

Now any $\gamma\in G$ that fixes $A$ setwise fixes an extended base set-wise but it
does not need to
fix  a base of an $\alpha$-point even setwise.

\begin{definition}\cite[Lemma 3.14]{BaldwinPao}\label{defcanam} Let $A \cap C =B$
 with $A,B,C \in \bK_0$.
We define $D := A \oplus_{B} C$ as follows:

\begin{enumerate}[(1)]
	\item the domain of $D$ is $A \cup C$;
	\item
a pair of points  $a \in A - B$  and $b \in C - B$ are on a non-trivial
line $\ell'$ in $D$ if and only if there is line $\ell$ based in $B$ such
that $a \in \ell$ (in $A$) and $b \in \ell$ (in $C$). Thus $\ell'=\ell$ (in
$D$).
\end{enumerate}
\end{definition}

\begin{lemma}\label{freeext}\begin{enumerate}
\item If $D \supseteq A \cup {B}\cup C$ where $A$ and $C$ are $0$-primitive over $B$, $B  \leq D$,
 and there is a relation among elements  $a_1 \in A-B$ and $a_2 \in C-B$ then both $|A-B|$ and $|C-B|$ are $1$.
\item \label{partition}
Each $\Ascr^{m}$ is partitioned into $*$-petals
 and
 there is no non-trivial line (even through $\Ascr^{m-1}$) connecting distinct $*$-petals.
That is, the $*$-petals are fully independently joined
\end{enumerate}

 \end{lemma}
\begin{proof} 1) If $R(a_1,a_2,b)$ then $\delta(A/BC) <\delta(A/C)$, unless there is
  a line $\ell\subseteq D$ with  $|\ell \cap B| \geq 2$
that contains both $a_i$. But each $a_i$ is then the only element
of a $0$-primitive extension of $B$.
2) Thus the collection of  $*$-petals (i.e. non-linear petals and the linear clusters $A^{m+1}_f$
of $\alpha$-points) are fully freely joined as a partition of
$\Ascr^{m+1}-\Ascr^m$.
\end{proof}

\begin{lemma}\label{fulllines} Fix a $G$-normal $\Ascr$ and a decomposition of
 height at least $3$, where $G\in \{ G_I, G_{\{ I\}}\}$.
Every non-trivial
partial line $\ell$ in $\Ascr$ is either contained (except for at most one point)
 in a single petal of the topmost strata $\Ascr^{m_0}$ or extends to a full line that intersects
  at most three strata.
\end{lemma}

\begin{proof} Let $m$ be least so that  $\ell$ is based in $\Ascr^m$.
If $|\ell|< \mu(\boldsymbol{\alpha})+2$, adding a new
point in $\Ascr^{m+1}$, that is related only to $\ell \cap \Ascr^m$ is
 a $0$-primitive extension giving an $\boldsymbol{\alpha}$-petal $A^{m+1}_{f,i}$.
 By Corollary~\ref{getmax},
$|\ell \cap \Ascr^{m+1}| = \mu(\boldsymbol{\alpha})+2$.
It is possible that one point of $\ell$, but, by choice of $m$, not two, is in $\Ascr^{m-1}$.
That is,  it may be  $|\ell \cap (\Ascr^{m +1} -  \Ascr^{m-1})| = \mu(A^{m+1}_{j,1}/B^{m+1}_j)
  +1$.
 This is the possibility that intersects three strata.  If $m = m_0$, the line may remain partial
but includes at most one point of $\Ascr^{m-1}$.
\end{proof}

\begin{definition}\label{gendef1} We say a petal $A^{m+1}_{j,1}
\  \mathrm{Steiner\mbox{-}determines}$
 a $*$-petal,
if there is a non-linear petal
 $A^m_{i,f}$ or a linear cluster $A^m_i$ which is the unique $*$-petal
 based in $\Ascr^{m-1}$ that intersects
  $B^{m+1}_j-\Ascr^{m-1}$.  (More precisely,
  $\langle A^{m+1}_{j,1},B^{m+1}_j,\Ascr^m\rangle$
  determines $\langle A^m_{i,f},B^m_i,\Ascr^{m-1}\rangle$.)
\end{definition}

\begin{lemma}\label{alphpetals}
Fix a decomposition of a $G$-normal set $\Ascr$, where $G\in \{G_I, G_{\{ I\}}\}$.
Suppose $A = \{a\}$  is an $\alpha$-point of $\Ascr	^{m+1}$ based on $B=\{b_1,b_2\} \subseteq \Ascr^m$ and a subset
of the linear cluster $A^{m+1}_j$.	
Let $\check B$ be the extended base of $a$ in $\Ascr^m$.
 Then,
\begin{enumerate}
\item If $m=0$, $\{a\}$ is in a linear cluster $A^1_j$ with
 $|A^1_j| = \mu(\boldsymbol{\alpha})-|I|$.
Since $I$ is independent, this is possible only if $|I|=2$.

 \item Let $m>0$. If a linear cluster satisfies $|A^{m+1}_j| = 1$
then $\check B-\Ascr^{m-1}$ is a subset of
one $*$-petal, say,  $A^m_{f,i}$, which is not a linear cluster.
 So, $A^{m+1}_j$ determines $A^m_{f,i}$ in this case.
\item Let $G=G_I$. Then  $\Smoves_m$ implies each $\alpha$-point
 $\{a\}$ over $\Ascr^m$ 
 is moved by $G_I$.
\begin{enumerate}
\item 
$G_{B}$ acts as the symmetric group
$S_{|A^m_f|}$  on a linear petal $A^m_f$ based on $B=\{b_1,b_2\}$.
Thus, $G_I$  moves such $\alpha$-points.

\item $A$ is a line based on $B \subsetneq A^m_{f,i}$ for some $f,i$. By $\Smoves_m$, $A^m_{f,i}$ is moved and {\em a fortiori}
so is $A$.
\end{enumerate}
\item
Let $A^{m+1}_j$ be a linear cluster
which contains at least two elements (that is, at least two $\alpha$-points)
and which is $G$-invariant.
If $d(\check B)\ge 2$ then $d(A^{m+1}_j)= 2$.
\end{enumerate}
\end{lemma}

\begin{proof}
1) Suppose $m = 0$. We have $R(b_1,b_2,a)$; $b_1,b_2$ are algebraically independent; else $a \in \Ascr^0$. Moreover
the definition of $\Ascr^0$ decrees $\neg R(b_1,b_2,b_3)$ for any 3 distinct
  $b_i \in \Ascr^0$.
By Corollary~\ref{getmax},
$\chi_M(\{a\}/\{b_1,b_2\}) = \mu(\boldsymbol{\alpha} )$
  yielding a linear cluster of cardinality $\mu(\boldsymbol{\alpha})-2$.

2) By Lemma~\ref{getmax} the line $\ell$ passing through $a, b_1, b_2$ is full and is equal
to $\check B \cup A^{m+1}_j$.
Then $|\check B| = \mu(\boldsymbol{\alpha}) +2 - |A^{m+1}_j| \ge 3$,
because $\mu(\boldsymbol{\alpha}) \ge 2$ and  $|A^{m+1}_j| = 1$.
By Lemma~\ref{fulllines} $|\check B\cap \Ascr^{m-1}| \le 1$,
so at least two elements of $\check B$ are in $\Ascr^m-\Ascr^{m-1}$.
If these two elements belong to different $*$-petals, then these
$*$-petals are not free over $\Ascr^{m-1}$, for a contradiction. (If there is a point
on the line and in $\Ascr^{m-1}$ or if there are three points in different petals,
the petals are dependent over $\Ascr^{m-1}$.)
Note that $\ell \cap (\Ascr^m-\Ascr^{m-1})$  is not a linear cluster
 because $|\check B\cap\Ascr^{m-1}|\le 1$,
while a base for a linear cluster contains at least 2 elements.

3) Any $\alpha$-point  $e$ is either on a linear petal with size $\geq 2$
 or $\icl(G_I(e))$ intersects  two distinct $*$-petals:

3a)  {\em $|A^{m+1}_j| > 1$ and is a linear cluster:} Then for $k\leq |A^{m+1}_j|$,
 all $k$-sequences from
 $A^{m+1}_j$ realize the same quantifier-free type over $\check B$
(and so over $\Ascr^m$ since $R(A^{m+1}_j,\Ascr^m) = R(A^{m+1}_j,\check B)$,
 so they are automorphic over $\Ascr^m$  in $\Ascr $ since $\Ascr^m A^{m+1}_j\leq \Ascr$.

3b) {\em $|A^{m+1}_j| = 1$:} Then $B^{m+1}_j \subseteq A^m_{f,i}$
which is a non-linear petal and so
$\Smoves_m$ implies $A = A^{m+1}_{j,1}$  is moved by $G_I$.

4)
Since $\check B$  is a partial line, $\delta(\check B) = 2$.
So, $\check B\le M$ because by the hypothesis $d(\check B)\ge 2$.
Then $\icl(A^{m+1}_j) \subseteq A^{m+1}_j \cup \check B$,
because $A^{m+1}_j$ is a $0$-primitive extension of $\check B$ and $\check B \le M$.
Since $|\icl(A^{m+1}_j)|\ge |A^{m+1}_j|\ge 2$
and $\icl(A^{m+1}_j)$ is contained in the line $A^{m+1}_j \cup \check B$,
 $d(A^{m+1}_j) = \delta (\icl(A^{m+1}_j)) =2$.	
\end{proof}

\begin{remark}\label{getfind} {\rm Note that there are $R$-relations  within
 a linear cluster; it lies on one line.
And at least one linear cluster is $G_I$-invariant, the line through $I=\{a,b\}$; others are easy to find.
 But Lemma~\ref{alphpetals} shows
no $\boldsymbol{\alpha}$-point is in $\dcl^*(I)$.  There are partial lines of various lengths in the $\Ascr^{m+1} - \Ascr^m$ that are not linear clusters.
 But each   is within a single non-linear petal (Lemma~\ref{alphpetals}).
(This depends essentially on the decomposition of the ambient $G$-normal $\Ascr$;
every pair of points is contained in a nontrivial line in $M$, but perhaps not in $\Ascr$.)}
\end{remark}

Lemma~\ref{alphpetals}.3a yields immediately the answer to the motivating  Problem~\ref{defqg}.1.
Recall a quasigroup satisfies for all $x$ and $y$, there exist unique $l$ and $r$
such that $lx=y$ and $xr=y$ (the multiplication table is a Latin square).
\cite{GanWer} show that if Steiner system has line-length $k$,
where $k$ is a prime-power,
then it is possible to impose a binary function $*$ on the universe such that:

($\#$) $a,b$, $a*b$ is on the line through $a,b$ and  $*$ is a quasigroup
 such that the restriction of $*$ to each line is generated
by any two elements of the line.

However, this function cannot be definable (without parameters)
in a strongly minimal structure $(M,R)$ studied here.  It suffices to find one line on which
the function is not defined. This is straight forward since any finite configuration is
strongly embedded in $M$. In detail,

\begin{theorem}\label{noqg}
No quasigroup $*$  restricted to each line and satisfying ($\#$) is definable in a
strongly minimal   Steiner system
 from \cite{BaldwinPao} with line length at least four.
\end{theorem}

\begin{proof} Take any independent pair $I =\{a_1,a_2\}$ contained in some $\Ascr^m$ and
suppose they generate the
line $A =\{a_1,a_2, \ldots a_k\}$. Then $A-I\subseteq \Ascr^{m+1}-
\Ascr^m$ is a linear cluster and by Lemma~\ref{alphpetals}.3a,
$G_I$ induces the symmetric group on $A-I$.

Suppose $a_1*a_2 = a_i$ and $a_2*a_1 = a_j$. Choose an element $a_k$ of $A$ distinct from  $a_i$.
There is a $g\in G_I$ with $g(a_i) =a_k$. The definition of a quasigroup is contradicted
unless $a_i =a_j =a_k$; in that event replace $a_k$ with an $a_{k'}$ distinct
 from all $a$'s previously considered; this is easy since $|A|\geq 4$.
\end{proof}

While this solution to Motivating Problem~\ref{defqg}.1 invokes the decomposition,
a more direct argument yields that result in \cite{BaldwinsmssII}.  However, we now show
the much stronger consequence of the decomposition asked for in Problem~\ref{defqg}.2, no
truly $n$-ary function.
For smoother reading, we mention results from Sections~\ref{decomp} and \ref{main}
that go through without any changes and pay attention to those results which requires
some adaptations.

Lemmas~\ref{cl0} and \ref{sum} work for Steiner  systems.
Lemma~\ref{alphpetals}.(2) and (3).(b) yield a stronger version Lemma~\ref{omni-G}.(1):
If $|A^{m+1}_{i,j}| =1$
is $G_I$-invariant then $A^{m+1}_{i,k}$ determines   a $G_I$-invariant
non-linear petal $A^m_{f,i}$.
Multiple realizations of $\boldsymbol{\alpha}$
in $\Ascr^{m+1}-\Ascr^{m}$ represent distinct petals but only one
$*$-petal (linear cluster).  We incorporate the role of  Lemma~\ref{omni-G}.(2)
in proving Lemma~\ref{Am-not-1-G} into  the proof of Lemma~\ref{nonlin}.

Comparing the argument for
Lemma~\ref{omni-G} with Figure~\ref{fig:steiner_line} explains the main
differences between Lemma~\ref{omni-G} for Hrushovski's examples and
Lemma~\ref{nonlin} for Steiner systems.
In Hrushovski's examples we obtain that $b_2$ is in two relations
$R(b_2, c_1^0, c_3^0)$ and $R(b_2, c_1^1, c_3^1)$
with $\Ascr^{m-1}$,
which contradicts $\Ascr^{m-1}\le M$. But, in Steiner systems we have just one line $\ell$,
which contains points from different copies $C^0$ and $C^1$ of $A^{m+1}_{j,1}$.

\begin{figure}[h]
\begin{center}
\begin{picture}(270,170)
\put(30,0){\dashbox{1}(240,70)[c]{~}} \put(0,43){$\Ascr^{m-1}$}
\put(50,20){\dashbox{1}(60,40)[c]{~}}\put(190,20){\dashbox{1}(60,40)[c]{~}}
\put(50, 10){$C^0$}\put(239, 10){$C^1$}
\put(30,70){\dashbox{1}(240,40)[c]{~}} \put(0,88){$\Ascr^{m}$} \put(100,88){$A^{m}_{i,1}$}
\put(120,75){\dashbox{1}(60,30)[c]{~}}
\put(30,110){\dashbox{1}(240,60)[c]{~}} \put(0,138){$\Ascr^{m+1}$}
\put(150, 90){\circle*{3}}\put(147, 80){$b_2$}
\put(110, 10){\circle*{3}}\put(96, 8){$b_1$}
\put(190, 10){\circle*{3}}\put(196, 8){$b_3$}
\qbezier(110,10)(95,30)(80,50)
\qbezier(190,10)(205,30)(220,50)
\qbezier(190,10)(125,20)(60,30)
\qbezier(110,10)(175,20)(240,30)
\qbezier(60,30)(117,90)(150,90)\qbezier(240,30)(183,90)(150,90)
\put(99.5, 24){\circle*{3}}\put(99, 27){$c^0_2$}
\put(60, 30){\circle*{3}}\put(52, 35){$c^0_1$}
\put(80, 50){\circle*{3}}\put(68, 50){$c^0_3$}
\put(200.5, 24){\circle*{3}}\put(191, 27){$c^1_2$}
\put(240, 30){\circle*{3}}\put(238, 35){$c^1_3$}
\put(220, 50){\circle*{3}}\put(222, 50){$c^1_1$}
\qbezier(110,10)(135,85)(160,160)\put(160, 160){\circle*{3}}
\qbezier(190,10)(165,85)(140,160)\put(140, 160){\circle*{3}}\put(150, 130){\circle*{3}}
\qbezier(140,160)(240,165)(150,90)
\put(128,115){\dashbox{1}(44,50)[c]{~}}\put(100,132){$A^{m+1}_{j,1}$}
%
\end{picture}
\end{center}
\captionof{figure}{Example with one line and two $C^d$'s}\label{fig:steiner_line} 
\end{figure}

\begin{lemma}\label{nonlin}
Fix a decomposition of a $G$-normal set $\Ascr$, where $G\in \{G_I, G_{\{ I\}}\}$.
Suppose $B= B^{m+1}_j$ is the base of a non-linear petal $A^{m+1}_{j,1}$ which is
$G$-invariant
and $\ell^{m+1}_j +1 < \mu^{m+1}_j$.
\begin{enumerate}
\item Let $G=G_{I}$ and $\Sdim_m$ hold; or
\item	let $G=G_{\{ I \}}$ and $\Ssdim_m$ hold.
\end{enumerate}	
It is impossible that $B$ has a non-empty intersection with a linear cluster $A^m_f$.
\end{lemma}

\begin{proof}  Suppose for contradiction
that $(B \cap \ell) -\Ascr^{m-1}\neq \emptyset$, witnessed by $b$ for some $\ell$ such
 that $A^m_f \cap \ell \neq \emptyset$  and $\ell\cap \Ascr^{m-1}= B^m_f = B'$,
 the extended base of $b/\Ascr^{m-1}$.

Step 1: We show $B$ contains a single point  $b$ from
$\Ascr^m-\Ascr^{m-1}$.
By Lemma~\ref{primchar}.2, there exist $x_1,x_2 \in A^{m+1}_{j,1}$
with $R(x_1,x_2,b)$.
Since $\ell^{m+1}_j + 1< \mu^{m+1}_j$,
there are (Figure~\ref{fig:steiner_line}) at least two disjoint embeddings $C^0$  and $C^1$ over $B$ of $A^{m+1}_{j,i}$
  into $\Ascr^m$; 
the image $C^i$ must contain copies $c_1^i$ and $c_2^i$ of $x_1$ and $x_2$,
which satisfy $R(c^i_1 , c^i_2 , b)$ for $i< 2$ and are disjoint from $B$.
Without relying on the inductive hypotheses, the proof of
Lemma~\ref{alphpetals}.3.a shows that if the $G$-invariant $B$ intersects
a linear cluster $A^m_f$, $B \cap (\Ascr^{m}-\Ascr^{m-1})$  contains $A^m_f$.
Since the $*$-petals are freely joined, all the $c^i_j$
 are in $\Ascr^{m-1}$. So they must be in $\ell$
since any element in $\Ascr^{m-1}$ related to  $b$ is in $\ell$. 
And $b$ is on a line with at least five elements\footnote{Note
that this situation is impossible unless $\mu(\boldsymbol{\alpha})
 \geq 3$.} based in $\Ascr^{m-1}$.

In fact, $A^m_f = (\ell \cap \Ascr^m)-\Ascr^{m-1} \subseteq B$ must be a singleton.
As, if $b'$ is a second point in $A^m_f$,
 $(\{ b, b'\},c^i_1)$ realizes $\boldsymbol{\alpha}$ with
the base contained in $B$. But this is a
 contradiction, because $A^{m+1}_{j,1}$ is a non-linear petal based on $B$
and $C^i$ is isomorphic to $A^{m+1}_{j,1}$ over $B$.

Step 2: Having shown $B$ contains a single point  $b$ from
$\Ascr^m-\Ascr^{m-1}$ there are two cases. In the first case suppose this $b$
 and so its extended base $B'$ are $G$-invariant.
Since $\{ b\}$ is $G$-invariant but not safe, this contradicts
$\Sdim_m$ or $\Ssdim_m$ depending on $G=G_i$ or $G=G_{\{I\}}$.

We are left with the case that  $A^m_f$ is a singleton but not $G$-invariant, i.e.
there exists $g\in G$ such that $g(A^m_f) \ne A^m_f$, but $|A^m_f| =1$.
Let $b=b_0, b_1, \ldots b_{k-1}$ enumerate the orbit of $b$ under $G$.
Then, for $u<k$ there is a
$g_u \in G$ satisfying  $R(g_u(x_1), g_u(x_2), g_u(b))$  and $g_u(x_1),
g_u(x_2)\in A^{m+1}_{j,1}$, $g_u(b)\in B$
because both $A^{m+1}_{j,1}$ and $B$ are $G$-invariant.
Let $\langle C^i: i<\nu = \mu^{m+1}_j-1\rangle$ enumerate the copies
of $A^{m+1}_{j,1}$ in $\Ascr^m$.
Now, as in step 1, for each $b_u,C^i$, there are elements
$d^{i,u}_1, d^{i,u}_2 \in C^i \cap \Ascr^{m-1}$ satisfying $R(d^{i,u}_1, d^{i,u}_2,b_u)$.
Again as in step 1, all the $C^i \subseteq \Ascr^{m-1}$ and for each $u$
 all the $d^{i,v}_w$ for $w<2, i < \nu$ are on the same line.
Now we consider the substructure
 $\overline{C} =B \cup \bigcup_{v<\nu}C^v$.
 If $C^1$ and $C^2$ are freely joined over $B$, $\delta(C^1 \cup C^2/B) =0$.
For each fixed $b_u$ we have one new line $\ell_u$ with at least five points on it and the nullity of
$\ell \cap (C^1 \cup C^2)$ is $4-2 =2$. As no points are added this reduces $\delta(C^1 \cup C^2/B)$
by $1$. (One line of length 3 in $BC^{2}$ in the computation of $\delta(C^1C^{2}/B)$
  has been replaced by two points added to $\ell_u$.)
	Each additional $C^i$ decrements another $1$ so with respect to $b_u \in B \cap (\Ascr^m-\Ascr^{m-1})$
	the line $\ell_u$ reduces $\delta(\overline{C}/B)$ by $(\nu -1)$.
But there are $k$ such $b_u$ and $\nu = \mu^{m+1}_{j,1} -1$,
so $\delta(\overline{C}/B)
\le k(1-\nu) =k(1-(\mu^{m+1}_{j,1} -1))
= k(2-\mu^{m+1}_j)$. Hence,

$$\delta(\overline{C}\cup B) - \delta(B) = \delta(\overline{C}/B) \le 2k - 2\mu^{m+1}_j \le 2k - k\delta(B).
$$
Consider the first and last terms and move $\delta(B)$ and
$\delta(\overline{C}\cup B)$ to the opposite sides of the inequality; then divide
by $k-1$
to get
$$\delta(B) \le \frac{2k-\delta(\overline{C}\cup B)}{k-1} \le
\frac{2k-2}{k-1} =2.$$
Recall, that $B$ is safe, so $2\le  d(B)\le \delta(B) \le \delta(\overline{C}\cup B) $.
This justifies the second inequality.
Thus, $d(B) = \delta(B) = 2$ and $B\le M$ and all $0$-primitive extensions of $B$ must be
independent; this contradicts  the existence of the lines  $\ell_u$.
\end{proof}

 Lemma~\ref{small-G} 1) and 3)  concern  only non-linear petals
and so goes through without changes. However, a small  new argument is needed for part 2).

\begin{lemma}\label{small-Gstein}
Assume that
  $A^{m+1}_{j,1}$ is $G$-invariant, $|A^{m+1}_{j,1}| > 1$,
	and $|A^m_{i,f}| > 1$ for each $i$, $f$
	such that $A^m_{i,f} \cap B \ne \emptyset$. Then, for any $d$ with $1\leq d \leq \nu =\nu^{m+1}_j$:
\begin{enumerate}[A]
\item \label{Bd-GS} For any $i$, $f$ such that $A^m_{i,f} \cap B \ne
    \emptyset$,   $C^d \cap A^m_{i,f}\ne \emptyset$, i.e., $C^d_-  \neq
    \emptyset$.

\item \label{smR-GS} $\delta(B_-/B_+ \cup \bigcup_{1 \leq d \leq
\nu} C^d_+) = \delta(B_-/B_+)$.
\item\label{shortS} If $C^d\cap \Ascr^{m-1}= \emptyset$, that is $C^d_+ =
    \emptyset$, then there is a unique petal $A^m_{i,f}$ that contains both
    $C^d$ and $B_-$. So, $A^m_{i,f}$ is $G$-invariant.
\end{enumerate}
\end{lemma}

\begin{proof}
A) As in Lemma~\ref{small-G}.1, for each $f,i$ for each $d$,  $A^m_{i,f} \cap B \ne
\emptyset$ implies $C^d \cap A^m_{i,f} \ne \emptyset$. For B) note that
if
 $\delta(B_-/B_+ \cup \bigcup_{1 \leq d \leq
\nu} C^d_+) = \delta(B_-/B_+)$ fails it is because there is a
line $\ell$  with $|\ell| \geq 3$ intersecting
 $B_-$ and $B_+ \cup \bigcup_{1 \leq d \leq
\nu} C^d_+$ with at most one point in $B_+$.
If $|\ell\cap B|=2$, then each $C^d$ is a linear petal.
Since $A^{m+1}_{j,1} \cong_B C^d$, $A^{m+1}_{j,1}$ is also linear; contradiction.
Then $|\ell\cap B|=|\ell\cap B_-| =1$ and $\ell$ is based in $\Ascr^{m-1}$.
Let $\{ b \} = \ell\cap B_-$.
Then $b\in A^m_{i,f}$ for some $i$ and $f$,
and $\{ b \}$ is a linear petal over $\Ascr^{m-1}$,
 contradicting the hypothesis that
$|A^m_{i,f}| > 1$ for each $i$, $f$
	such that $A^m_{i,f} \cap B \ne \emptyset$.  C) follows as in Lemma~\ref{small-G}.3.
\end{proof}

From Lemma~\ref{nonlin}, we know that if $B^{m+1}_j$  is the base of a $G$-invariant $A^{m+1}_{j,1}$,
$B^{m+1}_j$ intersects only non-linear petals.
Lemma~\ref{long} relies on `$\mu$-triples' but involves only calculations justified by the axiomatic properties of
$\delta$, so we can apply it here to conclude:

\begin{corollary}\label{Sdetthm} Fix a decomposition of a $G$-normal set $\Ascr$,
 where $G\in \{G_I, G_{\{ I\}}\}$.
Suppose a non-linear petal $A^{m+1}_{j,1}$ is $G$-invariant.
Assume $\mu$ triples, (so  $\mu^{m+1}_{j}(A/B) \geq 3$ when $B$ is not a singleton).
\begin{enumerate}
\item Let $G=G_{I}$ and $\Sdim_m$ hold; or
\item	let $G=G_{\{ I \}}$ and $\Ssdim_m$ hold.
\end{enumerate}	
Then, $A^{m+1}_{j,1}$ $\Sdet$ a non-linear petal $A^m_{i,1}$.
\end{corollary}

We restate and prove Theorem~\ref{mrs} using essentially the same induction
 as in Section~\ref{geq3}; the difference is that Lemma~\ref{alphpetals} makes the
 treatment of $\alpha$-petals easier while we apply Corollary~\ref{Sdetthm} for determinacy
 of non-linear petals.

\begin{theorem}[no definable truly $n$-ary function]\label{wrapitup2}  Suppose $T^S_\mu$ is a Steiner-system as in Definition~\ref{defT}.
Assume $\mu$ triples. 
Let $I$ be a finite independent set that contains at least 2 elements.
Fix a $G$-normal $\Ascr \leq M \models \hat T_\mu$ with height $m_0$.

Then for every $m \leq m_0$, $\Ascr^m \cap \dcl^*(I) = \emptyset$.

 Thus,
 $\dcl^*(I) \cap \Ascr = \emptyset$; so there is no $n$-ary $\emptyset$-definable
 function for $n\ge 2$ and  $\hat T_\mu$ does not admit elimination of imaginaries.

 As a corollary, we obtain that $\dcl(J) = \bigcup_{a\in J}\dcl(a)$ for any
independent  set $J$.
\end{theorem}

As in Section~\ref{geq3}, the decomposition for Theorem~\ref{wrapitup2} is with
respect to $G_I$.

\begin{proof} We show $\Smoves_m$
  and $\Sdim_{m}$ jointly imply  $\Smoves_{m+1}$. Suppose
  for contradiction that $A^{m+1}_{j,1}$ is $G_I$-invariant.
By Lemma~\ref{alphpetals}.3, we can assume $A^{m+1}_{j,1}$ is non-linear.
Then Corollary~\ref{Sdetthm} implies that
$A^{m+1}_{j,1}$ $\Sdet$ a non-linear petal $A^m_{i,1}$; but this
 contradicts $\Smoves_{m}$.

Fix $m$ with $1 \leq m \leq m_0$. Since Lemma~\ref{petals}
uses only the notions depending on abstract properties of the $\delta$-function:
 $\Smoves_{m+1}$ and $\Sdim_m$ imply $\Sdim_{m+1}$.
Thus by induction as in
Section~\ref{geq3}, we have $\Smoves_m$ for all $m\leq m_0$ and finish.
\end{proof}

Before attacking the symmetric function case in general, we  prove the Steiner
version of Claim~\ref{omni2}, describing the implications of
the existence of a $G_{\{I\}}$-invariant $\alpha$-petal.

\begin{lemma}\label{omni2S}  Let $T^S_\mu$ be a strongly minimal Steiner-system
as  described in Definition~\ref{defT} with $\mu(\boldsymbol{\alpha})> 1$.
Suppose $\sdim_m$, $|A^{m+1}_{j,1}| = 1$ and $A^{m+1}_{j,1}$ is $G_{\{ I\}}$-invariant. Then
\begin{enumerate}
\item
$A^{m+1}_{j,1}$ determines some non-linear
$A^m_{i,1}$; $B=B^{m+1}_j \leq M$, $B$ and each $C^k = C^{m+1,k}$ is contained in  $A^m_{i,1}$.
\item Moreover, $\mu^{m+1}_j\geq 3$;
\item and $\mu^m_i \geq 3$.
\end{enumerate}
\end{lemma}
\begin{proof}Lemma~\ref{alphpetals}.2 shows $A^{m+1}_{j,1}$ determines some non-linear
$A^m_{i,1}$.
The analog of Lemma~\ref{omni2}.1, showing each $C^{m+1,q}
\subseteq A^m_{i,1}\cup B^m_i$,  has both a shorter proof
 and is  stronger.
Let $B= B^{m+1}_j$ be the extended base of $A^{m+1}_{j,1}$.
By Lemma~\ref{fulllines} $B\cap\Ascr^{m-1}$ contains at most one element,
but since this set is $G_{\{ I\}}$-invariant and a one-element set cannot be safe,
$\Ssdim_m$ implies that $B\cap \Ascr^{m-1} = \emptyset$.
Also $\Ssdim_m$ implies $B\leq M$.
So, each $C^{m+1,q}\subseteq B\subseteq A^m_{i,1}$.
2) and 3) now follow exactly by the argument in
Claim~\ref{omni2}.
\end{proof}

As in Section~\ref{geq2}, we now drop the $\mu$-triples requirement and still show there
is no {\em symmetric} definable function.

\begin{theorem}[no definable symmetric function]\label{Steiner-no-sym-fun}
 If $T^S_\mu$ is a Steiner-system as in Definition~\ref{defT},
then there is no symmetric $\emptyset$-definable $v$-ary function for $v\ge 2$,
i.e., $\sdcl^*(I) = \emptyset$ for any $v$-element independent set $I$.

  That is,
 there is no definable function of $v$ variables whose value does not depend on the order of the arguments.
	Thus, $\hat T_\mu$ does not admit elimination of imaginaries.

As a corollary, we obtain that $\sdcl(J) = \bigcup_{a\in J}\sdcl(a)$ for
 any independent set $J$.
\end{theorem}

\begin{proof} We break the proof from Section~\ref{geq2}
into several sections and indicate changes necessary for Steiner systems.

First, note Claim~\ref{dim0}  obviously works for Steiner systems, i.e.
$\Ssdim_0$ holds. The inductive proof of safety of $A^{m+1}_{j,1}$ from $\sdim_m$,
 Claim~\ref{invm}, follows for
non-$G_{\{I\}}$-invariant non-linear petals (or $\alpha$-petals)  from $\delta$-calculations as
in Lemma~\ref{petals}.

We now fix on a $G_{\{I\}}$-invariant $A^{m+1}_{j,1}$ that determines
$G_{\{I\}}$-invariant $A^{m}_{i,1}$.
Our aim is to prove $A^{m+1}_{j,1}$ is safe.
They respectively have $\mu^{m+1}_j$ and $\mu^{m}_i$ realizations in $M$.
Results \ref{grptriv} through \ref{rhoAinv} establish the result when $\mu^{m}_i=2$. This can only happen when
$A^{m+1}_{j,1}$ is non-linear by Lemma~\ref{omni2S}, which is the analog of Lemma~\ref{omni2}.
These results
are properties of automorphisms of finite structures
and hold for the same reasons as in Subsection~\ref{geq2}.

As in Lemma~\ref{backward} we have reduced to the case that $\mu^{m+1}_j  \geq 3$. But with Lemmas~\ref{Sdetthm}
and \ref{omni2S}, while $A^{m+1}_{j,1}$ may be linear, every element of the sequence it determines is non-linear.
Moreover, if $A^{m+1}_{j,1}$ is linear, $B^{m+1}_{j} \subseteq A^{m}_{i,1}$.

The analogs of Lemmas~\ref{B+0} through \ref{quickstop} complete the proof when $\mu^m_i =2$. They go through
in the Steiner case with little change. (Lemma~\ref{omni2S} includes for Steiner systems the more difficult conclusion
in Lemma~\ref{controlC}.)

This leaves us with the analog of Lemmas~\ref{desseq} to \ref{notnew}, which formulate and  carry out the
complicated double induction. But again, one can check that the arguments go through with minor modifications.
\end{proof}

\section{Further work}\label{furtherwork}
We worked throughout in this paper in a vocabulary with a
single ternary relation symbol. We now explain a conjectured   sufficient condition
for the elimination of imaginaries in arbitrary finite and  infinite vocabularies,
 using Hrushovski's $\delta$ and definition of $\bK_0$.

In \cite{Verbovskiy}, the second author constructed
a variant of Hrushovski's example
 with elimination of imaginaries.
The idea is that for each $n\ge 3$ we add an $n$-ary relation $R_n$ and
put $\mu(\{a_1\} / \{a_2, \dots, a_n\}) = 1$,
where the tuple $(a_1, a_2, \dots, a_n)$ satisfies $R_n$.
This gives us an $(n-1)$-ary symmetric function.
Thus we can construct {\em in an infinite vocabulary} a Hrushovski strongly minimal set
which has elimination of imaginaries.
The conjecture is that in some sense it is the only way to get a symmetric function in Hrushovski's examples.
Recall that the constraint $\mu(A/B) \ge \delta(B)$ has a crucial role in proving the amalgamation property.
However, as it was shown in \cite{Verbovskiy06},
for good pairs $(A/B)$ satisfying $r(\{a\}, A\cup B, \{b\})>0$ for  each $a\in A$ and $b\in B$
we may put $\mu(A/B)$ equal to any positive number while
preserving the amalgamation property.
A slight modification should construct definable truly $n$-ary functions.
So, the exact formulation of the conjecture is the following:

\begin{conjecture}\label{mainconjecture}
We take the class  $\bL_0$ to be all finite $\tau$-structures that satisfy the hereditarily
positive $\epsilon$ dimension defined in Axiom~\ref{yax}.2.
Assume that there is a natural number $N$, such that $\mu(A/B) \ge \delta(B)$ for any
 good pair $(A/B)$
with $\delta(B)\ge N$; then $\sdcl^*(I) = \emptyset$ for any independent set
 $I$ with $|I|\ge \max\{N, 5\}$. Thus, no Hrushovski construction in a finite relational
 vocabulary  $\tau$ (that is,
 where $\bK_0$ contains all finite $\tau$-structures)
 has elimination of imaginaries.
 \end{conjecture}

 More generally, one might ask

\begin{question}\label{conj2}  No strongly minimal set in a finite vocabulary with a strictly flat $\acl$-geometry
admits elimination of imaginaries.

{\rm \cite{Merflat} makes a step in this direction by representing each strictly
flat geometry
by a Hrushovski construction.  However, he deals only with the $\omega$-stable case
and takes $\bL^*_0$ as all finite structures of
a given relational vocabulary.  So this question has to be made more precise.
The methods here must be extended as examples in
\cite{BaldwinsmssII,BaldwinsmssIII} show that with two ternary relations one can
define global ternary functions.
  Moreover, the binary function can be commutative\footnote{E.g.,we built built the strongly minimal
  set  from a commutative variety of block algebras over $F_5$ (\cite[p 7]{GanWer2})
  by the method of \cite[\S 3]{BaldwinsmssIII}.
  These examples fall into
  Remark~\ref{classes}.(2).(b).(ii).}.}
	This contradiction with the conclusion of
  Theorem~\ref{mrs}.3  is explained because we have widened our method of construction to include
  amalgamation classes of finite structures, which are defined by $\forall \exists$-formulas.
\end{question}

\begin{question}[Flat Geometries]

We have provided several properties distinguishing among strongly minimal theories with flat geometries and provided
some examples.  Three directions of inquiry are:

\begin{enumerate} \item   Are there further useful distinction among the theories of flat $\acl$-geometries?
\item Are there further useful syntactic distinctions among the theories themselves?  \item Are there further applications
in combinatorics using the methods developed here?
\end{enumerate}
{\rm Linear spaces and quasi-groups are only a glimpse at the structures that can
obtained when we remove the restriction that we are imposing the dimension function on {\em all} finite structures for a given vocabulary.
Moreover as exemplified in \cite{BaldwinsmssIII}, new phenomena are obtained by varying $\mu$.
In a different direction, one can ask whether these methods might be useful higher in the
 complexity classification.}
 \end{question}

 \begin{question}[Model Theoretically  Complex Classes]\label{BarCas}{\rm
\cite{BaldwinsmssII,{BaldwinsmssIII}}, constructs strongly minimal quasigroups using the graph of the quasigroup operation
as  in the study of model complete Steiner triple system of
Barbina and Casanovas \cite{BarbinaCasa}. As  noted in Remark 5.27 of
\cite{BaldwinPao}, their generic structure $M$ differs radically from ours:
$\acl_M(X) = \dcl_M(X)$.}

Do the strongly minimal quasigroups in last paragraph satisfy elimination of imaginaries?
Is it possible to develop a theory of $q$-block algebras for arbitrary
prime powers similar to that for Steiner quasigroups in their paper? That
is, to find a model completion for each of the various varieties of quasigroups
 discussed in \cite{BaldwinsmssII}. Where do the resulting
theories lie in the stability classification? \cite{HorsleyWebb} prove the existence of
generics for classes of Steiner systems omitting certain finite Steiner Systems. Are
there strongly minimal theories for such classes?
\end{question}

%
%

{\bf Errata}
We thank Gregory Cherlin for pointing to some inaccuracies in the published
text and in particular for noting that much of condition 2.9 b) is
unnecessary with the correct version in 2) below.

1)  Notation 2.7  should  be defined only for independent sets. This is
essential to show $\sdcl^* \subseteq \dcl^*(X)$.

2) Definition 2.9 of truly n-ary is defective; it should read:

 Let $\xbar =\langle x_0 \ldots x_{n-1}\rangle$:  a function $f(\xbar)$
truly {\em truly depends on $x_i$} if for any {\em independent}
    sequence $(a_1, \ldots, a_i, \ldots, a_n)$ and some (hence any) $b$ such
that $( a_1, \ldots, a_i,\ldots, a_n , b)$ is independent,\\ $f( a_1, \ldots,
a_i, ..., a_n ) \ne f(a_1, \ldots, a_{i-1}, b, a_{i+1}, \ldots, a_n)$.

$f$ is truly n-ary if it depends on all its arguments.

 A short argument (pointed out by Cherlin) shows we can omit (b) in Definition 2.9.2.

3. regarding footnote 17: We introduce the notion "truly dependent" because
we want to distinguish standard dependence and dependence "almost everywhere"

For Gr\"{a}tzer dependence on $x_i$ means there is a string $a_1, a_2, a_3
\ldots$ such that changing only $x_i$ argument changes the value. So $xy +z$
depends on both $x$ and $y$.  Our more restrictive condition of {\em truly
depends} requiring the string to be independent, makes $xy +z$ independent.

(Our definition is in the context of a strongly minimal theory and
independence is thus algebraic independence.  So $xy + z$ becomes truly
independent.)

More succinctly, Let $A$ be a finite set and
$f(x,y) = xy$ if $x\in A$ and
$f(x,y) = y$ otherwise.
Then $f$ depends on 2 variables but is not truly binary.

4) Remark 2.11 The Steiner example.  The intent is that this is clearly a
definable $(k-1)$-ary function in a strongly minimal k-steiner system,
meaning in a vocabulary with a single ternary relation,  two points determine
a line and all lines have $k$-points.

We want it not be truly $k$-ary under our definition. And this is trivially
the case because the definition makes the $a_1, \ldots a_{k-1}$ dependent.

\bigskip
Typos:

- on page 1, "Theorem 5.18" is actually 5.19.

bottom page 5 Axiom 1.1.1.5 (line -2) is  referring to 1.1.1.6

- Def 1.1.2: subscripts M should be N here.

- Def. 2.9 para 1, end: $b_i$ should be $b_{i+1}$ there.

- Page 11, top, (a): `disagrees with a' (there is an unwanted prime there)

\end{document}